\documentclass[11pt]{article}

\usepackage[english]{babel}
\usepackage{amsmath,amsthm,amssymb}
\usepackage[all]{xy}
\usepackage{setspace}
\usepackage{tikz-cd, mathrsfs, mathtools}
\usepackage{color}
\definecolor{grau}{rgb}{0.3,0.3,0.3}
\usepackage[colorlinks, linkcolor=grau, citecolor=grau, urlcolor=grau]{hyperref}
\usepackage{breakurl}
\usepackage{bibspacing}
\usepackage[top=1.3in, bottom=1.6in, left=1.3in, right=1.3in]{geometry}
\usepackage{booktabs}
\usepackage[ruled,vlined]{algorithm2e}
\usepackage{enumerate}

\newtheorem{theorem}{Theorem}[section]
\newtheorem{conjecture}[theorem]{Conjecture}
\newtheorem{corollary}[theorem]{Corollary}
\newtheorem{lemma}[theorem]{Lemma}

\newtheorem{proposition}[theorem]{Proposition}

\theoremstyle{definition}

\numberwithin{equation}{section}

\title{Rational points and rational moduli spaces} 

\author{ Shijie Fan \ \ and \ \  Rafael von K\"anel 
}
\newcommand{\OL}{\mathcal O}

\newcommand{\isomto}{\overset \sim \to}

\newcommand{\ZZ}{\mathbb Z}
\newcommand{\Z}{\mathbb Z}
\newcommand{\Aut}{\textup{Aut}}

\newcommand{\Q}{\mathbb Q}

\newcommand{\Hom}{\textup{Hom}}

\newcommand{\spec}{\textnormal{Spec}}

\newcommand{\QQ}{\mathbb Q}

\newcommand{\End}{\textup{End}}

\newcommand{\absg}{\underline{A}_g}

\newcommand{\sch}{\textnormal{(Sch)}}
\newcommand{\sets}{\textnormal{(Sets)}}

\newcommand{\CC}{\mathbb C}

\newcommand{\RR}{\mathbb R}
\newcommand{\cmp}{\mathcal M_{\mathcal P}}

\newcommand{\rad}{\textnormal{rad}}

\def\sl2{\textnormal{SL}_2}
\newcommand{\NN}{\mathbb N}
\newcommand{\mb}{\bar{M}}
\newcommand{\mbd}{\bar{M}^{\textnormal{deg}}}
\newcommand{\sabsg}{\underline{\mathcal A}_g}
\newcommand{\xt}{\tilde{X}}
\newcommand{\qb}{\bar{\mathbb Q}}
\newcommand{\bs}{\mathbb S}

\def\proj{\textnormal{Proj}}
\def\gl2{\textnormal{GL}_2}

\def\mgl2{M_{\textnormal{GL}_2,g}}

\newcommand {\OK}  {{\mathcal O_{K}}}

\newcommand{\DarkGreenempty}[1]{}
\newcommand{\M}{\mathcal{M}}
\renewcommand{\P}{\mathcal{P}}

\usepackage[continuous, page]{pagenote}
\makepagenote

\begin{document}
\date{}
\maketitle

{\scriptsize
\begin{abstract}
Let $X$ be a variety over $\QQ$. We introduce a geometric non-degenerate criterion for $X$ using moduli spaces $M$ over $\QQ$ of abelian varieties. If $X$ is non-degenerate, then we construct via $M$ an open dense moduli space $U\subseteq X$ whose forgetful map defines a Par{\v{s}}in construction for $U(\QQ)$. For example if $M$ is a Hilbert modular variety then $U$ is a coarse Hilbert moduli scheme and $X$ is non-degenerate iff a projective model $Y\subset \bar{M}$ of $X$ over $\QQ$ contains no singular points of the minimal compactification $\bar{M}$. We motivate our constructions when $M$ is a rational variety over $\QQ$ with $\dim(M)>\dim(X)$.

We study various geometric aspects of the non-degenerate criterion and we deduce arithmetic applications: If $X$ is non-degenerate, then $U(\QQ)$ is finite by Faltings. Moreover, our constructions are made for the effective strategy which combines the method of Faltings (Arakelov, Par{\v{s}}in, Szpiro) with modularity and Masser--W\"ustholz isogeny estimates. When $M$ is a coarse Hilbert moduli scheme, we use this strategy to explicitly bound the height and the number of $x\in U(\QQ)$ if $X$ is non-degenerate.

We illustrate our approach in the case when $M$ is the Hilbert modular surface given by the classical icosahedron surface studied by Clebsch, Klein and Hirzebruch. For any curve $X$ over $\QQ$, we construct and study explicit projective models $Y\subset\mb$ called ico models. If $X$ is non-degenerate, then we give via $Y$ an effective Par{\v{s}}in construction and an explicit Weil height bound for $x\in U(\QQ)$. As most ico models are non-degenerate and $X\setminus U$ is controlled,  this establishes the effective Mordell conjecture for large classes of (explicit) curves over $\QQ$. We also solve the ico analogue of the generalized Fermat problem by combining our height bounds with Diophantine approximations.



\end{abstract}
}


\newpage

{\scriptsize
\tableofcontents
}


\newpage

\section{Introduction}


Let $X$ be a variety\footnote{A variety over a field $k$ is a separated finite type $k$-scheme. A curve (resp. surface) over a field $k$ is a variety over $k$ whose irreducible components all have dimension one (resp. two).} over $\QQ$. In this paper we introduce a geometric non-degenerate criterion for $X$ using moduli spaces $M$ over $\QQ$ of abelian varieties.  If $X$ is non-degenerate, then we construct via $M$ an open dense moduli space $U\subseteq X$ whose forgetful map \eqref{def:forgetful} defines a Par{\v{s}}in construction for $U(\QQ)$. Our constructions are made for the (effective) strategy which was developed over the last 70 years by many people: It combines the method of Faltings (Arakelov, Par{\v{s}}in, Szpiro) with modularity and Masser--W\"ustholz isogeny estimates; see the surveys \cite{faltings:history,rvk:cetraro}. If $X$ is non-degenerate, then we use this strategy to prove (effective) finiteness for $U(\QQ)$. We motivate our constructions when $\dim(M)>\dim(X)$ and the moduli space $M$ is rational, that is $M$ is a rational variety over $\QQ$. 

In this introduction we write $\mathbb P^m=\mathbb P^m_\QQ$ and $\mathbb A^m=\mathbb A^m_\QQ$ for any $m\in \NN=\ZZ_{\geq 1}$.

\subsection{Fermat problem}

To provide some motivation, we now discuss for $m=4$ the Fermat problem (F) inside any projective rational surface $S\subseteq\mathbb P^m$ over $\QQ$. We denote by $S(\ZZ)$ the
 set of $x\in\ZZ^{5}$ with $(x_i)\in S(\QQ)$ and $\textnormal{gcd}(x_0,\dotsc,x_4)=1$, and we call $x\in S(\ZZ)$ trivial if all $x_i\in\{-1,0,1\}$. 
 
\vspace{0.3cm}

\noindent {\bf Problem (F).}\emph{ 
For arbitrary nonzero $a,b,c,d,e\in\ZZ$, try to construct $n_0\in \NN$ such that all solutions of the generalized Fermat equations $(F_n)$ are trivial when $n\geq n_0$}:
\begin{equation}\label{eq:fermatintro}
(F_n)  \quad ax_0^n+bx_1^n+cx_2^n+dx_3^n+ex_4^n=0, \quad x\in S(\ZZ), \quad n\in\NN.
\end{equation}

If $S=\mathbb P^2\subset \mathbb P^4:x_3=0=x_4$ then (F) is the classical Fermat problem solved by Wiles~\cite{wiles:modular} for $a=b=-c=1$ with the optimal $n_0=3$. Now, replace $S=\mathbb P^2$ by the birationally equivalent surface $S^{\textnormal{ico}}\subset \mathbb P^4:\sigma_2=0=\sigma_4$ for $\sigma_i$ the $i$-th elementary symmetric polynomial. Then our Theorem~E and Diophantine approximations solve (F).

\vspace{0.3cm}

\noindent {\bf Corollary F.} \emph{If $S=S^{\textnormal{ico}}$ then all solutions $x$ of $(F_n)$, $n\in\NN$, satisfy $\log |x_i|\leq \kappa \nu^{24}$ for $\nu=\rad(abcde)$ and there is $n_0\in\NN$ such that all solutions of $(F_n)$ are trivial when $n\geq n_0$.}

\vspace{0.3cm}

 Here $\kappa=10^{10^{12}}$. The analogue of Corollary~F  is still open for general $a,b,c$ in the classical case $S=\mathbb P^2$. However for many $a,b,c$ optimal results are known if $S=\mathbb P^2$, while our $n_0$ is ineffective. As the underlying geometry of (F) is equivalent, we conjecture that (F) behaves similarly (or might be even related) for $\mathbb P^2$, $S^{\textnormal{ico}}$ and other $S$; see Section~\ref{sec:fermat}.


\subsection{Par{\v{s}}in constructions} 
We next briefly explain and motivate the concept of a Par{\v{s}}in construction. For each $g\in\NN$ we denote by $\absg$ the `space' classifying  abelian varieties of dimension $g$, see Section~\ref{sec:moduli}. Let $X$ be a variety over $\QQ$. Then one can try to construct an integer $g\in\NN$, an open $T\subseteq \spec(\OK)$ with $K$ a number field, and a finite map called Par{\v{s}}in construction:
$$
\phi: X(\QQ)\to \absg(T).
$$
Faltings~\cite{faltings:finiteness} proved the (polarized) Shafarevich conjecture which gives that $\absg(T)$, and thus $X(\QQ)$, is finite. In particular $X(\QQ)$ is finite if $X$ is a curve of geometric\footnote{We say that a curve $X$ over $\QQ$ has geometric genus $\geq 2$ if all irreducible components (equipped with the reduced scheme structure) of $X_{\qb}$ have geometric genus $\geq 2$. For each $g\in \ZZ_{\geq 0}$, we say that a curve $X$ over $\QQ$ is of geometric genus $g$ if  $(X_{\bar{\QQ}})_{\textnormal{red}}$ is integral of geometric genus $g$.} genus $\geq 2$ via Par{\v{s}}in's original construction~\cite{parshin:construction} for a suitable model of $X$. However the following fundamental Diophantine problems are still widely open for many varieties $X$ over $\QQ$: 
\begin{itemize}
\item[(i)] Prove or disprove  finiteness for $X(\QQ)$ when $X$ has dimension $\geq 2$.
\item[(ii)] Control the `size' of the points in $X(\QQ)$ when $X(\QQ)$ is finite.
\end{itemize}
Any $\phi$ suffices for (i). But (ii) requires (\cite[$\mathsection$8]{rvk:cetraro}) an effective $\phi$ with $\phi(X(\QQ))$ contained in a `subspace' of $\absg(T)$ for which the effective Shafarevich conjecture is known. There already exist various Par{\v{s}}in constructions in the literature. Par{\v{s}}in~\cite{parshin:construction}, Kodaira~\cite{kodaira:construction} and Lawrence--Venkatesh~\cite{lave:mordell} give maps $\phi$ (for curves) which all have their own advantages. Szpiro's idea~\cite[p.98]{szpiro:faltings} gives a map $\phi$ for many projective $X$ with $\dim(X)\geq 2$. Moreover the moduli formalism, developed by Katz--Mazur~\cite{kama:moduli} for elliptic curves, turned out to be very useful for (effectively) studying Par{\v{s}}in constructions; see $\mathsection$\ref{sec:parsinmotivation}. 

\paragraph{New constructions.}In this paper we give new Par{\v{s}}in constructions $\phi$ which we shall describe in more detail in $\mathsection$\ref{sec:introgeneral}. As the constructions in \cite{rvk:intpointsmodell,vkkr:hms,vkkr:chms}, they are defined by forgetful maps \eqref{def:forgetful} of coarse moduli schemes (of abelian varieties) introduced in Section~\ref{sec:moduli}. Our new idea is to use certain coarse moduli schemes $M$ which are rational varieties over $\QQ$: The key observation is that via such $M$ one can study $X(\QQ)$ for many $X$ with $\dim(X)<\dim(M)$ by first constructing $\phi$ only for an open dense $U\subseteq X$ and by separately analysing the lower dimensional variety $X\setminus U$. 
Our new constructions allow us to solve problem (ii) for large classes of (explicit) curves over $\QQ$. Moreover, for any $n\in \NN$, we solve problems (i) and (ii) for (infinitely) many varieties over $\QQ$ of dimension $n$.

\subsection{Ico models of curves and rational points}\label{sec:introico}

An important feature of our general construction ($\mathsection$\ref{sec:introgeneral}) is its utility for explicitly studying Diophantine equations. To demonstrate this, we conducted some effort to work out a special case in which we can describe everything without using any `moduli terminology'.

\paragraph{Ico models.}Let $X$ be a curve over $\QQ$. To introduce certain explicit models of $X$ inside $\mathbb P^4_R$ for $R\subseteq \QQ$ a subring, we recall that $\sigma_i$ is the $i$-th elementary symmetric polynomial. Let $f_j\in R[x_0,\dotsc,x_4]$ be homogeneous of degree $n_j\geq 1$ such that $(X_f)_\QQ$, where
\begin{equation}\label{def:icomod}
X_f \subset \mathbb P^4_R:\sigma_2=0=\sigma_4, \, f=0,  \ \quad f=(f_1,\dotsc,f_m),
\end{equation}
is a curve  over $\QQ$.  We say that $X_f$ is an ico model of $X$ over $R$ if there is an open dense $U\subseteq X$ which is isomorphic to an open dense of $(X_f)_\QQ$; write $U_f\subseteq X$ for such an open dense which is maximal as in \eqref{def:controlledopen}. Our approach is based on the following observation which we deduce from classical constructions of Clebsch--Klein \cite{clebsch:clebschsurface,klein:cubicsurfaces}.

\vspace{0.3cm}

\noindent{\bf Theorem A.} Any integral curve over $\QQ$ admits an ico model over $R$.

\vspace{0.3cm}

 Let $\sum a_{ij}x_i^{n_j}$ be the diagonal part of $f_j$.  We say that  $X_f$ is degenerate if there exists $i$ with $a_{ij}=0$ for all $j$, and 
 we say that $X$ is degenerate if all its ico models are degenerate. Theorem~\ref{thm:main} gives that any  $X$ of geometric genus $g$ is degenerate if $g<2$, while in ico form most $X$ are non-degenerate if $g\geq 2$. A general ico model of degree $n$ is a non-degenerate smooth curve of genus $(2n+1)^2$, see Proposition~\ref{prop:generalico} for (explicit) moduli spaces. It is an open problem to classify all $X$ of geometric genus $g\geq 2$ which are non-degenerate.

\paragraph{Effective Par{\v{s}}in.} Let $X$ be a non-degenerate integral curve over $\QQ$, with a Weil height $h$ as in \eqref{def:weilheight}. Let $U\subseteq X$ be open with $U=U_f$ for some non-degenerate ico model $X_f$ of $X$ over $\ZZ$. On combining Theorem~A with \cite[Thm E]{vkkr:chms}, we prove the following:
 
\vspace{0.3cm}

\noindent{\bf Theorem B.} \emph{There is an effective Par{\v{s}}in construction $\phi:U(\QQ)\to \underline{A}_2(T)$ of $\gl2$-type.} 

\vspace{0.3cm}
Roughly speaking effective means that $T$ is controlled and $\phi$ is compatible with heights, while of $\gl2$-type means that each $A\in \phi(U(\QQ))$ is of $\gl2$-type; see $\mathsection$\ref{sec:ratpoints}. In Corollary~\ref{cor:effparico} we show that $\phi$ has additional geometric properties which are crucial for applications. 

\paragraph{Effective Mordell.}Let $d_X$ and $h(X)$ be the normalized degree and height of $X$, defined in $\mathsection$\ref{sec:ratpoints}. Theorem~B and the $\gl2$-case of the effective Shafarevich conjecture established in \cite{rvk:gl2,vkkr:chms} give the following result in which $\nu_f=\textnormal{rad}\bigl(\prod_{a_{ij}\neq 0}a_{ij}\bigl)$ is as in \eqref{eq:defnuf}.
\vspace{0.3cm} 

\noindent{\bf Theorem C.}
\emph{Any $x\in U(\QQ)$ satisfies
$h(x)\leq c\cdot d_X\nu_f^{24}+h(X)$.}

\vspace{0.3cm}
Here $c=10^{10^{12}}$. Usually $h$ is controlled on $(X\setminus U)(\QQ)$ and thus on the whole $X(\QQ)$ by Theorem~C. 
To illustrate this, we let $F\in \ZZ[x,y]$ be of degree $d\geq 1$ and we introduce for the plane curve $X=V(F)\subset \mathbb A^2\subset\mathbb P^2$ the following simple non-degeneracy criterion:
$$(\tau) \ \, \textnormal{The closed image }\overline{\tau(X)} \textnormal{ contains no } e_i.$$
Here $\tau:\mathbb P^2\dashrightarrow \mathbb P^4$ is the explicit rational map over $\QQ$ in \eqref{def:taui} which is birational onto its image, and $e_i\in\mathbb P^4$ are the five permutations of $e_1=(1,0,\dotsc,0)$. For any given $F$ one can compute (over $\QQ$ or $\CC$) whether $(\tau)$ holds. Set $|F|=\max_{\iota} |c_{\iota}|$ for $c_{\iota}$ the coefficients of $F$, and write $h(x)=\log \max (|a|,|b|)$ for $x=\tfrac{a}{b}$ with $a,b\in \ZZ$ coprime. Theorem~C leads to:
\vspace{0.3cm} 

\noindent{\bf Corollary D.}
\emph{
If $X=V(F)$ satisfies $(\tau)$, then the curve $X$ over $\QQ$ is non-degenerate and any solution $(x,y)\in\QQ\times \QQ$ of $F(x,y)=0$ has height $\max(h(x),h(y))\leq \mu|F|^{\kappa}$.}

\vspace{0.3cm} 

Here $\kappa=8^8d^2$ and $\mu=8^{\kappa^{2}d}$. There exist ($\mathsection$\ref{sec:families}) large classes of explicit plane curves satisfying $(\tau)$. For example the space $\mathbb A^r\setminus\bigl(\cup_{i=1}^5V(z_i)\bigl)$ of dimension $r\sim 4n^2$ parametrizes such curves of degree $12n$ if $n\in \NN$, and a moduli space of dimension $\sim g$ parametrizes such curves of geometric genus $g$ if $g\geq 2$ is an odd square. In Section~\ref{sec:effmordell} we discuss:

\vspace{0.3cm} 

\noindent{\bf Conjecture (EM).}
\emph{
If $X=V(F)$ has geometric genus $\geq 2$, then any $(x,y)\in\QQ\times \QQ$ with $F(x,y)=0$ satisfies $\max(h(x),h(y))\leq \mu|F|^{\kappa}$ for effective $\mu,\kappa$ depending only on $d$.}
             
\vspace{0.3cm}

Corollary~D proves this conjecture for all plane curves satisfying $(\tau)$.  
Many non-degenerate plane curves fail the simple criterion $(\tau)$. Theorem~C allows to prove Conjecture~(EM) for other explicit  classes of  non-degenerate plane curves satisfying  variations $(\tau')$ of $(\tau)$, see $\mathsection$\ref{sec:taurho}. In the case of ico models, Theorem~C becomes:


\vspace{0.3cm}

\noindent {\bf Theorem E.}{\emph{
Let $X_f$ be any ico model over $\ZZ$ as in \eqref{def:icomod}. If $X_f$ is non-degenerate, then the curve $X=(X_f)_\QQ$ over $\QQ$ is non-degenerate and all $x\in X(\QQ)$ satisfy $h(x)\leq c\nu_f^{24}$.}

\vspace{0.3cm}

Here $h$ is the usual (\cite[p.16]{bogu:diophantinegeometry}) logarithmic Weil height on $\mathbb P^4$. Our height bounds have useful features as discussed in $\mathsection$\ref{sec:ratpoints}. For example $\nu_f$ depends only on the diagonal part of $f_j$, and  $\nu_f$ is independent of $n_j=\deg(f_j)$ which is crucial for solving (F) in Corollary~F.

\paragraph{Known height bounds.}Let $X$ be a projective curve over $\QQ$ of geometric genus $\geq 2$. Problem (i) is widely open for many $X$. In fact the analogous problem for integral points is also not yet solved, see Corvaja--Lombardo--Zannier~\cite{coloza:effsiegel} for important recent progress for integral points (in any number field). We now discuss known bounds for $X(\QQ)$.

Suppose that $X$ satisfies the Manin--Dem'janenko criterion. Then the method of \cite{demjanenko:rational,manin:torsion} gives height bounds for $X(\QQ)$ in many cases. Moreover, in the elliptic setting,  Checcoli--Veneziano--Viada~\cite{chvevi:effmordell1,chvevi:effmordell2,vevi:effmordell} established in all cases strong explicit height bounds for $X(\QQ)$. They use a different method based on the theory of anomalous intersections of Bombieri--Masser--Zannier~\cite{bomaza:anomalous}. Their results hold over any number field and apply to large families of explicit curves with growing genus.

As $X$ is projective, the results of \cite{rvk:gl2,vkkr:hms,vkkr:chms} give explicit height bounds for $X(\QQ)$ if $X$ satisfies \raisebox{0.15ex}{{\scriptsize($\gl2$)}}, \raisebox{0.15ex}{{\scriptsize$(H)$}} or \raisebox{0.15ex}{{\scriptsize$(cH)$}} for certain geometric criteria  \raisebox{0.15ex}{{\scriptsize($\gl2$)}}, \raisebox{0.15ex}{{\scriptsize$(H)$}} and \raisebox{0.15ex}{{\scriptsize$(cH)$}} in \eqref{def:criteria}. Here the involved heights $h_\phi$ can always be compared to Weil heights, but explicit height comparisons usually require substantial additional effort. Each of the criteria is satisfied by large classes of projective curves $X$. A first explicit example of such $X$ satisfying \raisebox{0.15ex}{{\scriptsize($\gl2$)}} was constructed by Alp\"oge~\cite[$\mathsection$6]{alpoge:modularitymordell} via the hypergeometric construction of Deines--Fuselier--Long--Swisher--Tu~\cite{fuselier:hypergeom}. A completely different construction via Hirzebruch's \cite{hirzebruch:ck} is used in \cite[Thm E]{vkkr:chms}: This explicit result for a surface produces in particular large classes of explicit projective curves $X$ satisfying \raisebox{0.15ex}{{\scriptsize$(cH)$}}. For example all non-degenerate ico models $X_f$ satisfy \raisebox{0.15ex}{{\scriptsize$(cH)$}}, and Theorem~E is a direct consequence of \cite{vkkr:chms}. 

\paragraph{Computing $X(\QQ)$.}Let $X$ be a curve over $\QQ$ of geometric genus $\geq 2$. Based on the works of Skolem~\cite{skolem:method}, Chabauty~\cite{chabauty:method}, Coleman~\cite{coleman:method}, Kim~\cite{kim:siegel} and many others, very powerful practical methods were developed: These methods allow to efficiently compute $X(\QQ)$ in many situations of interest. Also, the explicit height bounds of Checcoli--Veneziano--Viada combined with additional tools (see Stoll~\cite[Appendix A]{chvevi:effmordell2}) allow to compute $X(\QQ)$ for large families of explicit curves $X$ with growing genus.

Suppose that $X$ is as in Corollary~D or Theorem~E. Then our explicit Weil height bounds  allow in principle to compute $X(\QQ)$ by listing all points of bounded height. However this is not practical without using additional tools.  To transform our strategy  into a practical method for computing $X(\QQ)$, one can try  to apply the two approaches of \cite{vkma:computation}: The first approach combines optimized height bounds with efficient sieves, while the second approach first determines the (relevant newforms for the) abelian surfaces in $\phi(X(\QQ))$ and then computes $X(\QQ)=\phi^{-1}\bigl(\phi(X(\QQ))\bigl)$. For example computing $X_f(\QQ)$  for non-degenerate $X_f$ with small $\nu_f$ might not be completely out of reach for these approaches.   


\subsection{General constructions}\label{sec:introgeneral}

To prove our results we use the strategy which combines the method of Faltings (Arakelov, Par{\v{s}}in, Szpiro) with modularity and Masser--W\"ustholz isogeny estimates. We now explain our new constructions which are tailored for  this strategy. Let $X$ be a variety over $\QQ$.

\paragraph{Moduli formalism.}Intuitively one can think of a (coarse) moduli scheme of finite level as a variety whose (geometric) points  parametrize pairs $(A,\alpha)$, where $A$ is an abelian variety and $\alpha$ lies in a certain quasi-finite set $\mathcal P(A)$ associated to $A$. The forgetful map is
$$(A,\alpha)\mapsto A.$$ Building on Katz--Mazur~\cite{kama:moduli}, we give in Section~\ref{sec:moduli} the formal definitions via algebraic stacks of a (coarse) moduli scheme of finite level, its forgetful map \eqref{def:forgetful} and its branch locus.  To simplify the exposition we only write (coarse) moduli scheme from now on. 

\paragraph{Examples.}The class of (coarse) moduli schemes contains large classes of Shimura varieties (\cite{fach:deg,lan:compactifications}), but also many other varieties. For example any smooth, projective and geometrically connected curve $X$ over $\QQ$ of genus $\geq 2$ is (\'etale locally) a moduli scheme whose forgetful map defines  a Par{\v{s}}in construction for $X(\QQ)$, see \cite[$\mathsection$6.2]{rvk:cetraro} which uses Kodaira's construction. The (coarse) Hilbert moduli schemes in \cite{vkkr:hms,vkkr:chms} and the moduli schemes of $\gl2$-type in \cite[$\mathsection$6.1]{rvk:cetraro} are all (coarse) moduli schemes, see Section~\ref{sec:moduli}.

\paragraph{Criteria \raisebox{0.15ex}{{\scriptsize($\gl2$)}}, \raisebox{0.15ex}{{\scriptsize$(H)$}}, \raisebox{0.15ex}{{\scriptsize$(cH)$}}.} The following criteria are geometric in the sense that they only depend on the variety $X$ over $\QQ$ but not on its arithmetic points: If $X$ extends over some $\ZZ[1/\nu]$, $\nu\in\NN$, to a moduli scheme of $\gl2$-type, a Hilbert moduli scheme, or a coarse Hilbert moduli scheme with empty branch locus, then we say that $X$ satisfies 
\begin{equation}\label{def:criteria}
(\gl2), \quad  (H) \ \ \textnormal{ or } \ \ (cH)
\end{equation}
respectively. If $X$ satisfies \raisebox{0.15ex}{{\scriptsize($\gl2$)}} or \raisebox{0.15ex}{{\scriptsize$(H)$}} then so does any $X'$ with a quasi-finite morphism $X'\to X$. This construction shows for example that \raisebox{0.15ex}{{\scriptsize($\gl2$)}} is satisfied by any cubic Thue curve (\cite{rvk:intpointsmodell}) and that \raisebox{0.15ex}{{\scriptsize$(H)$}} is satisfied by any $X$ with a quasi-finite morphism to a representable Hilbert modular variety (\cite{vkkr:hms}). We next discuss a birational variant of \eqref{def:criteria}.

\paragraph{Non-degenerate criteria.}Let $M$ be a coarse moduli scheme over $\QQ$, with a compactification $\mb$. Suppose that $M$ satisfies the (harmless technical) assumption of $\mathsection$\ref{sec:results}. Then we introduce in \eqref{def:mbd} the following non-degenerate criterion: We say that $X$ satisfies \raisebox{0.15ex}{{\scriptsize$(M)$}} if $X$ has a model\footnote{A model of $X$ in $\mb$ is a closed subscheme $Y$ of $\mb$ such that there exist open dense $U\subseteq X$ and $U'\subseteq Y$ with an isomorphism $U\isomto U'$ of $\QQ$-schemes.} in $\mb$ which is disjoint to the `degenerate locus' $\mbd$ defined in \eqref{def:mbd}. For example $\mbd$ is the singular locus of $\mb$ if $M$ is a Hilbert modular variety of dimension $\geq 2$ with minimal compactification $\mb$. Our motivation for studying \raisebox{0.15ex}{{\scriptsize$(M)$}} comes from:

\vspace{0.3cm}

\noindent {\bf Theorem G.} \emph{Let $
X$ be a variety over $\QQ$. If $X$ satisfies} \raisebox{0.15ex}{{\scriptsize$(M)$}}\emph{, then there exists a controlled open dense $U\subseteq X$ which is a coarse moduli scheme over $\QQ$ with empty branch locus.}
\begin{itemize}
\item[(i)] \emph{The forgetful map of the coarse moduli scheme $U$ defines a Par{\v{s}}in construction $\phi:U(\QQ)\to \absg(T)$ and thus $U(\QQ)$ is finite by Faltings.}
\item[(ii)] \emph{In the Hilbert case, each $x\in U(\QQ)$ satisfies $h_\phi(x)\leq c\nu_U^{\kappa}$ and it holds $|U(\QQ)|\leq c\nu_U^{\kappa}$.}
\end{itemize}

\vspace{0.3cm}

Here $c,\kappa$ are the explicit constants in \eqref{def:corconstants} which depend only on $M$, while the height $h_\phi$ is the pullback via $\phi$ of the stable Faltings height~\cite[p.354]{faltings:finiteness}. The height $h_\phi$ can be compared to Weil heights, but explicit height comparisons usually require substantial additional effort. The definitions of a controlled open $U\subseteq X$ and of the integral degeneration $\nu_U$ are given in $\mathsection$\ref{sec:terminology}. Intuitively one can think of the prime divisors of $\nu_U$ as the closed points of $\spec(\ZZ)$ over which a `minimal model of $X$ in $\mb_\ZZ$' degenerates. 

The Hilbert case refers to the case when $M$ is a coarse Hilbert moduli scheme. In particular we can apply Theorem~G~(i) and (ii) with any Hilbert modular variety $M$ over $\QQ$. As explained below, this gives for each $n\in \NN$ large classes of integral projective varieties $X$ over $\QQ$ of dimension $n$ such that $X$ satisfies criterion \raisebox{0.15ex}{{\scriptsize$(M)$}} and Theorem~G~(i) and (ii) hold for $X$ with $U=X$. It is an open problem to explicitly construct such $X$ if $n\geq 2$. Theorem~G~(ii) is a direct consequence of the construction in (i) and results in \cite{vkkr:chms}.

\paragraph{Main idea.} We next explain the idea underlying our construction in Theorem~\ref{thm:main}. The key observation is as follows: If $M$ is a  rational variety over $\QQ$, then any integral variety $X$ over $\QQ$ with $\dim(X)<\dim(M)$ has models in $\mb$ and the variety $X$ satisfies \raisebox{0.15ex}{{\scriptsize$(M)$}} iff at least one of these models is disjoint to $\mbd$. In particular if
\begin{equation}\label{eq:constructionidea}
(1) \  M \textnormal{ is a rational variety over } \QQ     \quad  \textnormal{ and } \quad    (2) \  m=\dim(\mb)-\dim(\mbd)\geq 2,
\end{equation}
then for each $n\in\NN$ with $n<m$ there exist many integral varieties $X$ over $\QQ$ of dimension $n$ which satisfy criterion \raisebox{0.15ex}{{\scriptsize$(M)$}}. For each of these $X$, we can apply Theorem~G to $U(\QQ)$ and we are reduced to study $(X\setminus U)(\QQ)$ where $X\setminus U$ has dimension at most $\dim(X)-1$. This construction motivates to find $M$ and $\mb$ which have properties (1) and (2).

\paragraph{Hilbert case.}Let $n\in \NN$. Many coarse Hilbert moduli schemes $M$ over $\QQ$ have property (2). For example any Hilbert modular variety $M$ over $\QQ$ of dimension $g\geq 2$ with minimal compactification $\mb$ has property (2) since $m=\dim(\mb)=g$. For each such $M$ with $g=n+1$, Theorem~\ref{thm:main}~(ii) then gives many integral projective varieties $X$ over $\QQ$ of dimension $n$ such that $X$ satisfies criterion \raisebox{0.15ex}{{\scriptsize$(M)$}} and Theorem~G~(i) and (ii) hold for $X$. Moreover Theorem~G (i) and (ii) hold with $U=X$ for infinitely many of these $X$. 

The rationality condition in (1) is very restrictive. In particular most Hilbert modular varieties are not rational by Tsuyumine~\cite{tsuyumine:hilbertkodaira}. However there exist several Hilbert modular surfaces $M$ over $\QQ$ which are known to have property (1). To discuss an example, we denote by $M^{\textnormal{ico}}$ the Hilbert modular surface over $\QQ$ of principal level 2 for the totally real field $\QQ(\sqrt{5})$.  Hirzebruch~\cite{hirzebruch:ck} identified $\mb^{\textnormal{ico}}$ with the classical icosahedron surface $S^{\textnormal{ico}}$  which is a rational surface over $\QQ$. In particular $M^{\textnormal{ico}}$ has property (1).

Given $M$ with (1) and (2), one can try to construct explicit classes of varieties $X$ over $\QQ$ satisfying criterion \raisebox{0.15ex}{{\scriptsize$(M)$}} and apply Theorem~G~(ii). For instance the results for non-degenerate curves in $\mathsection$\ref{sec:introico} were obtained by working out our construction \eqref{eq:constructionidea} for $M^{\textnormal{ico}}$ using the geometric results of \cite[$\mathsection$11]{vkkr:chms}. Here \cite[$\mathsection$11]{vkkr:chms} proves that $S^{\textnormal{ico}}$  extends over the explicit base $\ZZ[\tfrac{1}{30}]$  to the minimal compactification of a coarse Hilbert moduli scheme over $\ZZ[\tfrac{1}{30}]$ with empty branch locus. This is crucial for applying Theorem~G~(ii) with an effective $\nu_U$, which in turn is crucial for deducing the effective results in $\mathsection$\ref{sec:introico}.

\paragraph{General case.} We are currently trying to use our construction \eqref{eq:constructionidea} to find new interesting classes of varieties for which one can deduce via Theorem~G~(i) finiteness results from Faltings~\cite{faltings:finiteness}. For large classes of projective (subvarieties of) Shimura varieties, Deligne--Szpiro~\cite[p.98]{szpiro:faltings} and Ullmo~\cite{ullmo:ratpoints} established the finiteness of rational points via Faltings~\cite{faltings:finiteness}. Moreover Javankpeykar--Loughran~\cite{jalo:stackychevalley} formalized this line of reasoning, see e.g. \cite[$\mathsection$3.1]{rvk:cetraro} for a survey of known finiteness results for moduli schemes. 

Another motivation for proving Theorem~G~(i) in general was as follows: If one can prove (conditional)  height bounds in new situations (e.g. conditional on Venkatesh's height conjecture~\cite{venkatesh:icm,venkatesh:heightconjecture}), then one can directly deduce the (conditional) analogue of Theorem~G~(ii) in these new situations and can combine it with our construction \eqref{eq:constructionidea}.

\paragraph{Number fields.}The constructions underlying Theorem~G~(i) are geometric and Faltings' finiteness results~\cite{faltings:finiteness} hold over any number field $K$. Thus one can work out the analogue of Theorem~G~(i) for any $K$ without introducing new ideas. However substantial new ideas and new modularity results are required to prove the analogue of Theorem~G~(ii) for any number field $K$, see also the corresponding discussions in \cite[$\mathsection$1.3.4]{vkkr:chms}.

\subsection{Outline}
The plan of the paper is as follows. After introducing in Section~\ref{sec:moduli} some terminology from the moduli formalism, we state in Section~\ref{sec:nondegvarieties} our main constructions (Theorem~\ref{thm:main}) for non-degenerate varieties over $\QQ$. We also briefly discuss a corollary which gives explicit bounds in the Hilbert case. In Section~\ref{sec:varietiesparsin} we construct our Par{\v{s}}in constructions for non-degenerate varieties over $\QQ$ and then we complete the proof of Theorem~\ref{thm:main} in Section~\ref{sec:construction}.

In the second part we explicitly work out a special case of our constructions in Theorem~\ref{thm:main} and we deduce applications. We study various aspects of non-degenerate curves in Section~\ref{sec:nondegcurves}. In particular we prove explicit height bounds. Then, in Section~\ref{sec:p2}, we illustrate our results and methods  for a certain class of non-degenerate curves: The class of plane curves which satisfy criterion $(\tau)$. In Section~\ref{sec:fermat} we study the Fermat problem and finally we discuss some aspects of the effective Mordell problem in Section~\ref{sec:effmordell}.

\subsection{Acknowledgements}

The second author is incredibly grateful to Arno Kret and Benjamin Matschke for countless discussions over the last decade about topics related to this work. He also would like to thank Akshay Venkatesh for motivating discussions about his height conjecture. The first author would like to thank Professor Ye Tian for inviting him to his seminar on December 19, 2024 to report on the results of this work. Both authors are very grateful to Shouwu Zhang and Ye Tian; without their support this joint work would not have been possible. Finally we would like to thank Vesselin Dimitrov and Shouwu Zhang for useful comments.

\section{Notation and Conventions}

Unless mentioned otherwise, we shall use the following conventions and terminology. By $\log$ we mean the principal value of the natural logarithm. The product taken over the empty set is defined as $1$.  For any set $M$ we denote by $\lvert M\rvert$ the number of elements of $M$ and we write $N^c$ for the complement in $M$ of a subset $N\subseteq M$.  A map of sets $f:M'\to M$ is called finite if for each $m\in M$ the fiber $f^{-1}(m)$ of $f$ over $m$ is finite. 

\paragraph{Algebra.} We write $\NN=\ZZ_{\geq 1}$. Let $k$ be a field. We denote by $\bar{k}$ an algebraic closure of $k$ and we write $\OK$ for the ring of integers of a number field $K$.  For any subset $\{x_0,x_1,x_2,\dotsc\}$ of $\ZZ$, we write $\gcd(x_i)$ for the greatest common divisor $\gcd(x_0,x_1,x_2,\dotsc)>0$.  The radical $\textnormal{rad}(m)$ of a nonzero $m\in\ZZ$ is defined as the product $\prod_{p\mid m} p$ of all the rational primes $p$ which divide $m$. We denote by $R^\times$ the group of units of a commutative ring $R$. 

\paragraph{Schemes.} Let $S$ be a scheme. We denote by $(\textnormal{Sch}/S)$ the category of schemes over $S$.     
If $T$ and $X$ are $S$-schemes, then we define $X(T)=\Hom_{(\textnormal{Sch}/S)}(T,X)$  and we write $X_T=X\times_S T$ for the base change of $X$ from $S$ to $T$. A variety $X$ over $S$ is an $S$-scheme $X$ whose structure morphism $X\to S$ is separated and of finite type.  We denote by $k(S)$ the function field of $S$ if $S$ is integral. A Dedekind scheme is a normal noetherian scheme of dimension 0 or 1. We often identify an affine scheme $S=\spec(R)$ with the ring $R$. In particular we write $X_R$ for $X_T$ and $X(R)$ for $X(T)$ if $T=\spec(R)$ is affine. Let $n\in\NN$ and let $I\subseteq A$ be a homogeneous ideal of $A=R[x_0,\dotsc,x_n]$. We equip $V_+(I)\subseteq\mathbb P^n_R$ with the closed subscheme structure given by $\proj(A/I)$ and we often identify $\proj(A/I)$ with $V_+(I)$. A variety over $R$ inside $V=V_+(I)$ is a closed subscheme of $V$ with $R$-scheme structure induced by $\mathbb P^n_R$. Let $m\in\NN$. For each $f\in A^m$ with $f_j$ homogeneous, the variety $X$ over $R$ defined inside $V$ by $X\subseteq V: f=0$ is the closed subscheme $X\subseteq V$ given by the ideal $(I,f)$ of $A$. For example the ico model $X_f$ over $R$ defined in \eqref{def:icomod} is the variety over $R$ defined inside $\mathbb P^4_R$ by $f^*=0$ where $f^*=(\sigma_2,\sigma_4,f)\in A^{m^*}$ with $m^*=m+2$.



\paragraph{Varieties.}  A rational variety over $k$ is an integral variety over $k$ which is birationally equivalent over $k$ to $\mathbb P^n_k$ for some $n\in\NN$.  A compactification $\mb$ of a variety $M$ over $k$ is a proper variety $\mb$ over $k$ such that there exists an open immersion $M\hookrightarrow \mb$ of $k$-schemes whose image is dense in $\mb$. A curve (resp. surface) over $k$ is a variety over $k$ whose irreducible components all have dimension one (resp. two). Let $X$ be a curve over $k$. We say that $X$ has geometric genus $\geq 2$ if all irreducible components (equipped with the reduced scheme structure) of $X_{\bar{k}}$ have geometric genus $\geq 2$. For any $g\in \ZZ_{\geq 0}$ we say that $X$ is of geometric genus $g$ if $(X_{\bar{k}})_{\textnormal{red}}$ is integral of geometric genus $g$.


\paragraph{Stacks.}Let $S$ be a scheme and let $\mathcal C$ be a category. A contravariant functor from $\mathcal C$ to the category of sets is called a presheaf on $\mathcal C$. Unless mentioned otherwise we will use the terminology and definitions of the Stacks project~\cite{sp}.  An algebraic stack over $S$ is an algebraic stack over $(\textnormal{Sch}/S)_{\textnormal{fppf}}$ in the sense of \cite{sp}. Throughout we shall freely use well-known standard results for algebraic stacks or DM stacks which all can be conveniently found in \cite{sp}. Usually we identify schemes (or algebraic spaces) over $S$ with the associated algebraic stacks over $S$. Let $\mathcal M$ be a category fibred in groupoids over $(\textnormal{Sch}/S)$. We denote by $[\mathcal M]$ the presheaf on $(\textnormal{Sch}/S)$ which sends an $S$-scheme $T$ to the set $[\mathcal M(T)]$ of isomorphism classes of the objects in the fiber $\mathcal M(T)$ of $\mathcal M$ over $T$.

\section{Moduli formalism}\label{sec:moduli}

We introduce in this section terminology from the moduli problem formalism which is useful for (effectively) studying Par{\v{s}}in constructions. The formalism was developed by Katz--Mazur~\cite{kama:moduli} for elliptic curves via Mumford's language of relative representability~\cite{mumford:picardmoduli}. Those parts which are relevant for us were extended via the language of algebraic stacks to the Hilbert case  in \cite{vkkr:hms,vkkr:chms} and to all abelian varieties in~\cite{rvk:cetraro}.  


\paragraph{Moduli schemes.}Let $g\in\NN$, let $\sabsg$ be the category fibred in groupoids over $\sch$ of abelian schemes of relative dimension $g$, and  let $\mathcal P$ be a presheaf on $\sabsg$. We call $\mathcal P$ a moduli problem. For each object $A$ of $\sabsg$ the elements of $\mathcal P(A)$ are called $\mathcal P$-level structures of $A$. Let $Y$ be a scheme. We say that $Y$ is a moduli scheme of $\mathcal P$, and we write $Y=M_{\mathcal P}$, if there exists an object in $\sabsg(Y)$ representing $\mathcal P$.  For example, consider the moduli problem  $$\mathcal P_{g,n}: \sabsg\to \sets, \quad A\mapsto \textnormal{Pol}(A)\times \mathcal P_n(A)$$ where $\textnormal{Pol}(A)$ denotes the set of principal polarizations of $A$ and  $\mathcal P_n(A)$ denotes for any $n\in\NN$ the set of principal level $n$-structures on $A$ as in \cite[p.5]{fach:deg}. There exists a moduli scheme $A_{g,n}$ of $\mathcal P_{g,n}$ if $n\geq 3$, which is automatically a $\ZZ[1/n]$-scheme since $\mathcal P_{g,n}$ is defined over $\ZZ[1/n]$. However there exists no moduli scheme of $\mathcal P_{g,n}$ if $n\leq 2$. 


\paragraph{Coarse moduli schemes.}Write $\mathcal M$ for $\sabsg$. Let $\mathcal P$ be a moduli problem on $\mathcal M$ and let $\cmp$ be the category fibred in groupoids over $\sch$ of pairs $(A,\alpha)$ with $A\in\mathcal M$ and $\alpha\in\mathcal P(A)$. We say that  $\mathcal P$  is algebraic (resp. arithmetic) if $\cmp$ is an algebraic stack (resp. a DM-stack which is separated and of finite type over $\ZZ$). Let $Y$ be a scheme and suppose that $\mathcal P$ is algebraic. We say that $Y$ is a coarse moduli scheme of  $\mathcal P$, and we write $Y=M_{\mathcal P}$, if there exists a coarse moduli space $\pi:\cmp\to Y$ in the usual sense (\cite[$\mathsection$11]{olsson:stacks}); we call $\pi$ an initial morphism. The branch locus $B_\mathcal P\subseteq Y$ of $\mathcal P$ is defined as the complement in $Y$ of the union of all open $U\subseteq Y$ such that $\pi_U$ is \'etale. For example, if $Y$ is a moduli scheme of $\mathcal P$ then it is a coarse moduli scheme of $\mathcal P$   whose branch locus $B_{\mathcal P}$ is empty. 

\paragraph{Isomorphism classes.}The presheaf $\absg=[\sabsg]$ on $\sch$ classifies isomorphism classes of abelian schemes of relative dimension $g$. If $g=1$ then $\sabsg\cong \mathcal A_1$ is a separated DM-stack over $\ZZ$. However the category $\sabsg$ over $\sch$ and the presheaf $\absg=[\sabsg]$ on $\sch$ both do not have a `reasonable' geometric structure when $g\geq 2$. The reason is that they both classify (isomorphism classes of) unpolarized abelian schemes of relative dimension $g$.

\paragraph{Forgetful map.}Let $Y$ be a coarse moduli scheme of an arithmetic moduli problem $\mathcal P$ on $\mathcal M$ with initial morphism $\pi:\cmp\to Y$, and let $k$ be an algebraically closed field. Composing $\pi^{-1}:Y(k)\isomto [\cmp(k)]$ with $[(A,\alpha)]\mapsto [A]$ defines the forgetful map
\begin{equation}\label{def:forgetful}
\phi_k: Y(k)\to \absg(k).
\end{equation}
Let $\phi:Y(R)\to \absg(T)$ be a map for subrings $R$ and $T$ of $k$. We say that $\phi_{k}$ defines $\phi$ if the restriction of $\phi_k$ to $Y(R)$ factors as $Y(R)\to^\phi\absg(T)\to\absg(k)$. 
Let $k_0$ be a field with algebraic closure $k$. We say that a subset $W$ of $\absg(T)$ is $G_{k_0}$-stable if for each $[A]\in W$ and any $\sigma\in \Aut(k/k_0)$ there is an isomorphism $A_k\isomto \sigma^*A_k$ of abelian varieties over $k$. 

\paragraph{Finite level.}We say that a variety $Y$ over $\Z$ is a coarse moduli scheme of finite level if it is a coarse moduli scheme of an arithmetic moduli problem $\mathcal P$ on $\mathcal M$ satisfying:
\begin{itemize}
\item[(i)] The set $\mathcal P(A)/\Aut(A)$ is finite for each abelian scheme  $A\in \mathcal M(T)$ and for any connected Dedekind scheme $T$ whose function field is algebraic over $\QQ$.
\item[(ii)] There exist $n\in\NN$, a scheme $Y'$ and a finite \'etale cover $Y'\to (\cmp)_{\ZZ[1/n]}$.
\end{itemize}
Here (i) assures (\cite[$\mathsection$6.1]{rvk:cetraro}) that the forgetful map $\phi_{\qb}$ is finite. Both (i) and (ii) are usually satisfied in situations of interest in arithmetic. In particular for all $n\in\NN$ the coarse moduli scheme $A_{g,n}$ of $\mathcal P_{g,n}$ is of finite level  by (the proof of) Narasimhan--Nori \cite{nano:polarizations}. 


\paragraph{Hilbert case.}On replacing in the above definitions the category $\mathcal M=\sabsg$ over $\sch$ by a Hilbert moduli stack over $\ZZ$ of Deligne--Pappas~\cite[$\mathsection$2]{depa:hilbertmodular}, one obtains the analogous Hilbert notions of \cite{vkkr:hms,vkkr:chms}. These Hilbert notions are compatible (\cite[$\mathsection$11.1]{rvk:cetraro}) with the above notions. In particular any (coarse) Hilbert moduli scheme of \cite{vkkr:hms,vkkr:chms} is a (coarse) moduli scheme. In fact any coarse Hilbert moduli scheme of an arithmetic moduli problem automatically satisfies condition (ii); this follows from \cite[Lem 5.1]{vkkr:chms} combined with Rapoport's~\cite[Lem~1.23]{rapoport:hilbertmodular}.
 
\paragraph{Over $\ZZ[1/n]$ or $\QQ$.}Let $S$ be either $\QQ$ or a nonempty open subscheme of $\spec(\ZZ)$. On replacing in the above definitions the base $\sch$ by $(\textnormal{Sch}/S)$, we obtain as in \cite[$\mathsection$3.5]{vkkr:chms} or \cite[$\mathsection$6.1]{rvk:cetraro} the analogues over $S$ of the above notions.  As $S\to \spec(\ZZ)$ is a flat monomorphism, the above notions are compatible with  base change to $S$; see Lemma~\ref{lem:coarsebasechange}. 
The compatibility with a non-flat base change depends on the automorphism groups. Let $Y$ be a coarse moduli scheme over $S$ of an arithmetic moduli problem $\mathcal P$. We say that $Y=M_{\mathcal P}$ is tame if $\cmp$ is a tame stack over $S$ in the usual sense (\cite[Def 11.3.2]{olsson:stacks}).} 

\subsection{Par{\v{s}}in constructions via forgetful maps}\label{sec:parsinmotivation}
The moduli formalism allows to (effectively) study Par{\v{s}}in constructions via the forgetful map \eqref{def:forgetful} of coarse moduli schemes of finite level. This approach to Par{\v{s}}in constructions has useful features as explained in \cite[$\mathsection$6.2]{rvk:cetraro}. It allows to conceptually simplify/unify many of the known constructions and it can provide crucial advantages for the effective study. For example it can be beneficial for solving equations (\cite[$\mathsection$4.2.6]{vkma:computation}). Moreover it can be very useful for finding/studying new (effective) Par{\v{s}}in constructions with additional geometric properties. Here an important property of the forgetful map \eqref{def:forgetful} is its compatibility with the $G_\QQ$-action: This assures that the image of the resulting Par{\v{s}}in construction is $G_\QQ$-stable, which is crucial for \cite{vkkr:chms} and for the effective results in this paper. This $G_\QQ$-stability is one of the key advantages of Par{\v{s}}in constructions via forgetful maps compared to other Par{\v{s}}in constructions for curves (e.g. via Kodaira's construction involving \'etale covers) which are not known to be compatible with $G_\QQ$. On the other hand, many of the other Par{\v{s}}in constructions work in general (i.e. for all relevant curves) and an effective Kodaira construction~\cite{kodaira:construction} is  known in general by R\'emond~\cite{remond:construction}.

\section{Non-degenerate varieties and Par{\v{s}}in constructions}\label{sec:nondegvarieties}

In this section we first state in Theorem~\ref{thm:main} our main constructions for non-degenerate varieties. Then we discuss a corollary which gives explicit bounds in the Hilbert case. 

\subsection{Results}\label{sec:results}

Let $M$ be a variety over $\QQ$ which extends over a nonempty open of $\spec(\ZZ)$ to a tame coarse moduli scheme of finite level, and let $\mb$ be a compactification of $M$. In what follows we refer by the Hilbert case to the case when here the coarse moduli scheme is a tame coarse Hilbert moduli scheme of finite level. Let $X$  be a variety over $\QQ$.

\paragraph{Non-degenerate varieties in $\mb$.}  A model of $X$ in $\bar{M}$ is a closed subscheme $Y$ of $\bar{M}$ such that there exist dense open $U\subseteq X$ and $U'\subseteq Y$ with an isomorphism $U\isomto U'$ of $\QQ$-schemes. We say that $X$ is non-degenerate in $\bar{M}$ if it has a model $Y$ in $\bar{M}$  with
\begin{equation}\label{def:mbd}
Y\cap \bar{M}^{\textnormal{deg}}=\emptyset, \quad \quad \mbd=B\cup (\bar{M}\setminus M)
\end{equation}
for $B\subseteq M$                                                                      the branch locus of $M$. Here the terminology `non-degenerate' is compatible ($\mathsection$\ref{sec:terminology}) with the theory of degenerations of abelian varieties \cite{fach:deg,lan:compactifications}. Our motivation for studying varieties which are non-degenerate in $\mb$ comes from Theorem~\ref{thm:main}.

\begin{theorem}\label{thm:main}
The following statements hold.
\begin{itemize}
\item[(i)]Let $X$ be a variety over $\QQ$. Suppose that $X$ is non-degenerate in $\bar{M}$. Then there exists a dense open $U\subseteq X$ which is a coarse moduli scheme over $\QQ$ of finite level with empty branch locus. Its forgetful map $U(\bar{\QQ})\to \absg(\bar{\QQ})$ defines a Par{\v{s}}in construction $\phi:U(\QQ)\to \absg(T)$ and thus  $U(\QQ)$ is finite by Faltings.
\item[(ii)]Firstly, if $M$ is a rational variety over $\QQ$ then each integral variety $X$ over $\QQ$ with $\dim(X)<\dim(M)$ has a model in $\bar{M}$. Secondly, if $\mb$ is an integral projective variety over $\QQ$ then for each $n\in\NN$ with $n<\dim(\mb)-\dim(\mbd)$ there exist infinitely many integral closed subschemes $X\subset\mb$ of dimension $n$ with $X\cap \mbd=\emptyset$.
\item[(iii)] Any curve over $\QQ$ which is non-degenerate in $\mb$ has geometric genus $\geq 2$.
\end{itemize}
\end{theorem}
Building on the constructions in \cite{vkkr:chms}, we will obtain in Theorem~\ref{thm:parshin} the finite map $\phi$ of (i). Moreover Theorem~\ref{thm:parshin} gives additional geometric properties of $U$ and $\phi$ which are important for the effective study of $U(\QQ)$ via $\phi$. We shall deduce the geometric statement (iii) from the proof of (i) and Faltings' finiteness result~\cite{faltings:finiteness} which is of arithmetic nature. It would be interesting to find a purely geometric proof of (iii). Both observations in (ii) are rather direct consequences of standard constructions in algebraic geometry. Variations of the second observation in (ii) are commonly used to construct compact subvarieties of moduli spaces; see Zaal~\cite{zaal:compactsubvarieties} and the references therein. We included (ii) in the theorem to clarify the idea of our construction as described in \eqref{eq:constructionidea}. 

\paragraph{Coarse Hilbert moduli schemes.}In the Hilbert case, we can describe more precisely in Theorem~\ref{thm:parshin} the image of the Par{\v{s}}in construction $\phi:U(\QQ)\to \absg(T)$ of Theorem~\ref{thm:main}~(i).  This description allows to apply for each $A\in \phi(U(\QQ))$ the $\gl2$-case of the  effective Shafarevich conjecture  established in \cite{rvk:gl2,vkkr:chms}, which leads to the following result. 

\begin{corollary}\label{cor:hmssimple}
Let $X$ be a variety over $\QQ$. Suppose that $X$ is non-degenerate in $\bar{M}$. Then in the Hilbert case there exists a controlled dense open $U\subseteq X$ as in Theorem~\ref{thm:main}~(i) such that any $x\in U(\QQ)$ satisfies $h_\phi(x)\leq c\nu_U^{\kappa}$ and such that $|U(\QQ)|\leq c\nu_U^\kappa.$
\end{corollary}
Here $h_\phi$ is the pullback of $h_F$ via $\phi$ from Theorem~\ref{thm:main}~(i) for $h_F$ the stable Faltings height introduced by Faltings~\cite[p.354]{faltings:finiteness}, and the quantity $\nu_U$ defined in \eqref{def:intdeg} measures the integral degeneration of $U$ inside $\mb$. Further $c$ and $\kappa$ are constants depending only on $M$ which are given explicitly in \eqref{def:corconstants}, while $U\subseteq X$ is controlled in the sense of \eqref{def:controlledopen}. 

\paragraph{Hilbert modular varieties.} Corollary~\ref{cor:hmssimple} can be applied with any Hilbert modular variety $M$ over $\QQ$, since each such $M$ has a canonical integral model with the required properties (\cite{rapoport:hilbertmodular,lan:compactifications}). Moreover many (connected) Hilbert modular varieties are defined over $\QQ$, including all with no level structure and all with principal level 2-structure. 

Suppose now that $M$ is a Hilbert modular variety over $\QQ$ of dimension $g\geq 2$. Then the non-degenerate criterion is more concrete, since the branch locus $B$ of $M$ is simply the singular locus $M^{\textnormal{sing}}$. Indeed \cite[(5.9)]{vkkr:chms} identifies $B(\CC)$ with the set of elliptic points of $M(\CC)\cong \mathbb H^g/\Gamma$ which is given by $M^{\textnormal{sing}}(\CC)$ since $g\geq 2$, and thus $B=M^{\textnormal{sing}}$. 
The explicit constants $c$ and $\kappa$ in \eqref{def:corconstants} can be directly computed for any given Hilbert modular variety $M$ over $\QQ$. However, to compute the integral degeneration $\nu_U$, one first needs to compute the branch locus $B_{\bs}$ of a canonical integral model $M_{\bs}=M_{\mathcal P}$ of $M$ over a controlled open $\bs\subseteq\spec(\ZZ)$. If $M$ is representable then $B_{\bs}$ is empty; there exist plenty of classical families of such representable $M$ over $\QQ$ for any $g\geq 2$. Furthermore, if $M$ has principal level 2-structure and $\bs=\ZZ[1/2]$, then $B_{\bs}$ is again empty by \cite[Cor 5.6]{vkkr:chms} whose proof relates $B_{\bs}$ to jumps of the automorphism groups of $\mathcal M_{\mathcal P}$. However, in the general case more machinery is required  to compute $B_{\bs}$ over a controlled $\bs$.

\subsection{Terminology}\label{sec:terminology}

In this section we first briefly discuss the origin of the notion `non-degenerate'. Then we introduce some terminology which we shall use in the effective study of Par{\v{s}}in constructions. We continue our notation and we let $M\subseteq \mb$ be as in $\mathsection$\ref{sec:results}. 

\paragraph{Non-degenerate.}A non-degenerate model in $\mb$ is defined as a closed subscheme of $\mb$ which is disjoint to $\mbd$. The notion `non-degenerate' comes from the following property of any non-degenerate model $Y$ in $\mb$: The points of $Y$ parametrize (\'etale locally) semi-abelian varieties which are all abelian and thus non-degenerate. This is compatible with Faltings--Chai \cite{fach:deg,chai:hilbmod} when $M$ is a Siegel or Hilbert modular variety with minimal compactification $\mb$, and with  Lan~\cite{lan:compactifications} when $M$ is a Shimura variety of PEL-type.

\paragraph{Controlled $U\subseteq X$.}We next introduce terminology which allows to `measure' quantities appearing in Par{\v{s}}in constructions defined by forgetful maps of coarse moduli schemes. Let $X$ be a variety over $\QQ$. For any model $Y$ of $X$ in $\mb$, we consider a dense open 
\begin{equation}\label{def:controlledopen}
U_{Y}\subseteq X
\end{equation}
 with a dominant open immersion $\iota:U_Y\hookrightarrow Y$ of $\QQ$-schemes such that: If $U_{Y}$ is contained in an open $U\subseteq X$ with an open immersion $U\hookrightarrow Y$ of $\QQ$-schemes extending $\iota$, then $U=U_{Y}$. Here $(U_Y,\iota)$ always exists by Zorn's lemma, but it is not necessarily unique.

\paragraph{Integral degeneration.}For what follows we need to specify a suitable base scheme: Let $\bs\subseteq\spec(\ZZ)$ be a nonempty open subscheme such that the variety $M$ over $\QQ$ extends over $\bs$ to a tame coarse moduli scheme $M_\bs$ of finite level. We denote by $B_\bs$ the branch locus of $M_\bs$. For any non-degenerate model $Y$ in $\mb$, we define its integral degeneration 
\begin{equation}\label{def:intdeg}
\nu_Y
\end{equation}
as the smallest $\nu\in\NN$ such that $R=\ZZ[1/\nu]$ is an $\bs$-scheme and  $Y$ extends to a proper scheme $Y_{R}$ over $R$ with an immersion $Y_R\hookrightarrow (M_\bs\setminus B_\bs)_R$ of $R$-schemes. Here $\nu_Y$ is well-defined since we can always spread out an immersion $Y\hookrightarrow \mb\setminus \mbd=M\setminus B$ of $\QQ$-schemes and since $B\cong(B_\bs)_\QQ$ by Lemma~\ref{lem:coarsebasechange}.  If an open $U\subseteq X$ is of the form $U=U_Y$ for some non-degenerate model $Y$ of $X$ in $\mb$, then we write $\nu_U$ for $\nu_{Y}$. 
\paragraph{Effective coverings.} Let $g\in\NN$ and let $T_0\subseteq \spec(\OK)$ be nonempty open with $K\subset \qb$ a number field. We introduce certain effective coverings of a subset $W$ of $\absg(T_0)$. Let $(d,n)\in\NN\times\NN$ and let $\mathcal F_{d,n}$ be the set of subfields $F\subset \qb$ of degree $[F:\QQ]\leq d$ with $F/\QQ$ normal and unramified outside $n$. We say that $W$ admits for $(d,n)$ an effective covering 
\begin{equation}\label{def:effcover}
\cup_T J(T)
\end{equation}
of $\gl2$-type with $G_\QQ$-isogenies if $T_0$ is an $R$-scheme for each $R\in \mathcal R$ and if base change defines a surjective map $\sqcup_{R\in \mathcal R} J(R)\to W$. Here $\mathcal R$ is the set of all $R$ such that $R$ is the integral closure of $\ZZ[1/n]$ in some $F\in\mathcal F_{d,n}$, while $J(R)$ is the set of all $[A]$ in $\absg(R)$ such that $A$ is of $\gl2$-type with $G_\QQ$-isogenies as in \cite[$\mathsection$7.1]{vkkr:chms} and such that $A_3(R)\isomto A_3(\qb)$.

\paragraph{Quantities underlying $M$.}In certain situations we need to specify the quantities underlying $M$. Let $\bs$ and $M_\bs$ be as above. Further let $g\in\NN$ and let $\mathcal P$ be an arithmetic moduli problem on $\mathcal M_\bs$, where $\mathcal M$ is either $\sabsg$ or a Hilbert moduli stack $\mathcal M^I$ over $\ZZ$ of Deligne--Pappas~\cite{depa:hilbertmodular} defined with respect to (\cite[$\mathsection$3]{vkkr:hms}) a totally real number field $L$ of degree $g$. Suppose that the moduli problem $\mathcal P$ satisfies  (i) and (ii) of the definition of finite level in Section~\ref{sec:moduli}. Now, if $M_\bs$ is a tame coarse moduli scheme over $\bs$ of $\mathcal P$, then we say that 
$\mathcal P$ underlies $M$ and in the Hilbert case we say moreover that $L$ underlies $M$.

\section{Par{\v{s}}in constructions}\label{sec:varietiesparsin}

Let $M$ be a variety over $\QQ$ which extends over a nonempty open of $\spec(\ZZ)$ to a tame coarse moduli scheme of finite level, and let $\mb$ be a compactification of $M$. The main goal of this section is to prove the following more precise version of Theorem~\ref{thm:main}~(i).

\begin{theorem}\label{thm:parshin}
Let $X$ be a variety over $\QQ$. If $X$ is non-degenerate in $\mb$, then there exists a dense open subscheme $U\subseteq X$ with the following properties. 
\begin{itemize}
\item[(i)] The variety $U$ over $\QQ$ extends over a dense open $S\subseteq \spec(\ZZ)$ to a coarse moduli scheme $U_S$ of finite level with empty branch locus. Its forgetful map $U(\bar{\QQ})\to\absg(\bar{\QQ})$ defines a Par{\v{s}}in construction $\phi:U(\QQ)\to\absg(T)$ with $G_\QQ$-stable image. 
\item[(ii)] We can take any $U$ of the form $U=U_{Y}$ with $Y$ a non-degenerate model of $X$ in $\mb$.  
\item[(iii)] In the Hilbert case, $U_S$ is a coarse Hilbert moduli scheme and $\phi(U_Y(\QQ))$ admits for $(d,n)=(3^{4g},3\nu_{Y})$ an effective covering $\cup_T J(T)$ of $\gl2$-type with $G_\QQ$-isogenies.
\end{itemize}  
\end{theorem}
Here the map $\phi$ comes from the following Par{\v{s}}in construction for $S$-integral points (in the sense of \cite[$\mathsection$3.3]{rvk:cetraro}) which was obtained in \cite{vkkr:hms,vkkr:chms} in the Hilbert case.  
 
\begin{proposition}\label{prop:parshinmodulischeme}
Let $Y$ be a variety over $\ZZ$, and let $Z\subset Y$ be a closed subscheme.
\begin{itemize}
\item[(i)] Suppose that $Y$ becomes over $\OL=\ZZ[1/\nu]$, $\nu\in \NN$, a coarse moduli scheme over $\OL$ of finite level with branch locus contained in $Z_{\OL}$. Then the forgetful map $Y(\bar{\QQ})\to\absg(\bar{\QQ})$ defines a Par{\v{s}}in construction $\phi:(Y\setminus Z)(\OL)\to\absg(T)$ with $G_\QQ$-stable image. 
\item[(ii)] If $Y_{\mathcal O}$ is in addition a coarse Hilbert moduli scheme, then $\phi\bigl((Y\setminus Z)(\OL)\bigl)$ admits for $(d,n)=(3^{4g},3\nu)$ an effective covering $\cup_TJ(T)$ of $\gl2$-type with $G_\QQ$-isogenies.
\end{itemize}
\end{proposition}
The Par{\v{s}}in construction of Theorem~\ref{thm:parshin} combined with the effective Shafarevich conjecture in \cite[Thm B]{vkkr:chms} gives the following more precise version of Corollary~\ref{cor:hmssimple}.
\begin{corollary}\label{cor:hms}
Let $X$ be a variety over $\QQ$ which is non-degenerate in $\bar{M}$, and let $U\subseteq X$ be an open subscheme of the form $U=U_Y$ with $Y$ a non-degenerate model of $X$ in $\mb$. In the Hilbert case, each $x\in U(\QQ)$ satisfies $h_\phi(x)\leq c_1\nu_Y^{e_1}$ and it holds $|U(\QQ)|\leq c_2 \nu_Y^{e_2}$.
\end{corollary}

Here $U_{Y}$ is as in \eqref{def:controlledopen} and $\nu_Y$ is its integral degeneration \eqref{def:intdeg}, while $h_\phi:U(\QQ)\to \RR$ is the pullback of $h_F$ by the forgetful map $U(\QQ)\hookrightarrow U(\bar{\QQ})\to \absg(\bar{\QQ})$ from Theorem~\ref{thm:parshin}; notice that $h_\phi$ coincides with the height in Corollary~\ref{cor:hmssimple}. One can take for example  
\begin{equation}\label{def:corconstants}
e_1=\max(24,5g), \ \,  c_1=7^{7^{7g}} \quad \textnormal{and} \quad e_2=6\cdot 3^{8g},  \ \, c_2=9^{9^{9g}}|\mathcal P|_{\qb}\Delta\log(3\Delta)^{2g-1}.
\end{equation}
Here $|\mathcal P|_{\qb}<\infty$ is the maximal number (\cite[(3.1)]{vkkr:hms}) of level $\mathcal P$-structures over $\qb$ for any moduli problem $\mathcal P$ underlying $M$, while $\Delta=|\textnormal{Disc}(L/\QQ)|$ for any totally real number field $L$ underlying $M$. In the remaining of this section we prove the above results.

\subsection{Non-degenerate varieties and coarse moduli schemes}\label{sec:coarsebasechange}

We continue our notation. Let $M\subseteq\mb$ be as above and write $\mathcal M=\sabsg$. We first show that non-degenerate models in $\mb$ are coarse moduli schemes of finite level.

\begin{lemma}\label{lem:extndmod}Any non-degenerate model in $\mb$ extends over an open dense  $S\subseteq \spec(\ZZ)$ to a proper coarse moduli scheme over $S$ of finite level with empty branch locus.
\end{lemma}

Then we deduce that any non-degenerate variety has an open dense subscheme which is a coarse moduli scheme of finite level by combining Lemma~\ref{lem:extndmod} with the following result, in which  $S$ is either $\QQ$ or a nonempty open of $\spec(\ZZ)$, $T$ is any $S$-scheme and $\mathcal M=\mathcal M_S$.

\begin{lemma}\label{lem:coarsebasechange}
Let $Y$ be a coarse moduli scheme over $S$ of an arithmetic moduli problem $\mathcal P$ on $\mathcal M$. Let $Y'\to Y$ be a monomorphism of $S$-schemes. If $Y=M_{\mathcal P}$ is tame or $Y'\to Y$ is flat, then $Y'$ is a coarse moduli scheme over $S$ of an algebraic moduli problem $\mathcal P'$ on $\mathcal M$.
\begin{itemize}
\item[(i)] There is an equivalence $\mathcal M_{\mathcal P'}\isomto\cmp\times_Y Y'$ of categories over $\mathcal M$.
\item[(ii)] The morphism $\mathcal M_{\mathcal P'}\isomto \cmp\times_Y Y'\to Y'$ is an initial morphism.
\item[(iii)] For each geometric point $\spec(k)\to S$, the forgetful map $Y'(k)\to \absg(k)$ is the composition of $Y'(k)\to Y(k)$  with the forgetful map $Y(k)\to \absg(k)$.
\item[(iv)] For each $x\in \mathcal M(T)$ the set $\mathcal P'(x)/\Aut(x)$ injects into $\mathcal P(x)/\Aut(x)$. If $Y=M_{\mathcal P}$ is of finite level and $Y'\to Y$ is of finite type, then $Y'=M_{\mathcal P'}$ is again of finite level.
\item[(v)] The morphism $Y'\to Y$ sends the branch locus $B_{\mathcal P'}$ of $Y'=M_{\mathcal P'}$ into the branch locus $B_{\mathcal P}$ of $Y=M_{\mathcal P}$.  If $f:Y'\to Y$ is \'etale, then it holds $B_{\mathcal P'}=f^{-1}(B_{\mathcal P})$.
\end{itemize}
\end{lemma}
Katz--Mazur~\cite{kama:moduli} proved (i) and (ii)  when $g=1$. We deduce (i) and (ii) from base change properties of coarse moduli spaces~(\cite{kemo:coarse,abvi:compactstable}), while (iii), (iv) and (v) are formal consequences of (i) and (ii). To construct $\mathcal P'$ satisfying (i) we use the formal: 

\begin{lemma}\label{lem:formal}
Let $S$ be a scheme, let $\mathcal M$ be a category fibred in groupoids over $S$  and let $\mathcal P$ be a presheaf on $\mathcal M$ such that $\mathcal M_{\mathcal P}$ is an algebraic stack over $S$. For any monomorphism $Y'\to Y$ of $S$-schemes and each morphism $\cmp\to Y$ of algebraic stacks over $S$, there exist a presheaf $\mathcal P'$ on $\mathcal M$ and an equivalence $\mathcal M_{\mathcal P'}\isomto\cmp\times_YY'$ of categories over $\mathcal M$.
\end{lemma}
This lemma also leads to the Hilbert analogues of Lemmas~\ref{lem:extndmod} and \ref{lem:coarsebasechange} given in \eqref{eq:hilbanalogue} below. Since we could not find a reference, we included a proof of Lemma~\ref{lem:formal} at the end of $\mathsection$\ref{sec:coarsebasechange}. We now apply Lemma~\ref{lem:formal} to prove Lemma~\ref{lem:coarsebasechange} and then we deduce Lemma~\ref{lem:extndmod}.

\begin{proof}[Proof of Lemma~\ref{lem:coarsebasechange}]
Recall that $S$ is either $\QQ$ or a nonempty open of $\spec(\ZZ)$ and that $\mathcal M=(\sabsg)_S$. Let $Y$ be a coarse moduli scheme over $S$ of an arithmetic moduli problem $\mathcal P$ on $\mathcal M$ with initial morphism $\pi:\cmp\to Y$. Let $Y'\to Y$ be a monomorphism of $S$-schemes.

We first prove (i) and (ii). Lemma~\ref{lem:formal} gives a presheaf $\mathcal P'$ on $\mathcal M$ with an equivalence $\mathcal M_{\mathcal P'}\isomto\cmp\times_Y Y'$ of categories over $\mathcal M$. In particular $\mathcal P'$ is an algebraic moduli problem on $\mathcal M$, since $\cmp$ and thus $\cmp\times_Y Y'$ is an algebraic stack over $S$. We define the morphism
$$\pi':\mathcal M_{\mathcal P'}\isomto\cmp\times_Y Y'\to Y'.$$
By assumption $\cmp$ is a finite type separated DM-stack, and $Y'\to Y$ is flat or the stack $\cmp$ is tame. Thus  \cite[Lem 2.2.2 and Lem 2.3.3]{abvi:compactstable} give that $\pi':\mathcal M_{\mathcal P'}\isomto \cmp\times_Y Y'\to Y'$ is a coarse moduli space.  This completes the proof of (i) and (ii).

We now prove (iii). Let $\spec(k)\to S$ be a geometric point. To relate the forgetful maps of $Y'(k)$ and $Y(k)$, we let $\mathcal M_{\mathcal P'}\to \cmp$ be the morphism of categories over $\mathcal M$ obtained by composing $\mathcal M_{\mathcal P'}\isomto \cmp\times_Y Y'$ with the first projection and we consider
$$
\xymatrix@R=4em@C=4em{
Y'(k) \ar[r]^{(\pi')^{-1}} \ar[d] & [\mathcal M_{\mathcal P'}(k)] \ar[r]^{(A,\beta)\mapsto A} \ar[d]  & \absg(k)\ar[d]^{\textnormal{id}} \\
Y(k) \ar[r]^{\pi^{-1}} &  [\cmp(k)] \ar[r]^{(A,\alpha)\mapsto A}  &  \absg(k).} 
$$
Here the left hand square commutes since $\mathcal M_{\mathcal P'}\to \cmp$ factors through $\cmp\times_Y Y'$ and since the maps $\pi^{-1}$ and $(\pi')^{-1}$ are bijective. On the other hand, the right hand square commutes since  $\mathcal M_{\mathcal P'}\to \cmp$ is a morphism of categories over $\mathcal M$. Thus we conclude that the forgetful map $Y'(k)\to \absg(k)$ in the top is the composition of $Y'(k)\to Y(k)$ with the forgetful map $Y(k)\to\absg(k)$ in the bottom. This completes the proof of (iii). 

We next show (iv). The first claim in (iv) implies the second claim, since (the pullback of) any monomorphism of schemes is separated and the pullback of any finite \'etale cover is again a  finite \'etale cover. We now show the first claim: For each $x\in \mathcal M(T)$ with $T$ any $S$-scheme, the quotient $J'_x=\mathcal P'(x)/\Aut(x)$ injects into $J_x=\mathcal P(x)/\Aut(x)$. As $\cmp\times_Y Y'\isomto \mathcal M_{\mathcal P'}$ from (i) is an equivalence of categories over $\mathcal M$, we obtain the commutative diagram: $$
\xymatrix@R=4em@C=4em{
[(\cmp\times_{Y}Y')(T)] \ar[r]^{ \ \ \ \ \sim} \ar[d] & [\mathcal M_{\mathcal P'}(T)]\ar[d]^{\phi'}  \\
[\cmp(T)] \ar[r]^\phi &  [\mathcal M(T)].} 
$$
Here $\phi$ and $\phi'$ are the forgetful maps induced by $(x,\alpha)\mapsto x$. The map on the left is injective since the projection $\cmp\times_Y Y'\to \cmp$ is a (base change of a) monomorphism. Sending $[(x,\alpha)]$ in $\phi^{-1}([x])=\{[(x,\alpha)],\, \alpha\in\mathcal P(x)\}$ to the $\Aut(x)$-orbit of $\alpha$ defines a bijection $\phi^{-1}([x])\isomto J_x$, and exactly the same argument gives $(\phi')^{-1}([x])\isomto J'_x$. Thus the commutative diagram injects $J'_x$ into $J_x$ as desired. This completes the proof of (iv). 

To show (v) we let $U\subseteq Y$ and $U'\subseteq Y'$ be the complements of the branch loci $B\subseteq Y$ and $B'\subseteq Y'$ respectively. The base change of $f:Y'\to Y$ via the open immersion $U\hookrightarrow Y$ is an \'etale morphism $Y'_U\to U$. As $\pi':\mathcal M_{\mathcal P'}\to Y'$ identifies with the base change $\pi_{Y'}$ of $\pi$ by (ii), we then deduce that $\pi'_{Y'_U}:(\mathcal M_{\mathcal P'})_{Y'_U}\to Y'_U$ is \'etale. It follows that $f^{-1}(B^c)=Y'_U\subseteq U'$ and hence $f(B')\subseteq B$ as claimed in (v). Suppose now in addition that $f:Y'\to Y$ is \'etale. To prove that $f^{-1}(B)=B'$, it suffices to show that $f(U')\subseteq U$ since we already proved $f^{-1}(U)\subseteq U'$. As $\cmp$ is a DM-stack over $S$, it admits an \'etale cover by a scheme. Then, on using the same arguments as in \cite[(5.9)]{vkkr:chms}, we obtain that $\pi_{f(U')}$ is \'etale and hence $f(U')\subseteq U$ as desired. This completes the proof of (v) and thus of Lemma~\ref{lem:coarsebasechange}.\end{proof}

\begin{proof}[Proof of Lemma~\ref{lem:extndmod}]
The variety $M$ over $\QQ$ extends over an open dense $S\subseteq \spec(\ZZ)$ to a tame coarse moduli scheme $M_S$ over $S$ of finite level with branch locus $B_{S}$. We can freely replace here $S$ by any open dense $S'\subseteq S$ since $(M_S)_{S'}$ is again a tame coarse moduli scheme over $S'$ of finite level with branch locus $(B_S)_{S'}$ by Lemma~\ref{lem:coarsebasechange}. Write $$M^0=M_S\setminus B_S.$$
Let $X_\QQ$ be a non-degenerate model in $\mb$. Then $X_\QQ$ is proper over $\QQ$ and there exists an immersion $X_\QQ\hookrightarrow M^0_\QQ$ of $\QQ$-schemes since the branch locus $B$ of $M$ satisfies $B=(B_S)_\QQ$ by  Lemma~\ref{lem:coarsebasechange}. After possibly replacing $S$ by an open dense subscheme, spreading out gives a proper scheme $X$ over $S$ extending $X_\QQ$ and an immersion  $$X\hookrightarrow M^0$$
extending  $X_\QQ\hookrightarrow M^0_\QQ$. The composition of $X\hookrightarrow M^0$ with the open immersion $M^0\hookrightarrow M_S$ is a monomorphism $X\hookrightarrow M_S$ of finite type and $M_S$ is tame. Thus Lemma~\ref{lem:coarsebasechange} gives that $X$ is a coarse moduli scheme over $S$ of finite level and that $X\hookrightarrow M_S$ sends the branch locus of $X$ into $B_S\subseteq M_S$. Hence the branch locus of $X$ is empty since $X\hookrightarrow M_S$ factors through $M^0=M_S\setminus B_S$. We conclude that $X$ has all properties claimed in Lemma~\ref{lem:extndmod}.\end{proof}

Let $\mathcal M^I$ be a Hilbert moduli stack over $\ZZ$ defined with respect (\cite[$\mathsection$3]{vkkr:hms}) to a totally real field $L$ of degree $g$. On taking in the proofs of Lemmas~\ref{lem:extndmod} and \ref{lem:coarsebasechange} the base
\begin{equation}\label{eq:hilbanalogue}
\mathcal M=\mathcal M^I_S,
\end{equation}
we see that the statements of Lemmas~\ref{lem:extndmod} and \ref{lem:coarsebasechange} both hold (in the Hilbert case) if one replaces therein `coarse moduli scheme' by `coarse Hilbert moduli scheme'. 

\begin{proof}[Proof of Lemma~\ref{lem:formal}]
We first clarify the structure of $\M_{\P}\times_YY'$ over $\mathcal M$. The forgetful functor $f:\M_{\P}\rightarrow\M$ defines a category fibred in sets over $\mathcal M$. As $Y\to Y'$ is a monomorphism, the projection $p:\M_{\P}\times_YY'\rightarrow\M_{\P}$ defines a category fibred in groupoids over $\M_{\P}$. Thus $\M_{\P}\times_YY'$ is a category fibred in groupoids over $\M$ via $fp$. 
We claim that $\M_{\P}\times_YY'$ is moreover a category fibred in setoids over $\M$ via $fp$. As $Y'\rightarrow Y$ is a monomorphism, the projection $p:\M_{\P}\times_YY'\rightarrow \M_{\P}$ is a monomorphism and thus $p$ is fully faithful by [Sta, 04ZZ]. Hence for any $U\in\textup{Ob}(\M)$ the induced functor
\[p_U:(\M_{\P}\times_YY')_U\rightarrow (\M_{\P})_U\]
is fully faithful by \cite[003Z]{sp}. For any $x\in (\M_{\P}\times_YY')_U$ and $\varphi\in \textup{Hom}_{(\M_{\P}\times_YY')_U}(x,x)$, we have $p_U(\varphi)\in \textup{Hom}_{(\M_{\P})_U}(p(x),p(x))$. As $\M_{P}$ is a category fibred in sets over $\M$ via $f$, it holds $p_U(\varphi)=\textup{id}$ and thus $\varphi=\textup{id}$ since $p_U$ is faithful. We conclude that $\M_{\P}\times_YY'$ is a category fibred in setoids over $\M$ via $fp$. This proves our claim.

The claim and \cite[0045]{sp} show that $\M_{\P}\times_YY'\to \M$ factors as $\M_{\P}\times_YY'\isomto \mathcal{X}\to\M$ with $\M_{\P}\times_YY'\isomto \mathcal{X}$ an equivalence of categories over $\M$ and $\mathcal{X}\to \mathcal M$ a category fibred in sets over $\M$. Then \cite[02Y2]{sp} provides a presheaf $\P'$ on $\M$ such that $\mathcal{X}\to \mathcal M$ factors as $\mathcal{X}\isomto\M_{\P'}\to\M$ with
$\mathcal{X}\isomto \M_{\P'}$ an equivalence of categories over $\M$. Thus the composition $\M_{\P}\times_YY'\isomto\mathcal X\isomto\M_{\P'}$ is an equivalence of categories over $\M$, and hence \cite[003Z]{sp} gives an equivalence $\M_{\P'}\isomto \M_{\P}\times_YY'$ of categories over $\M$ as desired.
\end{proof}

\subsection{Proof of Proposition~\ref{prop:parshinmodulischeme}}

Let $Y$ be a variety over $\ZZ$, and let $Z\subset Y$ be a closed subscheme. As in Proposition~\ref{prop:parshinmodulischeme} we assume that there exists an open dense $S\subseteq \spec(\ZZ)$ such that $Y_S=M_{\mathcal P}$ is a coarse moduli scheme over $S$ of finite level with branch locus $B\subseteq Z_S$. We write $\mathcal M=(\sabsg)_S$.

\begin{proof}[Proof of Proposition~\ref{prop:parshinmodulischeme}]
To prove the claims, we can freely shrink $S$ to some open dense $S'\subseteq S$ and we may and do assume that $Z$ and the branch locus $B$ of $Y_S=M_{\mathcal P}$ are both empty. Indeed for $Y^0=Y_{S'}\setminus B_{S'}$ the arguments in \cite[(10.11)]{vkkr:chms} show $$(Y\setminus Z)(S)\hookrightarrow Y^0(S')$$   and   $Y^0$ is a variety over $\ZZ$ with the following properties by  Lemma~\ref{lem:coarsebasechange}: It is a coarse moduli scheme over $S'$ of finite level  with empty branch locus and with forgetful map factoring as $Y^0(\qb)\hookrightarrow Y(\qb)\to\absg(\qb)$ for $Y(\qb)\to\absg(\qb)$ the forgetful map of $Y_S=M_{\mathcal P}$.  

As $Y_S=M_{\mathcal P}$ is of finite level by assumption, there exist (after possibly shrinking $S$) a scheme $Y'$ and a finite \'etale cover $Y'\to\cmp$ over $S$. The initial morphism $\cmp\to Y_S$ is proper and quasi-finite since $\cmp$ is a separated finite type DM stack over $S$, and $\cmp\to Y_S$ is \'etale since $B$ is empty by assumption. Thus we obtain a finite \'etale cover $$Y'\to Y_S.$$  
Then \cite[(6.5)]{rvk:cetraro} gives an injection $Y_S(S)\hookrightarrow Y'(T)$ for an open $T\subseteq \spec(\OK)$ with $K$ a number field. Let $Y'\to \M$ be the composition of $Y'\to\cmp$  with the forgetful morphism $\cmp\to\mathcal M$. It corresponds via 2-Yoneda to an object $y\in \mathcal M(Y')$ and there exists an equivalence $Y'\isomto \mathcal M_{\mathcal P'}$ of stacks over $S$ for $\mathcal P'=\Hom_{\M}(-,y)$ such that the diagram $$
\xymatrix@R=4em@C=4em{
Y'(T) \ar[r]^{\sim} \ar[d] & [\mathcal M_{\mathcal P'}(T)]\ar[d]  \\
[\cmp(T)] \ar[r] &  \absg(T)} 
$$
commutes. Here the map on the left is finite since $Y'\to Y_S$ is a finite morphism of varieties over $S$, while the bottom map is finite since $Y_S=M_{\P}$ is of finite level by assumption. As the diagram commutes, the forgetful map $Y'(T)\to \absg(T)$ of the moduli scheme $Y'=M_{\mathcal P'}$ over $S$ is thus finite. Composing with $Y(S)\hookrightarrow Y_S(S)\hookrightarrow Y'(T)$ gives a Par{\v{s}}in construction $$\phi:Y(S)\to  \absg(T).$$
We claim that $\phi$ is defined by the forgetful map $Y(\bar{\QQ})\to \absg(\bar{\QQ})$. Indeed this follows since the injection $Y_S(S)\hookrightarrow Y'(T)$ is defined by pulling back $S\to Y_S$ via $Y'\to Y_S$, the morphism $Y'\to \mathcal M$ factors as $Y'\to \cmp\to \mathcal M$, and the above displayed diagram commutes. Moreover, on using the same formal arguments as in step 2. of the proof of \cite[Thm 10.1]{vkkr:chms}, we see that the image of $\phi$ is $G_\QQ$-stable. This proves (i). 
To show (ii) we assume that $Y_{\OL}$ is a coarse Hilbert moduli scheme over $\OL=\ZZ[1/\nu]$. Then we go into the constructions of \cite{vkkr:chms}: On using precisely the same arguments as in the proofs of \cite[Thm 10.1 and Thm 4.1]{vkkr:chms}, we obtain a Par{\v{s}}in construction $$\phi:(Y\setminus Z)(\OL)\to \absg(T_0)$$ such that the image of $\phi$ admits for $(d,n)=(3^{4g},3\nu)$ an effective covering $\cup_TJ(T)$ of $\gl2$-type with $G_\QQ$-isogenies. Here $T_0\subseteq \spec(\OL_K)$ is the open subscheme given by the integral closure of $\ZZ[\tfrac{1}{3\nu}]$ in the
compositum $K\subset\qb$ of all number fields $L\subset\qb$ of degree $[L:\QQ]\leq 3^{4g}$ such that $L/\QQ$
is unramified at each rational prime $p\nmid 3\nu$; see step 6. in \cite[$\mathsection$10.1]{vkkr:chms} or \cite[(6.5)]{rvk:cetraro}. This proves (ii) and thus completes the proof of Proposition~\ref{prop:parshinmodulischeme}.
\end{proof}

\subsection{Proof of Theorem~\ref{thm:parshin}}

Let $X$ be a variety over $\QQ$. Further let $M$ be a variety over $\QQ$ which extends over an open dense subscheme of $\spec(\ZZ)$ to a tame coarse moduli scheme of finite level and let $\mb$ be a compactification of $M$.  We assume that $X$ is non-degenerate in $\mb$.
\begin{proof}[Proof of Theorem~\ref{thm:parshin}]
As $X$ is non-degenerate in $\mb$, there exists an open dense $U_\QQ\subseteq X$ with an open immersion $U_\QQ\hookrightarrow Y_\QQ$ of $\QQ$-schemes into a non-degenerate model $Y_\QQ$ of $X$ in $\mb$. Lemma~\ref{lem:extndmod} gives that $Y_\QQ$ extends over an open dense $S\subseteq\spec(\ZZ)$ to a proper coarse moduli scheme $Y$ over $S$ of finite level with empty branch locus. As $Y$ is proper over $S$, we obtain $Y(\QQ)\isomto Y(S)$ and thus Proposition~\ref{prop:parshinmodulischeme}~(i) provides a Par{\v{s}}in construction $$\phi:U_\QQ(\QQ)\hookrightarrow Y_\QQ(\QQ)\to^{\phi_Y}\absg(T) \quad \textnormal{with} \quad \phi_Y:Y_\QQ(\QQ)\isomto Y(\QQ)\isomto Y(S)\to\absg(T)$$ 
 defined by the forgetful map $Y(\qb)\to\absg(\qb)$. After possibly replacing $S$ by an open dense subscheme, spreading out gives a variety $U$ over $S$ extending $U_\QQ$ and an open immersion $U\hookrightarrow Y$ extending $U_\QQ\hookrightarrow Y_\QQ$.
As $U\hookrightarrow Y$ is a flat monomorphism of finite type, we may and do apply Lemma~\ref{lem:coarsebasechange}. This gives that $U$ is a coarse moduli scheme over $S$ of finite level with empty branch locus such that the forgetful map $U(\bar{\QQ})\to \absg(\bar{\QQ})$ factors as $$U(\bar{\QQ})\hookrightarrow Y(\bar{\QQ})\to \absg(\bar{\QQ}).$$
It follows that $\phi:U_\QQ(\QQ)\to \absg(T)$ is a Par{\v{s}}in construction defined by the forgetful map $U(\bar{\QQ})\to \absg(\bar{\QQ})$. Moreover the image of $\phi$ is $G_\QQ$-stable since the image of $\phi_Y$ is $G_\QQ$-stable by Proposition~\ref{prop:parshinmodulischeme}~(i). This proves (i) and also (ii) since the above arguments work for any $U_\QQ$ of the form $U_\QQ=U_{Y_\QQ}$ as in \eqref{def:controlledopen} with $Y_\QQ$ a non-degenerate model of $X$ in $\mb$.

To show (iii) we suppose that $U_\QQ=U_{Y_\QQ}$ for some non-degenerate model $Y_\QQ$ of $X$ in $\mb$ and we assume that we are in the Hilbert case. Then (the proof of) Lemma~\ref{lem:extndmod} and \eqref{eq:hilbanalogue} give that $Y$ is a proper coarse Hilbert moduli scheme over $\mathcal O=\ZZ[1/\nu_Y]$ of finite level with empty branch locus, where $\nu_Y$ is the integral degeneration \eqref{def:intdeg}. Thus, on replacing in the above arguments Proposition~\ref{prop:parshinmodulischeme}~(i) by (ii), we obtain a Par{\v{s}}in construction $$\phi:U_\QQ(\QQ)\hookrightarrow Y_\QQ(\QQ)\to^{\phi_Y}\absg(T_0)$$ whose image admits for $(d,n)=(3^{4g},3\nu_Y)$ an effective covering $\cup_TJ(T)$ of $\gl2$-type with $G_\QQ$-isogenies. Moreover, as $Y$ is a coarse Hilbert moduli scheme over $\OL$, the above arguments show that $\phi$ is defined by the forgetful map $U(\qb)\to\absg(\qb)$ of a coarse Hilbert moduli scheme $U$ over $S$ of finite level with empty branch locus where $S\subseteq \spec(\ZZ)$ is nonempty open. This proves (iii) and thus completes the proof of Theorem~\ref{thm:parshin}.
\end{proof}

\subsection{Proof of the explicit bounds in Corollary~\ref{cor:hms}}

Let $M\subseteq \mb$ be as above and assume that we are in the Hilbert case. As in Corollary~\ref{cor:hms}, let $X$ be a variety over $\QQ$ which is non-degenerate in $\bar{M}$ and let $U\subseteq X$ be an open subscheme of the form $U=U_Y$ with $Y$ a non-degenerate model of $X$ in $\mb$.
\begin{proof}[Proof of Corollary~\ref{cor:hms}]
As we are in the Hilbert case, Theorem~\ref{thm:parshin} gives a Par{\v{s}}in construction $\phi:U(\QQ)\to \absg(T)$ whose image admits for $(d,n)=(3^{4g},3\nu_Y)$ an effective covering $\cup_TJ(T)$ of $\gl2$-type with $G_\QQ$-isogenies. Then, on using precisely the same arguments as in the proofs of \cite[Thm 10.1 and Thm 4.1]{vkkr:chms}, we see that the $\gl2$-case  of the effective 
Shafarevich conjecture (\cite[Prop 7.13 and Prop 7.14]{vkkr:chms}) leads to the claimed bounds $$h_\phi(x)\leq c_1\nu_Y^{e_1}\quad \textnormal{ and }\quad |U(\QQ)|\leq c_2\nu_Y^{e_2}.$$
Here we can take the same numerical constants as in \cite[Thm~4.1]{vkkr:chms} since the proof of \cite[Thm~4.1]{vkkr:chms} also reduces to the situation where $3$ is invertible in the base. To obtain the constants $c_2$ and $e_2$ depending only on $M$, we used the injections $$U(\QQ)\hookrightarrow Y(\QQ)\hookrightarrow (M_{\bs}\setminus B_{\bs})(R)$$
where $R=\ZZ[\tfrac{1}{3\nu_Y}]$ and $M_\bs$ is the coarse Hilbert moduli scheme over $\bs$ with branch locus $B_\bs$ which appears in the definition \eqref{def:intdeg} of $\nu_Y$. The displayed injections come from the immersions constructed in the proofs of Lemma~\ref{lem:extndmod}, Proposition~\ref{prop:parshinmodulischeme} and  Theorem~\ref{thm:parshin}.
\end{proof}

\section{Completing the proof of Theorem~\ref{thm:main}}\label{sec:construction}

In this section we complete the proof of Theorem~\ref{thm:main}. We include the proofs of the two observations in (ii) and we deduce (iii) by using the proof of (i).

\subsection{Proof of Theorem~\ref{thm:main}~(ii)}
To prove the two observations in Theorem~\ref{thm:main}~(ii), we now give two lemmas which both can be deduced (in a rather direct way) from standard results in algebraic geometry. 

Let $k$ be an infinite field, let $M$ be a variety over $k$ and let $\mb$ be a compactificaton of $M$. The first observation in Theorem~\ref{thm:main}~(ii) is a special case of the following result.

\begin{lemma}
Let $X$ be an irreducible geometrically reduced variety over $k$. Suppose that $M$ is rational over $k$ and that $\dim(X)<\dim(M)$. Then there exist a closed subscheme $Y\subseteq \mb$ and open dense $U\subseteq X$ and $U'\subseteq Y$ with an isomorphism $U\isomto U'$ of $k$-schemes.
\end{lemma}
\begin{proof}
We work over $k$. It suffices to construct an open dense $U\subseteq X$ with an immersion $U\hookrightarrow \mb$. This can be done by composing an immersion $U\hookrightarrow\mathbb P_k^{m}$ with a birational map to $\mb$ for $m=\dim(M)$. To assure that the composition is well-defined, we apply a suitable automorphism of $\mathbb P^{m}_k$ whose construction uses that $k$ is infinite. 

We include the details.   
 As $\dim(X)<m$, our $X$ admits an immersion into an irreducible geometrically reduced variety of dimension $m-1$. 
Thus we may and do assume that $\dim(X)=m-1$. Then  \cite[A.11.5]{bogu:diophantinegeometry} gives an open dense $U\subseteq X$ with an immersion $$U\hookrightarrow \mathbb P^m_{k}.$$
As $M$ is rational over $k$, there exists an open dense $W\subseteq \mathbb P^m_k$ with an
immersion $W\hookrightarrow M$. Let $x\in\mathbb P^m_k$ be the image of the generic point of $U$ under $U\hookrightarrow \mathbb P^m_{k}$ and let $G$ be the automorphism group of $\mathbb P^m_k$. We claim that there exists $g\in G$ with $gx\in W$. After composing $U\hookrightarrow \mathbb P^m_k$ with $g$, we may and do assume that $x\in W$. Then, after possibly replacing $U$ by an open dense subset of $U$, we obtain immersions $$U\hookrightarrow W\hookrightarrow M\hookrightarrow \mb$$
 as desired. It remains to prove our claim. If the claim does not hold, then the orbit $Gx$ is contained in the complement $V$ of $W$ in $\mathbb P^m_k$ which gives a contraction: The orbit $Gx$ is dense in $\mathbb P^m_k$ since $k$ is infinite,   but $V$ is not dense since it is a proper closed subset of $\mathbb P^m_k$. This proves our claim, and thus completes the proof of the lemma. \end{proof}

In the setting of Theorem~\ref{thm:main}~(ii), we obtain the second observation by applying the next lemma with the closed subset $Z=\mbd=B\cup (\bar{M}\setminus M)$ of $Y=\mb$. Here $\mbd$ is indeed closed in $\mb$, since $M$ is open in $\mb$ and the branch locus $B\subseteq M$ is closed in $M$. 
\begin{lemma}\label{lem:primeavoidance}
Let $Y$ be an integral projective variety over $k$ and let $Z\subset Y$ be a closed subset. For each $n\in\ZZ$ with $0\leq n<\dim(Y)-\dim(Z)$, there exist infinitely many integral closed subschemes $X\subset Y$ with $\dim(X)=n$ and $X\cap Z=\emptyset$.
\end{lemma}
\begin{proof}The idea is as follows: We first use the prime avoidance lemma to show for each $Z$ the existence of at least one $X$ with the desired properties, and then we apply this existence result with suitable $Z$ to obtain infinitely many $X$. We include the arguments.


To produce at least one $X$, we do induction on  $m=\dim(Y)$. As $X$ exists when  $m=1$, we now assume that $m\geq 2$. Write $Y=\textnormal{Proj}(A)$ with $A$ a finitely generated $k$-algebra. The prime avoidance lemma implies $(*)$: For any finite subset $T\subset Y$ and each open neighbourhood $U\neq Y$ of $T$, there exists a homogeneous $f\in A_+$ with $T\subseteq D_+(f)\subseteq U$. Now, for the induction $m-1\mapsto m$, we apply $(*)$ with  $T$ the finite set of points $z\in Z$ with $\dim(\overline{\{z\}})=\dim(Z)$.  This gives a homogeneous $f\in A_+$ and an integral closed subscheme $Y'$ of $V_+(f)$ of dimension $m-1$ and then applying the induction hypothesis with $Y'$ and $Z'=Y'\cap Z$ provides an integral closed subscheme $X\subset Y$ with the desired properties.

 To deduce the existence of infinitely many $X$ with the desired properties, we assume that the set $\{X_i\}$ of such $X$ is finite and we derive a contradiction. Let $x_i\in X_i$ be a closed point. As $T=\cup_i\{x_i\}$ is a finite set of closed points of $Y$, we obtain that $Z'=Z\cup T$ is a closed subset of $Y$ with $\dim(Z')=\dim(Z)$. Then an application of the above existence result with $Z'$ gives an integral closed subscheme $X\subset Y$ of dimension $n$ which satisfies $X\cap Z'=\emptyset$ and thus $X\cap Z=\emptyset$. Hence $X$ lies in $\{X_i\}$. In particular $X$ contains some $x_i$,  which contradicts $X\cap Z'=\emptyset$. This completes the proof of Lemma~\ref{lem:primeavoidance}. \end{proof}

\subsection{Proof of Theorem~\ref{thm:main}~(iii)}

We now deduce Theorem~\ref{thm:main}~(iii) using the proof of (i) which relies on Faltings' finiteness result \cite[Thm 6]{faltings:finiteness}. It would be interesting to find a purely geometric proof of (iii).

\begin{proof}[Proof of Theorem~\ref{thm:main}~(iii)]
To prove the claim, we assume for the sake of contradiction that the statement does not hold. Hence the curve $X_{\qb}$ over $\qb$ contains an irreducible component of geometric genus 0 or 1. Thus there exists a number field $K$ such that $X_K$ has infinitely many $K$-rational points. As $X$ is non-degenerate by assumption, we can apply the geometric arguments underlying the proof of Theorem~\ref{thm:main}~(i) which work exactly the same over any $K$. After passing to an open dense $U\subseteq X$, this gives a finite map $U(K)\to\absg(T)$ for some $g\in \NN$ and some open $T\subseteq \spec(\OL_L)$ where $L$ is a number field. As $\absg(T)$ is finite by Faltings~\cite{faltings:finiteness}, we thus obtain that $U(K)$ and $X_K(K)$ are also finite. This gives a contradiction and hence proves Theorem~\ref{thm:main}~(iii).
\end{proof}
\section{Non-degenerate curves and rational points}\label{sec:nondegcurves}

In this section we study various aspects of non-degenerate curves over $\QQ$ and their rational points. We continue the notation and terminology which we introduced in $\mathsection$\ref{sec:introico}.

\subsection{Geometric results}\label{sec:geomicoresults}

To study rational points on non-degenerate curves, we prove several geometric results for such curves and their ico models. In this section we summarize our geometric results.

\paragraph{Geometric properties.}Theorem~\ref{thm:main}~(ii) combined with a classical construction going back at least to the works of Clebsch~\cite{clebsch:clebschsurface} and Klein~\cite{klein:cubicsurfaces} gives part (i) of the following result. On the other hand, part (ii) is a consequence of Theorem~\ref{thm:main}~(iii).
\begin{corollary}\label{cor:geom} 
Let $X$ be a curve over $\QQ$.
\begin{itemize}
\item[(i)] If $X$ is integral, then it has an ico model over $\ZZ$.
\item[(ii)] If $X$ is non-degenerate, then $X$ has geometric genus $\geq 2$.
\end{itemize}
\end{corollary}

Our proof of the geometric statement in (ii) crucially relies on arithmetic since Theorem~\ref{thm:main}~(iii) uses Faltings' finiteness result~\cite[Thm 6]{faltings:finiteness}. It would be interesting to find a purely geometric proof of (ii).  The converse of Corollary~\ref{cor:geom}~(ii) is an open problem:

\vspace{0.3cm}
 $(deg)$ Determine which curves over $\QQ$ of geometric genus $\geq 2$  are non-degenerate.
\vspace{0.3cm}

{
\noindent In ico form most curves over $\QQ$ of geometric genus $\geq 2$ are non-degenerate by Proposition~\ref{prop:generalico} below, and  we solve problem $(deg)$ for certain classes of plane curves in Section~\ref{sec:p2}. However $(deg)$ is still widely open and solving it in general requires new ideas.
\paragraph{General ico models.} We next discuss  basic geometric properties of general ico models. Let $n\in\NN$ and let $s_1,\dotsc,s_r$ be monomials in $\QQ[x_0,\dotsc,x_4]$ of degree $n$ which form a basis of the $\QQ$-module given by the $n$-th graded part $A_n$ of $A=\QQ[x_0,\dotsc,x_4]/(\sigma_2,\sigma_4)$. Here 
\begin{equation}\label{eq:dimr}
r=4n^2-4n+6
\end{equation}
if $n\geq 2$ and $r=5$ if $n=1$, and $\sigma_k$ denotes the $k$-th elementary symmetric polynomial where $k\in\NN$. We may and do assume that $s_i=x_{i-1}^n$ for $i\leq 5$.   
Write $\mathbb P^4=\mathbb P^4_\QQ$ and $\mathbb A^r=\mathbb A^r_\QQ$. For each $v\in \mathbb A^r(\QQ)$, we consider the closed subscheme $X_v\subset \mathbb P^4$ defined by      
\begin{equation}\label{def:genicomodel}
X_v\cong \textnormal{Proj}(A/f)\cong X_f, \quad f=\sum v_is_i.
\end{equation}
We say that a general ico model $X_v$ over $\QQ$ of degree $n$ has a property $\mathcal P$ if there exists an open dense $U\subseteq \mathbb A^r$ such that $X_v$ has property $\mathcal P$ for all $v\in U(\QQ)$. Part (i) of the following proposition solves the above problem $(deg)$ for a general ico model $X_v$ over $\QQ$.

\begin{proposition}\label{prop:generalico}
For each $n\in \NN$, the following statements hold.
\begin{itemize}
\item[(i)] A general ico model $X_v$ over $\QQ$ of degree $n$ is a non-degenerate curve over $\QQ$.
\item[(ii)] A general ico model $X_v$ over $\QQ$ of degree $n$ is a smooth, projective and geometrically connected curve over $\QQ$ of genus $(2n+1)^2$ which is non-degenerate.
\item[(iii)] If $U\subseteq\mathbb A^r$ denotes the open dense subset of (ii), then $v\mapsto X_v$ injects $U(\QQ)/\QQ^\times$ into the set $\mathcal C$ of smooth, projective and geometrically connected curves inside $\mathbb P^4$. 
\end{itemize}
\end{proposition}
In (i) we can take the explicit $U^{\textnormal{nd}}=\mathbb A^r\setminus \bigl(\cup_{i=1}^{5} V(z_i)\bigl)$, and Bertini applied with the normal surface $\textnormal{Proj}(A)$ over $\QQ$ gives the open dense $U\subseteq U^{\textnormal{nd}}$ of (ii). Let $g\in \NN$. If $g\geq 2$ is an odd square, then taking $n=\tfrac{\sqrt{g}-1}{2}$ in Proposition~\ref{prop:generalico} provides a moduli space $U$ of dimension $r\sim g$ by \eqref{eq:dimr} which parametrizes non-degenerate curves in $\mathcal C$ of genus $g$. 
}

\paragraph{Moduli interpretation.}We deduce from Theorem~\ref{thm:parshin} and \cite[Thm E]{vkkr:chms} that any non-degenerate curve $X$ admits an open dense $U\subseteq X$ whose points have (\'etale locally) a simple  moduli interpretation in terms of abelian surfaces. More precisely we deduce:

\begin{corollary}\label{cor:icomoduliint} 
Let $X$ be a curve over $\QQ$ which is non-degenerate, and let $U\subseteq X$ be open.
\begin{itemize}
\item[(i)] Suppose that $U$ is of the form $U=U_f$ with $X_f$ a non-degenerate ico model of $X$ over $\ZZ$. Then $U$ extends over a nonempty open  $S\subseteq\spec(\ZZ)$ to a coarse Hilbert moduli scheme $U_S$ of finite level with empty branch locus and with moduli interpretation $(*)$.
\item[(ii)] Moduli interpretation $(*)$: Let $k$ be an algebraically closed field with $\textnormal{char}(k)\in S$. Any point in $U_S(k)$ identifies with an isomorphism class of $(A,\alpha)$ where $A$ is a principally polarized abelian surface over $k$ with real multiplication $\ZZ[\tfrac{1+\sqrt{5}}{2}]\hookrightarrow \End(A)$ and where $\alpha$ is a restricted symplectic level 2-structure of $A$ as in \eqref{eq:moduliint}. 
\item[(iii)] In the case $X=(X_f)_\QQ$, one can take $U=X$ and  $S=\ZZ[1/\nu]$ with $\nu=30\nu_f$.
\end{itemize}
\end{corollary}
Here we recall from $\mathsection$\ref{sec:introico} that $\nu_f$ is defined as follows. If $f=(f_j)$ and $\sum a_{ij}x_i^{n_j}$ is the diagonal part of the homogeneous $f_j\in\ZZ[x_0,\dotsc,x_4]$ of degree $n_j\geq 1$, then
\begin{equation}\label{eq:defnuf}
\nu_f=\textnormal{rad}\bigl(\prod a_{ij}\bigl)
\end{equation}
with the product taken over all nonzero diagonal coefficients $a_{ij}$ of $f=(f_j)$. 
\subsection{Rational points}\label{sec:ratpoints}

In this section we first give an effective Par{\v{s}}in construction for the rational points of non-degenerate curves over $\QQ$. Then we deduce explicit bounds for the rational points of such curves. We work over $\QQ$ and we write $\mathbb P^n$ for $\mathbb P^n_\QQ$ where $n\in\NN$.

\paragraph{Degree and height.} 
Let $X$ be a curve over $\QQ$. Suppose that $X$ admits a finite morphism $X\to \mathbb P^n$ for some $n\in\NN$. Then pulling back the usual (\cite[p.16]{bogu:diophantinegeometry}) logarithmic Weil height $h_{\mathbb P^n}$ on $\mathbb P^n$ along $X\to \mathbb P^n$, defines a (logarithmic) Weil height 
\begin{equation}\label{def:weilheight}
h:X(\qb)\to \RR.
\end{equation}
Assume $X$ is integral. To define a normalized degree of $X$,  write $\deg(X)$ for the degree of the scheme theoretic image of $X\to \mathbb P^n$ and  let $\xt\to X$ be a normalization. Then we take $$d_X=\deg(\xt\to X)\deg(X).$$
To define a height of $X$,  choose an immersion $U_f\hookrightarrow (X_f)_\QQ\subset\mathbb P^4$ for $X_f$ an ico model of $X$ over $\ZZ$. This induces a finite morphism $\varphi:\xt\to \mathbb P^4$ via the universal property of normalization of curves. 
The composition $\psi:\xt\to X\to \mathbb P^n$ is also finite. Then we define
$$h(X)=h_{1,1}\bigl((\varphi,\psi)(\xt)\bigl).$$
Here $h_{1,1}$ is the biprojective height \cite[$\mathsection$5]{remond:construction} of index (1,1),  and  $(\varphi,\psi)(\xt)$ is the scheme theoretic image of the composition $\xt\to^{\Delta} \xt\times \xt\to^{\varphi\times \psi} \mathbb P^4\times \mathbb P^n$. In what follows, when writing `let $h$ be a Weil height on $X$ as in \eqref{def:weilheight}' we always assume that we are given a finite morphism $X\to \mathbb P^n$ and when using the quantities $d_X$ and $h(X)$ we always assume that they are defined with respect to this given $X\to\mathbb P^n$ and a chosen immersion $U_f\hookrightarrow (X_f)_\QQ$.


\subsubsection{Effective Par{\v{s}}in and effective Mordell for non-degenerate curves}
Let $X$ be a curve over $\QQ$ which is non-degenerate and let $U\subseteq X$ be open. Suppose that $U=U_f$ for some non-degenerate ico model $X_f$ of $X$ over $\ZZ$ and let $\nu_f$ be as in \eqref{eq:defnuf}. 

\paragraph{Effective Par{\v{s}}in.} Corollary~\ref{cor:icomoduliint} gives that $U$ is a coarse Hilbert moduli scheme over $\QQ$. We deduce from Theorem~\ref{thm:parshin} an effective Par{\v{s}}in construction for $U(\QQ)$.
 
\begin{corollary}\label{cor:effparico}The forgetful map $U(\qb)\to\underline{A}_2(\qb)$ of the coarse moduli scheme $U$ defines an effective Par{\v{s}}in construction $\phi:U(\QQ)\to\underline{A}_2(T)$ with the following properties.

\begin{itemize}
\item[(i)]  The image $\phi(U(\QQ))$ admits for $(d,n)=(3^8,30\nu_{f})$ an effective covering $\cup_T J(T)$ of $\gl2$-type with $G_\QQ$-isogenies in the sense of \eqref{def:effcover}.
\item[(ii)] Let $h$ be a Weil height on $X$ as in \eqref{def:weilheight} and let $h_\phi$ be the pullback of $h_F$ along $U(\qb)\to\underline{A}_2(\qb)$. Firstly, if $X$ is integral then these heights on $U(\qb)$ satisfy
$$h\leq d_X\bar{h}_\phi+h(X), \ \ \quad \bar{h}_\phi=2h_\phi+8^8\log(h_\phi+8).$$
Secondly, if $X=(X_f)_\QQ$ and $h=h_{\mathbb P^4}$, then $X=U$ and  it holds $h\leq \bar{h}_\phi$ on $X(\qb)$.
\end{itemize}  
\end{corollary}
We now briefly discuss (ii). The second bound $h\leq \bar{h}_\phi$ is essentially optimal: This bound and the function $\bar{h}_\phi$ both come from \cite[Prop 11.13]{vkkr:chms}, see \cite[$\mathsection$11.5]{vkkr:chms} for discussions of the shape of $\bar{h}_\phi$ and of the constants appearing in $\bar{h}_\phi$. The first bound $h\leq d_X\bar{h}_\phi+h(X)$ is relatively simple but rarely optimal: Its proof combines \cite[Prop 11.13]{vkkr:chms} with R\'emond's comparison \cite[Prop 5.2]{remond:construction} of Weil heights on curves.

\paragraph{Effective Mordell.} Let $h$ be a Weil height on $X$ as in \eqref{def:weilheight}.  The effective Par{\v{s}}in construction of Corollary~\ref{cor:effparico} combined with the bound for $h_\phi$ in Corollary~\ref{cor:hms} gives the following estimate for the Weil height $h$ in which $c=10^{\kappa}$ for $\kappa=10^{12}$.

\begin{corollary}\label{cor:effmordell}
If $X$ is integral, then any $x\in U(\QQ)$ satisfies $h(x)\leq c d_X\nu_f^{24}+h(X)$.
\end{corollary}
Usually one can directly control $h$ on the finite set $(X\setminus U)(\QQ)$ and thus on the whole $X(\QQ)$ by Corollary~\ref{cor:effmordell}. In Section~\ref{sec:p2} we illustrate this for plane curves $X\subset\mathbb P^2$ and we show how one can control $\nu_f$ in terms of $X$. We also obtain the following result.

\begin{corollary}\label{cor:effmordellxf}
If $X=(X_f)_\QQ$ and $h=h_{\mathbb P^4}$, then any $x\in X(\QQ)$ satisfies $h(x)\leq c\nu_f^{24}$.
\end{corollary}
In the above bounds, the term $c\nu_f^{24}$ comes from \cite{rvk:gl2,vkkr:chms} and we refer to these works for discussions of various aspects of this bound; see for example \cite[$\mathsection$1.3.3]{vkkr:chms}.

The above height bounds have useful features. For example $\nu_f$ is independent of the degrees of $f=(f_j)$. This allows us to solve in Section~\ref{sec:fermat} an analogue of the Fermat problem by combining Corollary~\ref{cor:effmordellxf} with Diophantine approximations. Moreover $\nu_f$ only depends on the diagonal parts of $f=(f_j)$. Hence varying the non-diagonal parts shows the following: Corollary~\ref{cor:effmordell} is of the form $c^*d_X+h(X)$ for large classes of $X$, and  Corollary~\ref{cor:effmordellxf} is uniform for many $f$. For the number of rational points, very strong uniform bounds were established by Dimitrov--Gao--Habegger+K\"uhne~\cite{digaha:mordell,kuhne:equidistribution} and Yuan~\cite{yuan:bogomolov}: Their bounds depend only on the genus and the MW rank, and they hold for all smooth, projective and geometrically connected curves of genus $\geq 2$ over any number field.

\subsection{Proof of geometric results}\label{sec:proofgeom}

\noindent In this section we prove various geometric properties of non-degenerate curves and ico models. In particular we deduce Corollary~\ref{cor:geom}  and we include a proof of Proposition~\ref{prop:generalico}.

\subsubsection{Equivalence of non-degenerate notions and proof of Corollary~\ref{cor:geom}}\label{sec:equivalencenondegnotions}

In $\mathsection$\ref{sec:introico} we introduced a notion `non-degenerate', while in $\mathsection$\ref{sec:results} we introduced a notion `non-degenerate in $\mb$'. To relate these notions, we write $Z_\QQ\subset\mathbb P^4_\QQ$ for the set of five closed points obtained by permuting the coordinates of $(1,0\dotsc,0)$ and we observe that
\begin{equation}\label{def:mmbar}
M=\mb\setminus Z_\QQ \quad \textnormal{and}\quad \mb\subset \mathbb P^4_\QQ: \sigma_2=0=\sigma_4
\end{equation}
satisfy all assumptions of $\mathsection$\ref{sec:results} for $\sigma_2, \sigma_4$ as in \eqref{def:icomod}. Indeed $\mb$ is a compactification of $M$, and it follows from \cite[Thm E and $\mathsection$5.2]{vkkr:chms} that $M$ extends to a tame  coarse Hilbert moduli scheme $M_{\mathbb S}$ over $\bs=\ZZ[1/30]$ of finite level with degeneracy locus $\mbd=Z_\QQ$.
\begin{lemma}\label{lem:relations}
 Let $X$ be a curve over $\QQ$, and let $X_f$ be an ico model of $X$ over $\QQ$.
\begin{itemize}
\item[(i)] The ico model $X_f$ identifies via \eqref{eq:modelident} with a model $Y$ of $X$ in $\mb$ such that $U_f=U_Y$, and any model of $X$ in $\mb$ identifies via \eqref{eq:modelident} with an ico model of $X$ over $\QQ$. 
\item[(ii)] The ico model $X_f$ is degenerate if and only if $Y\cap\mbd\neq \emptyset$.
\item[(iii)] The curve $X$ is non-degenerate if and only if $X$ is non-degenerate in $\mb$. 
\end{itemize}
\end{lemma}
To study curves which are non-degenerate, we can now freely use (via Lemma~\ref{lem:relations}) the results and constructions for varieties which are non-degenerate in $\mb$. 

\begin{proof}[Proof of Corollary~\ref{cor:geom}]
Lemma~\ref{lem:relations}~(i) shows that (i) directly follows from Theorem~\ref{thm:main}~(ii) combined with the classical result that $\mb$, and thus $M$, is rational over $\QQ$; an explicit birational map from $\mathbb P^2_\QQ$ to $\mb$ can be found in \eqref{def:taui}. Moreover Lemma~\ref{lem:relations}~(iii) shows that (ii) is equivalent to a special case of Theorem~\ref{thm:main}~(iii). This completes the proof.
\end{proof}

\begin{proof}[Proof of Lemma~\ref{lem:relations}]
All three statements of Lemma~\ref{lem:relations} can be proved via standard  arguments by unwinding the definitions. In particular (iii) is a direct consequence of (i) and (ii). We now include the arguments underlying (i) and (ii) in a more general form, working over any subring $R\subseteq \QQ$. This additional generality will be used in $\mathsection$\ref{sec:icocoarse}.

To show (i) write $X_f=\textnormal{Proj}(A/I_f)$ with $A=R[x_0,\dotsc,x_4]$ and $I_f=(\sigma_2,\sigma_4,f)$ for homogeneous $f\in A^m$, $m\in \NN$. The inclusion $(\sigma_2,\sigma_4)\subseteq I_f$ identifies $X_f$ with its scheme theoretic image $Y$ in $\mb_R$ under the induced closed immersion  \begin{equation}\label{eq:modelident}
X_f=\textnormal{Proj}(A/I_f)\hookrightarrow \textnormal{Proj}\bigl(A/(\sigma_2,\sigma_4)\bigl)=\mb_R.
\end{equation}
Conversely, any closed subscheme $Y'$ of $\mb_R$ is uniquely determined by its ideal sheaf. Thus \cite[01QP]{sp} implies that $Y'$ is the scheme theoretic image of the closed immersion $\textnormal{Proj}(A/I)\to \mb_R$ induced by some $I=I_g$ for homogeneous $g\in A^n$, $n\in\NN$. 
Moreover  $X_g$ is an ico model of $X$ over $R$ if and only if $Y'_\QQ$ is a model of $X$ in $\mb$. This proves (i).

To prove (ii) we recall that $\mbd=Z_\QQ$ consists of the five closed points $e_i$ obtained by permuting the coordinates of $e_1=(1,0,\dotsc,0)$. 
Consider $f=(f_j)$ for $f_j\in\QQ[x_0,\dotsc,x_4]$ homogeneous of degree $n_j\geq 1$ with diagonal part $\sum a_{ij}x_i^{n_j}$. After identifying via \eqref{eq:modelident} the ico model $X_f$ with $Y\subset \mb$, we compute that $Y$ contains the point $e_i$ if and only if $0=f_j(e_i)=a_{ij}$ for all $j$. As $\mbd=Z_\QQ$, we then conclude that $Y\cap \mbd\neq \emptyset$ if and only if there exists  $i$ with $a_{ij}=0$ for all $j$. This is precisely the statement of (ii).
\end{proof}

\subsubsection{Non-degenerate curves and coarse Hilbert moduli schemes}\label{sec:icocoarse}
We continue our notation. Let $X_f$ be an ico model over $\ZZ$ as in \eqref{def:icomod}, and let $\nu_f$ be as in \eqref{eq:defnuf}. We next go in the proof of Theorem~\ref{thm:parshin} to show the following refinement for $X_f$.

\begin{lemma}\label{lem:xfmoduli}
If $X_f$ is non-degenerate and $\nu=30\nu_f$, then $X_f$ becomes over $\ZZ[1/\nu]$ a coarse Hilbert moduli scheme over $\ZZ[1/\nu]$ of finite level with empty branch locus.
\end{lemma}
We first introduce some terminology which we shall use in the proofs below. It was shown in \cite[Prop 11.5]{vkkr:chms} that $M_{\mathbb S}$ is a coarse Hilbert moduli scheme over $\mathbb S=\ZZ[1/30]$ of the moduli problem $\mathcal P$ on $\mathcal M_{\mathbb S}$ of symplectic level 2-structures, where we write
\begin{equation}\label{eq:mbarmoduli}
M_R=(\mb_\ZZ\setminus Z)_R, \quad \mb_\ZZ\subset \mathbb P^4_\ZZ:\sigma_2=0=\sigma_4, \quad Z=\cup P_i
\end{equation}
for any subring $R\subseteq \QQ$. Here $P_i\subset \mathbb P^4_\ZZ$ denotes the image of the $\ZZ$-point $e_i$ given by the $i$-th projection $\ZZ^5\to \ZZ$ and $\mathcal M=\mathcal M^{D^{-1}}$ denotes the Hilbert moduli stack over $\ZZ$ of Deligne--Pappas~\cite{depa:hilbertmodular} associated to the totally real field $L=\QQ(\sqrt{5})$ as in \cite[$\mathsection$3.1]{vkkr:hms}.

\begin{proof}[Proof of Lemma~\ref{lem:xfmoduli}]
Suppose that $X=X_f$ is non-degenerate. Then Lemma~\ref{lem:relations} identifies $X_\QQ$ with a non-degenerate model $Y_\QQ$ in $\mb$. In the proof of Lemma~\ref{lem:extndmod} and \eqref{eq:hilbanalogue}, we showed that $X_\QQ\cong Y_\QQ$ extends to a coarse Hilbert moduli scheme $X_S\cong Y_S$ over $S$ of finite level with empty branch locus by applying Lemma~\ref{lem:coarsebasechange} with an immersion 
$$
X_S\cong Y_S\hookrightarrow M_S
$$
for some nonempty open subscheme $S\subseteq \bs$. To prove Lemma~\ref{lem:xfmoduli}, it thus suffices to show that one can take here $S=\ZZ[1/\nu]$ with $\nu=30\nu_f$. We identify $X_\bs$ via \eqref{eq:modelident} with $Y\subset \mb_{\bs}$ and then we obtain a closed immersion $X_S\hookrightarrow M_S$ over $S=\ZZ[1/\nu]$ if the intersection $$
Y_S\cap Z_S=\emptyset.
$$
Notice that $Z_\bs=\cup e_i(\bs)$, and write $\sum a_{ij}x_i^{n_j}$ for the diagonal part of the homogeneous $f_j\in\ZZ[x_0,\dotsc,x_4]$ of degree $n_j\geq 1$. Then the definition of $\nu_f$ shows that $Y_S\cap Z_S=\emptyset$ for $S=\ZZ[1/\nu]$ if each point $y$ in $Y\cap e_i(\bs)$ lies over a closed point $p\in \bs$ with 
\begin{equation}\label{eq:intersection}
p\mid a_{ij} \ \textnormal{ for all }j.
\end{equation}
This can be verified by the following computation. As $e_i(\bs)\subseteq D_+(x_i)\subset \mathbb P^4_\bs$, the section $e_i$ defines a ring morphism $\rho_i:R[\tfrac{x_k}{x_i}]\to R=\ZZ[\tfrac{1}{30}]$ and $y=e_i(p)$ lies in $Y\cap D_+(x_i)=V(I_{(x_i)})$ where  $I=(\sigma_2,\sigma_4,f_1,\dotsc,f_m)\subseteq R[x_0,\dotsc,x_4]$.  
It follows that $I_{(x_i)}\subseteq \rho_i^{-1}(p)$ in $R[\tfrac{x_k}{x_i}]$, 
and hence $(0,0,a_{i1},\dotsc,a_{im})=\rho_i(I_{(x_i)})\subseteq (p)$ in $R$. This implies \eqref{eq:intersection} provided that $p\neq 0$. As $X_f$ is non-degenerate, there exists $j$ with $a_{ij}\neq 0$ and thus indeed $p\neq 0$ since $(a_{ij})\subseteq (p)$. This completes the proof of \eqref{eq:intersection}, and thus of Lemma~\ref{lem:xfmoduli}.
\end{proof}

 We are now ready to deduce Corollary~\ref{cor:icomoduliint} from Theorem~\ref{thm:parshin} and \eqref{eq:mbarmoduli}. Let $X$ be a curve over $\QQ$ which is non-degenerate, and let $U\subseteq X$ be open. Suppose that $U$ is of the form $U=U_f$ with $X_f$ a non-degenerate ico model of $X$ over $\ZZ$.
\begin{proof}[Proof of Corollary~\ref{cor:icomoduliint}]
To prove (i) we identify $(X_f)_\QQ$ via  Lemma~\ref{lem:relations}~(i) with a non-degenerate model $Y$ of $X$ in $\mb$ such that $U_Y=U_f=U$. As $M_{\mathbb S}$ is a coarse Hilbert moduli scheme, Theorem~\ref{thm:parshin} gives that $U$ extends over a dense open $S\subseteq \spec(\ZZ)$ to a coarse Hilbert moduli scheme $U_S$ of finite level with empty branch locus. This proves (i).

To show (ii) we go into the proof of Theorem~\ref{thm:parshin}. Therein we constructed immersions $U_S\hookrightarrow Y_S\hookrightarrow M_S$ and we deduced from Lemma~\ref{lem:coarsebasechange} that $U_S$ is a coarse Hilbert moduli scheme over $S$ of an arithmetic moduli problem $\mathcal P'$ on $\mathcal M_S$ such that $\mathcal P'(x)/\Aut(x)$ injects into $\mathcal P(x)/\Aut(x)$ for all $x\in \mathcal M_S(T)$ with $T$ any $S$-scheme. Thus for each algebraically closed field $k$ with $\textnormal{char}(k)\in S$, we obtain the moduli interpretation
\begin{equation}\label{eq:moduliint}
\pi^{-1}:U_S(k)\isomto [\mathcal M_{\mathcal P'}(k)]
\end{equation}
with respect to some initial morphism $\pi:\mathcal M_{\mathcal P'}\to U_S$. In other words, the bijection in \eqref{eq:moduliint} identifies any point in $U_S(k)$ with an isomorphism class of $(x,\alpha)$ where $x=(A,\iota,\varphi)$ lies in $\mathcal M_S(k)\cong \mathcal M(k)$ and where $\alpha\in\mathcal P'(x)$ is a restricted\footnote{Up to an automorphism we can always view $\mathcal P'(x)$ inside $\mathcal P(x)$ via $\mathcal P'(x)/\Aut(x)\hookrightarrow \mathcal P(x)/\Aut(x)$.} symplectic level 2-structure of $x$. As the ring of integers $\OL_L=\ZZ[\tfrac{1+\sqrt{5}}{2}]$ has class number one, this moduli interpretation identifies with the moduli interpretation for the points in $U_S(k)$ claimed in (ii). 

To prove (iii) we assume that $X=(X_f)_\QQ$. Then we can take $U=X$, and Lemma~\ref{lem:xfmoduli} gives (i) with $S=\ZZ[1/\nu]$ for $\nu=30\nu_f$. Moreover, on combining the above proof of (ii) with the arguments of Lemma~\ref{lem:xfmoduli}, we see that (ii) also holds with $S=\ZZ[1/\nu]$ in the case $X=(X_f)_\QQ$. This proves (iii), and thus completes the proof of Corollary~\ref{cor:icomoduliint}.
\end{proof}

\subsubsection{Geometric properties of general ico models}

In this section we prove Proposition~\ref{prop:generalico}. Let $n\in \NN$ and recall that $s_1,\dotsc,s_r$ are monomials in $\QQ[x_0,\dotsc,x_4]$ of degree $n$ which form a basis of the $\QQ$-module $A_n$, where $A=\QQ[x_0,\dotsc,x_4]/(\sigma_2,\sigma_4)$  and $s_i=x_{i-1}^n$ for $i\leq 5$. We write $\mathbb A^r=\mathbb A^r_\QQ$. For any $v\in \mathbb A^r(\QQ)$ let $X_v\subset \mathbb P^4_\QQ$ be the closed subscheme given by $X_f$ for
 $f=\sum v_is_i$.

\begin{proof}[Proof of Proposition~\ref{prop:generalico}]
We first prove (i). Write $U^{\textnormal{nd}}=\mathbb A^r\setminus V$ with $V=\cup_{i=1}^5 V(z_i)$ and let $v\in U^{\textnormal{nd}}(\QQ)$. Then each diagonal coefficient of $f=\sum v_is_i$ is nonzero since $s_i=x_{i-1}^n$ for $i\leq 5$. Thus $X_v\cong X_f$ is a non-degenerate ico model over $\QQ$. Here $X_v$ is automatically a curve over $\QQ$ by the following argument: Working inside the integral surface $X=\textnormal{Proj}(A),$ we obtain that $X_f\subsetneq X$ since the diagonal coefficients of $f$ are nonzero. Hence the scheme $X_f$ has $\dim(X_f)=1$ and thus its irreducible components have co-dimension 1 in $X$. This implies that $X_v\cong X_f$ is indeed a curve over $\QQ$. Thus $U^{\textnormal{nd}}$ has the desired property in (i). 

We now deduce (ii) from Bertini. Write $k=\QQ$ and $\mathcal L=\mathcal O_{X}(n)$. Identifying a polynomial in $k[x_0,\dotsc,x_4]_n$ with a global section of $\mathcal O_{\mathbb P^4_k}(n)$ defines an injective morphism of $k$-vector spaces $\iota:A_n\to H^0(X,\mathcal L)$.  For any $v\in \mathbb A^r(k)$ let $Z(s)$ be the zero scheme of the global section $s=\iota(f)$ of $\mathcal L$ for $f=\sum v_i s_i$. As  $X$ is normal and $k$ has characteristic 0,  we may and do apply Bertini's theorem (\cite[3.4.9]{bertini:book} or the proof of \cite[0FD6]{sp}). This gives an open dense $U^{\textnormal{sm}}\subseteq U^{\textnormal{nd}}$ such that for each $v\in U^{\textnormal{sm}}(k)$ the curve $$X_v\cong \textnormal{Proj}(A/f)\cong Z(s)$$ is normal and thus smooth over $k$. Let us study connectedness. As $v$ lies in $U^{\textnormal{sm}}\subseteq U^{\textnormal{nd}}$, all diagonal coefficients of $f$ are nonzero. Then we compute (via the injective $\iota$) that the global section $s_{\bar{k}}$ of $\mathcal L_{\bar{k}}$ is nonzero, and thus regular since $X_{\bar{k}}$ is integral. Hence, as $X_{\bar{k}}$ is normal, the zero scheme $Z(s_{\bar{k}})\cong(X_v)_{\bar{k}}$ is connected by  \cite[0FD9]{sp}. This proves that $X_v$ is geometrically connected. Next, we compute the arithmetic genus $$g_a(X_v)=\sum_{m=1}^3(-1)^{m+1}\sum_{1\leq i_1<\dotsc<i_m\leq 3}\varphi(-d_{i_1}-\dotsc -d_{i_m})=(2n+1)^2$$
via the formula (see e.g. \cite[Cor 2]{arse:genus}) obtained  by computing the Hilbert polynomial of the complete intersection $(X_v)_{\bar{k}}$. Here $\varphi(z)=\tfrac{1}{24}(z+1)(z+2)(z+3)(z+4)$, while $d_1=2$, $d_2=4$ and $d_3=n$ are the degrees of $\sigma_2$, $\sigma_4$ and $f$ respectively. The geometric genus $g$ of $X_v$ satisfies $g=g_a(X_v)$, since the curve $X_v$ is smooth. We conclude that the non-degenerate ico model $X_v$ is a smooth, projective and geometrically connected curve over $k$ of  genus $g=g_a(X_v)=(2n+1)^2$. Thus $U^{\textnormal{sm}}$ has the desired property in (ii). 

To prove (iii) we freely use that $X$ is a geometrically integral normal projective variety over  $k$. Write $U$ for $U^{\textnormal{sm}}$ and recall that $\mathcal C$ denotes the set of smooth, projective and geometrically connected curves inside $\mathbb P^4_k$. Then $v\mapsto X_v$ defines a map $U(k)\to \mathcal C$ by (ii). Each $\lambda\in k^\times$ is invertible in $A$ and hence $X_v=X_{\lambda v'}$. Thus we obtain an induced map $$U(k)/k^\times\to \mathcal C.$$ To show that this map is injective, we assume that $v,v'$ in $U(k)$ satisfy $X_v=X_{v'}$. Define $s=\iota(f)$ and $s'=\iota(f')$ where $f=\sum v_is_i$ and $f'=\sum v_i's_i$. As $v,v'$ lie in $U\subseteq U^{\textnormal{nd}}$, the global sections $s,s'$ of $\mathcal L$ are nonzero and $X_v=X_{v'}$ is irreducible. Then we deduce that the supports of the Weil divisors $\textnormal{div}(s)$ and $\textnormal{div}(s')$ are equal and irreducible. It follows that $\lambda=s/s'$ lies in $k(X)$ with $\textnormal{supp}(\lambda)_\infty=\emptyset$, and hence $\lambda\in \mathcal O_X(X)$. As $s,s'$ are nonzero and $\mathcal O_X(X)=k$, we deduce that $\lambda\in k^{\times}$. Moreover, it follows from $\lambda=s/s'\in k(X)=\mathcal O_{X,\eta}$ that $s_\eta=(\lambda s')_\eta$ inside $\mathcal L_\eta$ for $\eta$ the generic point of $X$ and thus $s=\lambda s'$ inside $H^0(X,\mathcal L)$. As $\iota$ is injective and $k$-linear, we  conclude that $v=\lambda v'$. Hence the displayed map is indeed injective. This completes the proof of (iii) and thus of Proposition~\ref{prop:generalico}. 
\end{proof}

\subsection{Proof of effective Par{\v{s}}in/Mordell  for non-degenerate curves}

In this section we deduce Corollaries~\ref{cor:effparico}, \ref{cor:effmordell} and \ref{cor:effmordellxf}. We continue our notation. Let $X$ be a curve over $\QQ$ which is non-degenerate and let $U\subseteq X$ be open. Suppose that $U$ is of the form $U=U_f$ with $X_f$  a non-degenerate ico model of $X$ over $\ZZ$.  
\begin{proof}[Proof of Corollary~\ref{cor:effparico}~(i)]
We recall from $\mathsection$\ref{sec:proofgeom} that $M_{\bs}$ is a tame coarse Hilbert moduli scheme over $\bs=\ZZ[1/30]$ of finite level, which parametrizes abelian surfaces. As $X_f$ is non-degenerate, Lemma~\ref{lem:relations} identifies $(X_f)_\QQ$ with a non-degenerate model $Y$ in $\mb$ such that $U_f=U_Y$. The integral degeneration $\nu_Y$ divides $\nu=30\nu_f$ for $\nu_f$ as in \eqref{eq:defnuf}, since in the proof of Lemma~\ref{lem:xfmoduli} we constructed a suitable immersion over $S=\ZZ[1/\nu]$. Then an application of Theorem~\ref{thm:parshin} with $U=U_f=U_Y$ proves Corollary~\ref{cor:effparico}~(i).
\end{proof}  
In the following proofs we work over $\QQ$ and we write again $\mathbb P^n$ for $\mathbb P^n_\QQ$ where $n\in\NN$.
\begin{proof}[Proof of first part of Corollary~\ref{cor:effparico}~(ii) and of Corollary~\ref{cor:effmordell}]We 
assume that $X$ is integral and that $X$ admits a finite morphism $X\to \mathbb P^n$ for some $n\in\NN$. Let $h$ be the associated Weil height \eqref{def:weilheight}, and let $d_X$ and $h(X)$ be defined as in $\mathsection$\ref{sec:ratpoints} with respect to a normalization $\pi:\tilde{X}\to X$ and an immersion $\tau:U\hookrightarrow X_f\subset\mathbb P^4$ where we write $X_f=(X_f)_\QQ$.

We now proof the first part of Corollary~\ref{cor:effparico}~(ii). The universal property of normalization of curves gives a finite morphism $\varphi:\xt\to \mathbb P^4$ such that  
$$
\begin{tikzcd}
\mathbb{P}^n & \tilde{U} \arrow[l, "\psi"'] \arrow[r, "\varphi"] \arrow[d, "\pi"] & \mathbb{P}^4 \\
             & U \arrow[lu] \arrow[ru, "\tau"']                                     &             
\end{tikzcd}
$$
commutes. 
Here $\tilde{U}=\pi^{-1}(U)$ is a normalization of $U$, and the morphism $\psi$ is the composition $\xt\to^\pi X\to\mathbb P^n$. Now, we let $x\in U(\qb)$ and we take $y\in \tilde{U}(\qb)$ with $\pi(y)=x$. As the diagram commutes, an application of R\'emond's result \cite[Prop 5.2]{remond:construction} with the finite morphisms $\varphi:\xt\to \mathbb P^4$ and $\psi:\xt\to \mathbb P^n$ gives \begin{equation}\label{eq:weilheightcomp}
h(x)\leq h_\psi(y)\leq d_X h_\varphi(y)+h(X)=d_Xh_\tau(x)+h(X).
\end{equation}
Here $h_f$ denotes the pullback along a morphism $f:Z\to \mathbb P^m$ of the $l^2$-normalized logarithmic Weil height (\cite[$\mathsection$2]{remond:rational}) on $\mathbb P^m$, where $m\in\NN$ and $Z$ is any variety over $\QQ$.  As  $X_f$ is non-degenerate, Lemma~\ref{lem:relations} identifies $X_f$ with $Y\subset M=\mb\setminus \mbd$ such that $\tau:U\to \mathbb P^4$ factors through $Y\subset M\subset \mb\subset \mathbb P^4$. Then, after comparing the $l^2$-norm as done in \eqref{eq:l2normbound} below, an application of \cite[Prop 11.13]{vkkr:chms} with $z=\tau(x)\in M(\qb)$ leads to
\begin{equation}\label{eq:faltheightcomp}
h_\tau(x)\leq 2h_{\phi'}(z)+8^8\log(h_{\phi'}(z)+8).
\end{equation}
Here $h_{\phi'}:M(\qb)\to \RR$ is the pullback of $h_F$ under the forgetful map $\phi':M(\qb)\to \underline{A}_2(\qb)$ of $M$. Recall from the proof of Theorem~\ref{thm:parshin} that the coarse moduli scheme structure on $U$ comes from Lemma~\ref{lem:coarsebasechange} applied with the immersions $U\hookrightarrow X_f\isomto Y\subset \mb$. Thus Lemma~\ref{lem:coarsebasechange} shows that the forgetful map $\phi:U(\qb)\to \underline{A}_2(\qb)$ factors as $\phi=\phi'\tau$ and hence
\begin{equation}\label{eq:hphicomp}
h_{\phi'}(z)=h_\phi(x).
\end{equation}
This together with \eqref{eq:weilheightcomp} and \eqref{eq:faltheightcomp} implies that $h\leq d_X\bar{h}_\phi+h(X)$ on $U(\qb)$ as claimed in the first part of Corollary~\ref{cor:effparico}~(ii).  To complete the proof, we now include the computations which we used to deduce \eqref{eq:faltheightcomp} from \cite{vkkr:chms}.  The $l^2$-normalized Weil height $h_\tau(x)$ satisfies 
\begin{equation}\label{eq:l2normbound}
h_\tau(x)\leq h_{\mathbb P^4}(z)+\tfrac{1}{2}\log 5,
\end{equation}
and the arguments surrounding \cite[(11.30)]{vkkr:chms} show that the upper bound \cite[Prop 11.13]{vkkr:chms} for $h_{\mathbb P^4}(z)$ also holds for $h_{\mathbb P^4}(z)+\tfrac{1}{2}\log 5$. Thus \eqref{eq:faltheightcomp} holds as claimed. 

We next prove Corollary~\ref{cor:effmordell}. On using precisely the same arguments as in the proof of Corollary~\ref{cor:effparico}~(i), we see that $X_f$ identifies with a non-degenerate model $Y$ in $\mb$ such that $U_f=U_Y$ and such that $\nu_Y$ divides $30\nu_f$. Then an application of Corollary~\ref{cor:hms} with $U=U_f=U_Y$ gives for each $x\in U(\QQ)$ a bound for $h_\phi(x)$, which together with the first part of Corollary~\ref{cor:effparico}~(ii) leads to the bound for $h(x)$ as claimed in Corollary~\ref{cor:effmordell}.\end{proof}

\begin{proof}[Proof of second part of Corollary~\ref{cor:effparico}~(ii) and of Corollary~\ref{cor:effmordellxf}]
Assume that $X=(X_f)_\QQ$ and that $h=h_{\mathbb P^4}$. Then we have $X=U_f=U$, and \eqref{eq:faltheightcomp} holds with $h$ in place of $h_\tau$ by \cite[Prop 11.13]{vkkr:chms}. Thus \eqref{eq:hphicomp} gives $h\leq \bar{h}_\phi$ on $X(\qb)$ as claimed in the second part of Corollary~\ref{cor:effparico}~(ii). Finally, on using precisely the same arguments as in the proof of Corollary~\ref{cor:effmordell}, we deduce from Corollary~\ref{cor:hms} and the second part of Corollary~\ref{cor:effparico}~(ii) that any $x\in X(\QQ)$ satisfies $h(x)\leq \tfrac{c}{10}\nu_f^{24}$. This proves Corollary~\ref{cor:effmordellxf}. \end{proof}

\section{Non-degenerate plane curves}\label{sec:p2}

In this section we illustrate for non-degenerate plane curves our results and constructions of the previous sections. We work over $\QQ$. Write $\mathbb P^n=\mathbb P_\QQ^n$ and $\mathbb A^n=\mathbb A^n_\QQ$ for $n\in\NN$. Let $F\in \ZZ[x,y,z]$ be homogeneous of degree $d\geq 1$ and define $X=V_+(F)\subset \mathbb P^2$. 
\paragraph{Geometric criterion $(\tau)$.} To illustrate some aspects of the non-degenerate condition, we introduce a simple non-degeneracy criterion for $X\subset \mathbb P^2$ which allows to explicitly produce large classes of non-degenerate plane curves. We use the explicit morphism
\begin{equation}\label{eq:taudef}
\tau:\mathbb P^2\setminus T_\tau\to \mathbb P^4 
\end{equation}
given by the five homogeneous polynomials $\tau_i\in\ZZ[x,y,z]$ of degree 12  defined in \eqref{def:taui} where $T_\tau=\cap V_+(\tau_i)\subset \mathbb P^2$ is finite.  Permuting the coordinates of $e_1=(1,0,\dotsc,0)$ defines five closed points $e_i\in \mathbb P^4$. Now, we consider the following geometric criterion:
$$
(\tau) \ \, \textnormal{The closure }\overline{\tau(X\setminus T_\tau)} \textnormal{ contains no } e_i.
$$
If $(\tau)$ holds then the curve $X$ is non-degenerate by Theorem~\ref{thm:p2}.
However $(\tau)$ is much stronger since many non-degenerate curves $X\subset\mathbb P^2$ do not satisfy $(\tau)$. The morphism $\tau$ is birational onto its image and there exist many `equivalent' maps $\tau'$ giving different non-degenerate criteria $(\tau')$ which would also be suitable for our purpose, see $\mathsection$\ref{sec:taurho}.

\paragraph{Effective Mordell.} Let $h$ be the usual logarithmic Weil height on $\mathbb P^2$, and define $|s|=\max_\alpha |s_\alpha|$ for any polynomial $s$ whose coefficients $s_\alpha$ all lie in $\ZZ$. An application of Corollary~\ref{cor:effmordell} with the irreducible components of  $X=V_+(F)\subset\mathbb P^2$ leads to:

\begin{corollary}\label{cor:effmordellp2}
If $X$ satisfies $(\tau)$, then any $x\in X(\QQ)$ has height $h(x)\leq \mu |F|^{\kappa}$.
\end{corollary}
Here one can take for example $\kappa=8^8d^2$ and $\mu=8^{\kappa^{2}d}$ where $d=\deg(F)$. Corollary~\ref{cor:effmordellp2} applies in particular to the explicit higher dimensional families of plane curves constructed in $\mathsection$\ref{sec:families}. For these families we also obtain a more uniform height bound in Corollary~\ref{cor:effmordellpullbackp2}.  

\paragraph{Illustration.} To illustrate how one can control via Corollary~\ref{cor:effmordell} the height $h$ on $X(\QQ)$, we let $C_\tau\subset\mathbb P^2$ be the curve defined in \eqref{def:ctau} and we proceed in three steps given in (i)\,-\,(iii).

\begin{theorem}\label{thm:p2}
Suppose that $X=V_+(F)\subset \mathbb P^2$ satisfies $(\tau)$. Then $X$ is non-degenerate, $U=X\setminus C_\tau$ is open dense in $X$, and the following statements hold.
\begin{itemize}
\item[(i)] There exists a non-degenerate ico model $X_f$ of $X$ over $\ZZ$ such that $U\subseteq U_f$ and such that any $x\in U(\QQ)$ satisfies $h(x)\leq c\nu_f^{24}$ where $\nu_f$ is as in \eqref{eq:defnuf}.
\item[(ii)] Any $x\in (X\setminus U)(\QQ)$ has height $h(x)=0$.
\item[(iii)] If $X$ is integral and $U(\QQ)\neq \emptyset$, then $(X_f)_\QQ$ is contained in a non-degenerate ico model $X_{\tilde{f}}$ over $\QQ$ for $\tilde{f}\in \ZZ[x_0,\dotsc,x_4]$ with $\deg(\tilde{f})\leq 128d$ and  $|\tilde{f}|\leq u|F|^{v}$.
\end{itemize}
\end{theorem}
Here $u=3^{(86d)^5}$ and $v=(258d)^2$ where $d=\deg(F)$.   Combining (i)\,-\,(iii) leads to the above corollary. We now discuss in more detail each of the three steps (i)\,-\,(iii).

\paragraph{Step (i).} Corollary~\ref{cor:effmordell} gives (i) with a bound for $h(x)$ which depends in addition on the normalized degree $d_X$ and the height $h(X)$. To remove here the dependency on $d_X$ and $h(X)$, we go into the proof of Corollary~\ref{cor:effmordell} and we replace therein the comparison of  Weil heights by a trick which exploits that the birational map $\tau$ has an explicit inverse $\rho$.  

\paragraph{Step (ii).} This step is usually simpler than step (i) since  
$\dim(X\setminus U)<\dim(U)$. For example, in the case of curves we can always use arithmetic B\'ezout to directly control the height $h$ on the finitely many points in the intersection of two curves $$X\setminus U=X\cap C_\tau.$$
The resulting bound would depend on $F$. To obtain our uniform bound in (ii), we avoid B\'ezout and instead we compute that $X\setminus U$ is contained in the finite uniform set $T_\tau$ given in \eqref{eq:criterion}. This computation crucially exploits the shape of $\tau$ and that $X$ satisfies $(\tau)$. 

\paragraph{Step (iii).}The idea of the proof is as follows: As $X_f=\overline{\tau(U)}$ is contained in $X_{\rho^*F}$, we can try to control $f$ in terms of $\rho^*F$ and then in terms of $F$. Unfortunately the ico model $X_{\rho^*F}$ is always degenerate and the image $X_f$ is usually not a complete intersection. Moreover $f$ might not be controlled in terms of $\rho^*F$. To circumvent these issues, we constructed a possibly larger $X_{\tilde{f}}\supseteq X_f$ which is still controlled and non-degenerate: 
\begin{equation}\label{eq:castelnuovooutline}
\tilde{f}=\sum \bigl(x_i^{r-d_i}f_i\bigl)^2, \quad r=\textnormal{reg}(Y_{\bar{\QQ}})\leq 8\deg(\rho^*F), \quad Y=X_f\subset V_+(f_i).
\end{equation}
Here $f_i\in \ZZ[x_0,\dotsc,x_4]$ is homogenous of degree $d_i\leq r$ with $e_i\in D_+(f_i)$, and $r$ is the Castelnuovo--Mumford regularity. We bounded $r$ via Giaimo--Gruson--Lazarsfeld--Peskine~\cite{grlape:castelnuovo,giaimo:castelnuovo}, using that $Y$ is contained in $X_{\rho^*F}$ and that $Y_{\bar{\QQ}}$ is connected by the assumptions in (iii). The $f_i$ are constructed in Lemma~\ref{lem:reg}. Moreover, as $Y\subset X_{\rho^*F}$,  we can modify each $f_i$ to control $|f_i|$ in terms of $|\rho^*F|$. Here we use height properties and Remond's~\cite[Prop 2.2]{remond:rational} which relies inter alia on Zhang's~\cite[Thm 5.2]{zhang:positivejams}.

\subsection{Explicit families of curves satisfying $(\tau)$}\label{sec:families}

We continue our notation. There exist large classes of explicit plane curves  satisfying $(\tau)$. To demonstrate this, we construct in this section high dimensional families/moduli spaces of such curves. Let $n\in \NN$. Recall from \eqref{eq:dimr} that $r=5$ if $n=1$ and that for $n\geq 2$ it holds $$r=4n^2-4n+6.$$ Let $\mathcal F_n$ be the family of homogeneous $f\in \ZZ[x_0,\dotsc,x_4]$ of degree $n$ with all diagonal coefficients nonzero,  and let $\mathcal C_n$ be the family of curves $X=V_+(F)\subset\mathbb P^2$ with $X=V_+(\tau^*f)$ as subsets of $\mathbb P^2$ for some $f\in\mathcal F_n$ where $\tau^*f=f(\tau_0,\dotsc,\tau_4)$. Let $v\in\mathbb A^r(\QQ)$ and define
$$X_v=V_+(F_v)\subset \mathbb P^2, \quad F_v=\sum v_i F_i, \quad F_i=\tau^*s_i.$$
Here $s_1,\dotsc,s_r$ are monomials in $\QQ[x_0,\dotsc,x_4]$ of degree $n$ which form a basis of the $\QQ$-module $A_n$, where $A=\QQ[x_0,\dotsc,x_4]/(\sigma_2,\sigma_4)$  and $s_i=x_{i-1}^n$ for $i\leq 5$ as in $\mathsection$\ref{sec:geomicoresults}. One can  compute such $s_i$ and thus the $F_i$ for each $n\in \NN$. We obtain the following result.

\begin{proposition}\label{prop:p2family}
Let $n\in \NN$ and write $M_n=\mathbb A^r\setminus \bigl(\cup_{i=1}^5V(z_i)\bigl)$. 
\begin{itemize}
\item[(i)]All curves $X\in\mathcal C_n$ satisfy $(\tau)$.
\item[(ii)] Suppose that $v\in M_n(\QQ)$. Then $X_v\subset \mathbb P^2$ is a plane curve of degree $12n$ and $X_v$ lies in $\mathcal C_n$. Moreover $v\mapsto X_v$ defines an injective map $M_n(\QQ)/\QQ^\times\hookrightarrow \mathcal C_n$.
\item[(iii)]There exists an open dense $U\subset M_n$ such that $v\mapsto (X_v)_{\textnormal{red}}$ defines an injective map $U(\QQ)/\QQ^\times\hookrightarrow \mathcal C_n$ with $(X_v)_{\textnormal{red}}$ geometrically integral of geometric genus $(2n+1)^2$. 
\end{itemize}
\end{proposition}
The proof shows in addition that one can take in (iii) the same open dense $U\subset\mathbb A^{r}$ as in Proposition~\ref{prop:generalico}. Let $g\in \NN$. If $g\geq 2$ is an odd square, then taking $n=\tfrac{\sqrt{g}-1}{2}$ in Proposition~\ref{prop:p2family} gives a moduli space $U$ of dimension $r\sim g$ by \eqref{eq:dimr} which parametrizes geometrically integral curves $X\subset\mathbb P^2$ of geometric genus $g$ satisfying $(\tau)$.

\paragraph{Height bounds.} As each curve $X$ in $\cup_n\,\mathcal C_n$ satisfies $(\tau)$ by Proposition~\ref{prop:p2family}, the usual logarithmic Weil height $h$  is explicitly bounded on $X(\QQ)$ by Corollary~\ref{cor:effmordellp2}. On combining Theorem~\ref{thm:p2} with Proposition~\ref{prop:p2family}, we obtain the following more uniform bound.
\begin{corollary}\label{cor:effmordellpullbackp2}
If $X$ is a curve in $\cup_{n}\,\mathcal C_n$, then all $x\in X(\QQ)$ satisfy $h(x)\leq c\nu_X^{24}$.
\end{corollary}
Here $c=10^{10^{12}}$ and  $\nu_X=\min \nu_f$ with the minimum taken over all $f$ in $\cup_{n}\,\mathcal F_n$ such that $X=V_+(\tau^*f)$ as subsets of $\mathbb P^2$ where $\nu_f$ is defined in \eqref{eq:defnuf}.

\subsection{The birational morphisms $\tau$ and $\rho$}\label{sec:taurho}

We continue our notation. In this section we use classical constructions of Clebsch~\cite{clebsch:clebschsurface} and Klein~\cite{klein:cubicsurfaces} to explicitly construct a birational equivalence between $\mathbb P^2$ and the surface $\mb\subset \mathbb P^4$  in \eqref{def:mmbar}. We also prove some basic geometric properties of this equivalence.

\paragraph{Construction of $\tau$ and $\rho$.} We first define homogeneous polynomials $\tau_i$ in $\ZZ[x,y,z]$ of degree 12 as follows: We put $\tau_i=-(\prod_{j\neq i}t_j)(\sum t_j)$ for $i\in \{0,1,2,3\}$ and $\tau_4=\prod t_j$, where
\begin{align*}
t_0 &=(y-z)(xy+xz-z^2),  &t_1=&xz^2+yz^2-x^2y-z^3,  \\
t_2 &=x(z^2-y^2-xz),  &t_3=&z(yz-xz+x^2-y^2).
\end{align*}
The polynomials $\tau_i$ define a morphism $\tau:\mathbb P^2\setminus T_\tau\to\mathbb P^4$ where $T_\tau=\cap V_+(\tau_i)$ is finite. 
We compute in Macaulay2 that the scheme theoretic image of $\tau$ is $\mb$ and thus we obtain 
\begin{equation}\label{def:taui}
\tau:\mathbb P^2\setminus T_\tau\to \mb.
\end{equation}
To construct an `inverse' morphism of $\tau$, we define homogeneous polynomials $\rho_i$ in $\ZZ[x_0,\dotsc,x_4]$ of degree 8 as follows: We write $r_i=\prod_{j\neq i}x_j$ and then we put
\begin{equation}\label{def:rhoi}
\rho_0=-(r_1+r_3)(r_0+r_1+r_2),\quad  \rho_1=r_0(r_0+r_1+r_2+r_3), \quad \rho_2=r_0(r_0+r_2). 
\end{equation}
The polynomials $\rho_i$ define a morphism $\rho: \mb\setminus T_\rho\to \mathbb P^2$ where $T_\rho$ is a curve given by the intersection of $\mb$ with $\cap V_+(\rho_i)$. Then $\tau$ and $\rho$ define a birational equivalence between $\mathbb P^2$ and $\mb$. We shall need controlled open subsets over which $\tau$ and $\sigma$ are isomorphisms. Let
\begin{equation}\label{def:ctau}
C_\tau=V_+(\lambda)\supset T_\tau, \quad U_\tau=\mathbb P^2\setminus C_\tau \quad \textnormal{and} \quad U_\rho=\tau(U_\tau)\subset \mb,
\end{equation}
where $\lambda=\rho_0(\tau_0,\dotsc,\tau_4)/x$ is a non-constant homogeneous polynomial in $\ZZ[x,y,z]$. The following lemma gives that $\tau$ and $\rho$ are indeed isomorphisms over $U_\tau$ and $U_\rho$.

\begin{lemma}\label{lem:isotaurho}
It holds $\tau:U_\tau\isomto U_\rho$ with $U_\rho\subseteq\mb\setminus T_\rho$  and $\rho:U_\rho\isomto U_\tau$ such that $\tau=\rho^{-1}$.
\end{lemma}
\begin{proof}
We first compute $\rho\tau$ and $\tau(U_\tau)$. The identities $(*)$ imply that $\tau:U_\tau\to \mb$ factors through $\mb\setminus T_\rho$ and that $\rho\tau:U_\tau\to \mathbb P^2$ is the inclusion $U_\tau\subset \mathbb P^2$. It follows that $\tau:U_\tau\to \mb\setminus T_\rho$ factors through $Z=\bigl(\mb\setminus T_\rho\bigl)\setminus \rho^{-1}(C_\tau)$ and that $U_\tau\to^\tau Z\to^{\rho} U_\tau$ is the identity on $U_\tau$. Thus $\tau:U_\tau\to Z$ is a closed immersion, which is surjective since $\mb$ is integral and $\dim(U_\tau)=2=\dim(Z)$. Hence $Z=\tau(U_\tau)=U_\rho$ and  $\tau:U_\tau\isomto U_\rho$ is an isomorphism. Then $\rho:U_\rho\isomto U_\tau$ is an isomorphism  since $\rho=\tau^{-1}$. This proves Lemma~\ref{lem:isotaurho}. 

We used the identities $(*)$: $\rho_0(\tau)=\lambda x$, $\rho_1(\tau)=\lambda y$, $\rho_2(\tau)=\lambda z$. Here $\tau=(\tau_0,\dotsc,\tau_4)$, and $\lambda=u^2v$ for $u=(\prod r_i)(\sum r_i)$ and $v\in \ZZ[x,y,z]$ homogeneous of degree 5. 
\end{proof}

The polynomials $t_i$ come from Polo-Blanco--Top~\cite{poto:cubicsurfacesalgo} who used classical constructions to obtain an algorithm computing explicit parametrizations (of small degree) for smooth cubic surfaces. There are many other polynomials which can be used to construct explicit birational equivalences between $\mathbb P^2$ and $\mb$, and these equivalences would  also be suitable for our purpose of explicitly studying rational points on plane curves via $\mb$.

\paragraph{Basic properties.}As a preparation for the proofs below, we now collect some basic geometric properties of $\rho$ and $\tau$ which all can be proved by direct computations.

\begin{lemma}\label{lem:taurho}
Let $F\in \QQ[x,y,z]$ be a non-constant homogeneous polynomial. Consider the curve $X=V_+(F)\subset\mathbb P^2$ and define $X_f\subset \mathbb P^4$ as in \eqref{def:icomod} with $f=F(\rho_0,\rho_1,\rho_2)$. 
\begin{itemize}
\item[(i)] We have $\tau(C_\tau\setminus T_\tau)\subseteq \{e_i\}$, and it holds $\rho^{-1}(X)=(\mb\setminus T_\rho)\cap X_f$.
\item[(ii)]If $X$ is integral and satisfies $(\tau)$, then $X\setminus C_\tau\hookrightarrow U_\tau\to^\tau \mb$ has scheme theoretic image $Y=\overline{\tau(X\setminus T_\tau)}$ which is an integral curve contained in the curve $X_f$.
\end{itemize}
\end{lemma}
\begin{proof}
We first prove (i). Recall that $C_\tau=V_+(\lambda)$ for the explicit $\lambda=\rho_0(\tau_0,\dotsc,\tau_4)/x$. Then we can compute in Macaulay2 that $\tau(C_\tau\setminus T_\tau)\subseteq \{e_i\}$. As $X_f=\mb\cap V_+(f)$ in $\mathbb P^4$, a direct computation on $\qb$-points leads to $\rho^{-1}(X)=(\mb\setminus T_\tau)\cap X_f$ as claimed.

To show (ii) we assume that $X$ is integral and satisfies $(\tau)$. Then $X\setminus C_\tau$ is nonempty by (i), and thus the scheme theoretic image $Y$ is integral. Lemma~\ref{lem:isotaurho} gives the isomorphism $\tau:U_\tau\isomto U_\rho$ with inverse $\rho:U_\rho\isomto U_\tau$. This implies that $\dim(Y)=\dim(X)=1$ and that $\tau(X\setminus C_\tau)=\rho^{-1}(X\setminus C_\tau) \subseteq X_f$ by (i). Hence $Y=\overline{\tau(X\setminus C_\tau)}$ is a curve contained in $X_f$, and thus $\dim(X_f)=1$ since $X_f\neq \mb$ by (i). Then each irreducible component of $X_f$ has co-dimension 1 in the integral $\mb$, and thus $X_f$ is a curve. As $X\setminus C_\tau$ and $X\setminus T_\tau$ have the same generic point, the curve $Y$  is of the form $Y=\overline{\tau(X\setminus T_\tau)}$ as desired.\end{proof}

\subsection{Proof of Theorem~\ref{thm:p2}~(i) and (ii)}

We continue our notation. Let $C_\tau\subset \mathbb P^2$ be the curve defined in \eqref{def:ctau}, and let $F\in \ZZ[x,y,z]$ be homogeneous of degree $d\geq 1$ such that $X=V_+(F)$ satisfies $(\tau)$.  We now go into the proof of Corollary~\ref{cor:effmordell} to complete step (i) of Theorem~\ref{thm:p2}.

\begin{proof}[Proof of Theorem~\ref{thm:p2}~(i)]
The idea is to construct the ico model $X_f$ as `the image' of $X$ under the rational map $\tau$. We now clarify the construction. Notice that $U=X\setminus C_\tau=X\cap U_\tau$ is the fiber product $X\times_{\mathbb P^2}U_\tau$ of the closed immersion $X\hookrightarrow \mathbb P^2$ with the open immersion $U_\tau\hookrightarrow \mathbb P^2$. In particular $U$ is an open subscheme of $X$ which is dense in $X$, since the curve $X$ satisfies $(\tau)$ and Lemma~\ref{lem:taurho} gives $\tau(C_\tau\setminus T_\tau)\subseteq\{e_i\}$. Moreover, the second projection $U=X\times_{\mathbb P^2}U_\tau\hookrightarrow U_\tau$ is a closed immersion and then we consider  
$$U\hookrightarrow U_\tau\isomto U_\rho\hookrightarrow \mb$$
with  $U_\tau\isomto U_\rho$ the isomorphism from Lemma~\ref{lem:isotaurho} given by the restriction of $\tau$ to $U_\tau$. The scheme theoretic image $Y$ of the displayed $U\to\mb$ is a closed subscheme of $\mb$ with 
$$Y=\overline{\tau(U)}\subset \mb.$$
We next show that $Y\cong (X_f)_\QQ$ for some non-degenerate ico model $X_f$ of $X$ over $\ZZ$. The closed immersion $U\hookrightarrow U_\tau\isomto U_\rho$ has scheme-theoretic image $U'=Y\cap U_\rho$ by \cite[01R8]{sp}. Hence $U'$ is an open dense subscheme of $Y$ and $\tau$ defines an isomorphism of $\QQ$-schemes $$U\isomto U'.$$
As $U\subseteq X$ is an open dense subscheme, this shows that $Y$ is a model of $X$ in $\mb$ and thus $U\subseteq U_Y$ for some maximal $U_Y$. Moreover the model $Y$ is non-degenerate, since $X$ satisfies $(\tau)$ and $Y=\overline{\tau(U)}$ is contained in $\overline{\tau(X\setminus T_\tau)}$. Then Lemma~\ref{lem:relations} gives a non-degenerate ico model $X_f$ of $X$ over $\ZZ$ such that $(X_f)_\QQ\isomto Y$ and such that $U_f=U_Y$. 

Now, we apply (the proof of) Corollary~\ref{cor:effmordellxf} with the non-degenerate ico model $X_f$ of $X$ over $\ZZ$. This gives that the usual logarithmic Weil height $h_{\mathbb P^4}$ of any $y\in X_f(\QQ)$ satisfies
\begin{equation}\label{eq:p2mordellbound}
h_{\mathbb P^4}(y)\leq \tfrac{c}{10}\nu_f^{24}.
\end{equation}
For each $x\in U(\QQ)$ consider $y=\tau(x)$ in $ U'(\QQ)\subseteq Y(\QQ)\cong X_f(\QQ)$. Lemma~\ref{lem:isotaurho} gives that $\rho(y)=x$. Thus the explicit construction of $\rho$ in \eqref{def:rhoi} leads to $h(x)\leq 8h_{\mathbb P^4}(y)+\log 6$, and then \eqref{eq:p2mordellbound} gives $h(x)\leq c\nu_f^{24}$ as desired. This completes the proof of Theorem~\ref{thm:p2}~(i). 
\end{proof}
We next complete step (ii) of Theorem~\ref{thm:p2}. This can be done by direct computation.

\begin{proof}[Proof of Theorem~\ref{thm:p2}~(ii)]
To uniformly bound $h$ on $X\setminus U=X\cap C_\tau$, it suffices to show the following: If $F\in \QQ[x,y,z]$ is homogeneous and $X=V_+(F)$ satisfies $(\tau)$, then \begin{equation}\label{eq:criterion}
X\cap C_\tau\subseteq T_\tau= T_\tau(\QQ)\cup \{x_0\}, \quad T_\tau(\QQ)=\{(x_i)\in \mathbb P^2, \ x_i= 0,1\}\setminus (1,1,0)
\end{equation}
with the closed point $x_0$ of $T_\tau$ given by the $\textnormal{Aut}(\qb/\QQ)$-orbit of $(1,1,\tfrac{1+\sqrt{5}}{2})$. The intersection $(X\setminus T_\tau)\cap \tau^{-1}(\{e_i\})$ is empty since $X$ satisfies $(\tau)$, and $C_\tau\setminus T_\tau$ is contained in $\tau^{-1}(\{e_i\})$ by Lemma~\ref{lem:taurho}. Thus $X\cap C_\tau\subseteq T_\tau$, and then we conclude \eqref{eq:criterion} after computing in Macaulay2 all points of the 0-dimensional variety $T_\tau$. This completes step (ii) of Theorem~\ref{thm:p2}. 
\end{proof}

\subsection{Proof of Theorem~\ref{thm:p2}~(iii)}

In this section we complete step (iii) of Theorem~\ref{thm:p2} using the arguments outlined in \eqref{eq:castelnuovooutline}. We continue our notation, and we continue to work over $\QQ$ unless specified otherwise. Recall from \eqref{def:ctau} that the curve $C_\tau$ contains the finite $T_\tau=\cap V_+(\tau_i)$ and that $U_\tau=\mathbb P^2\setminus C_\tau$.

Let $F\in \ZZ[x,y,z]$ be homogeneous of degree $d\geq 1$. Suppose that $X=V_+(F)$ is integral and satisfies $(\tau)$.  Then the scheme theoretic image $Y$ of $U=X\setminus C_\tau\hookrightarrow U_\tau\to^{\tau}\mb$ is an integral curve $Y=\overline{\tau(X\setminus T_\tau)}$ by Lemma~\ref{lem:taurho}. We write $k=\qb$ and we use \eqref{eq:modelident} to identify $$Y=\textnormal{Proj}(A/I) \quad \textnormal{and}\quad Y_{k}=\textnormal{Proj}(A_{k}/I_{k})$$ where $I$ is a homogeneous saturated ideal of $A=\QQ[x_0,\dotsc,x_4]$. 
We will deduce Theorem~\ref{thm:p2}~(iii) from the following result in which $u=3^{(86d)^5}$ and $v=(258d)^2.$

\begin{proposition}\label{prop:p2}
Suppose that $Y$ is geometrically connected. Then there exists a homogeneous $f\in I$ with all coefficients in $\ZZ$ such that $D_+(f)$ contains all $e_i$ and such that $$\deg(f)\leq 128d \quad \textnormal{and} \quad |f|\leq u|F|^{v}.$$
\end{proposition}

We divided the proof of Proposition~\ref{prop:p2} into several lemmas. Let us start with a result from commutative algebra using the Castelnuovo--Mumford regularity $r=\textnormal{reg}(I_{k})$ of $I_k$.

\begin{lemma}\label{lem:reg}
If $e\in \{e_i\}$ then $e\in D_+(f)$ for a homogeneous $f\in I$ with $\textnormal{deg}(f)\leq r$. 
\end{lemma}
\begin{proof}
Let $e\in \{e_i\}$. To construct $f$, we apply to the finitely generated graded $A_{k}$-module $I_{k}$ a standard commutative algebra result for the regularity. This gives that $I_{k}$ is generated as an $A_{k}$-module by finitely many homogeneous $g_j\in I_{k}$ with $\deg(g_j)\leq r.$ As $X$ satisfies $(\tau)$, our point $e$ lies not in $Y=\overline{\tau(X\setminus T_\tau)}$ and thus not in $Y_{k}=\cap V_+(g_j)$. Hence there exists $g\in\{g_j\}$ 
with $e\in D_+(g)$. Viewing the homogeneous $g$ inside $(I_{k})_{n}\cong I_{n}\otimes_\QQ k$ where $n=\deg(g)$,
we obtain homogeneous $f_l\in I_{n}$ and $z_l\in k$ such that $g=\sum z_l f_l$. As $e\in D_+(g)$ there exists $f\in\{f_l\}$ with $e\in D_+(f)$. Moreover $\deg(f)=n=\deg(g)$ is at most $r$ since $g\in\{g_j\}$. Thus the homogeneous $f\in I$ has all the desired properties. \end{proof} 
To control the regularity $r$, we write $\rho^*F=F(\rho_0,\rho_1,\rho_2)$ and we apply to $Y_k$ the explicit Castelnuevo theorem of Gruson--Lazarsfeld--Peskine~\cite{grlape:castelnuovo} which was extended by Giaimo~\cite{giaimo:castelnuovo}  to connected  curves that are not necessarily irreducible. 

\begin{lemma}\label{lem:regdegbound}
If $Y$ is geometrically connected, then 
$r\leq 8\deg(\rho^*F)$.
\end{lemma}
\begin{proof}
We first bound  $r$ in terms of the degree of $Y_{k}\subset \mathbb P^4_{k}$. All irreducible components of $Y_k$ are curves since $Y$ is an integral curve, and $Y_k$ is connected by assumption. Thus $Y_k$ is a connected reduced curve over the algebraically closed field $k=\qb$. Then an application of \cite[Main Thm]{giaimo:castelnuovo} with $Y_k$ gives that the regularity $\textnormal{reg}(Y_{k})$ in \cite{giaimo:castelnuovo} satisfies $$r=\textnormal{reg}(Y_{k})\leq \deg(Y_{k})-l+2$$
for $l$ the dimension of the smallest linear subspace $L\subseteq \mathbb P^4_k$ containing $Y_k$. Here $r$ equals $\textnormal{reg}(Y_{k})$  since the ideal $I$, and thus its base change $I_{k}=I\otimes_\QQ k$, is saturated. We next bound $\deg(Y_{k})$. Let $Z\subset \mathbb P^4_k$ be the reduced closed subscheme determined by $(X_f)_k$ where $f=\rho^*F$.  Lemma~\ref{lem:taurho} implies that $Y_{k}\subseteq Z\subset \mb_k$  and that any irreducible component of $Z$ has dimension one.  Then basic degree properties give  
$$\deg(Y_{k})\leq \deg(Z)\leq \deg(\mb_{k})\deg(\rho^*F).$$
The integral variety $\mb_{k}=\textnormal{Proj}(A_{k}/(\sigma_2,\sigma_4)\bigl)$ is a complete intersection of $\mathbb P^4_{k}$. It follows that $\deg(\mb_{k})=\deg(\sigma_2)\deg(\sigma_4)=8$ and then the displayed bounds prove Lemma~\ref{lem:regdegbound} in the case when $l\geq 2$. Suppose now that $l=1$.  As any linear subspace $L\subseteq \mathbb P^4_k$ is irreducible with $\deg(L)=1$, we then obtain $Y_k=L$ and $\deg(Y_k)=1$. Hence the displayed bound for $r$ gives $r\leq 2\leq 8\deg(\rho^*F)$ as desired. This completes the proof. 
\end{proof}
For any closed $V\subseteq\mathbb P^4_{k}$ with irreducible components $V_i$, define $h(V)=\sum h(V_i)$ for $h(V_i)$ the usual  height (defined via Arakelov theory or Chow form) of the integral closed subscheme $V_i\subseteq\mathbb P^4_k$ determined by $V_i$. Here we normalize $h(V_i)$ as in \cite[$\mathsection$2]{remond:rational}. 

\begin{lemma}\label{lem:heightybound}
It holds
$h(Y_{k})\leq 8(\log |F|+100d)$.
\end{lemma}
\begin{proof}
In the proof of Lemma~\ref{lem:regdegbound} we showed that $Y_k\subset \mathbb P^4_k$ is a reduced curve over $k$, contained in the reduced curve $Z\subset \mathbb P^4_k$ over $k$ determined by $(X_f)_k$ for $f=\rho^*F$. Thus $$h(Y_k)\leq h(Z)$$ and then standard height properties lead to Lemma~\ref{lem:heightybound}. For example one can bound $h(Z)$ as follows. As $h(\mathbb P^4_{k})=\tfrac{77}{24}$ and $\deg(\mb_k)=8$, applications of R\'emond's \cite[Prop 2.3]{remond:rational} with $X=\mb_{k}$, $P_1=\sigma_2$, $P_2=\sigma_4$, $P_3=f$ and with $X=\mathbb P^4_k$, $P_1=\sigma_2$, $P_2=\sigma_4$ give $$\tfrac{1}{8}h(Z)\leq dh(\mb_k)+h(f)+2 \quad \textnormal{and} \quad h(\mb_k)\leq 68.$$
Here $h(f)$ is the usual projective height (\cite[p.21]{bogu:diophantinegeometry}) of the polynomial $f$, and we used that the modified height $h_m$ in \cite{remond:rational} satisfies $h_m(f)\leq h(f)+2$ by \cite[$\mathsection$2]{remond:rational}. Then the claimed bound for $h(Y_{k})$ follows by combining the above displayed estimates with 
$$h(f)\leq \log |\rho^*F|\leq \log \sum_{\iota\in I} |a_\iota\rho_0^{\iota_0}\rho_1^{\iota_1}\rho_2^{\iota_2}|\leq \log |F|+30d,$$
where $a_\iota\in \ZZ$ is the $\iota$-th coefficient of $F$ and where $I$ is the set of all $\iota=(\iota_0,\iota_1,\iota_2)$ with $\iota_l\in \ZZ_{\geq 0}$ and $\iota_0+\iota_1+\iota_2=d=\tfrac{n}{8}$. Here we used \cite[Lem 1.6.11]{bogu:diophantinegeometry}, and we exploited that $|\rho_0|=|\rho_1|=|\rho_2|=1$ holds by \eqref{def:rhoi}. This completes the proof of Lemma~\ref{lem:heightybound}.
\end{proof}

We now prove Proposition~\ref{prop:p2} by combining the above lemmas with R\'emond's result \cite[Prop 2.2]{remond:rational} which relies on Zhang's theorem \cite[Thm 5.2]{zhang:positivejams}.

\begin{proof}[Proof of Proposition~\ref{prop:p2}]
Set $r=\textnormal{reg}(I_{k})$. For each $e_i$, Lemma~\ref{lem:reg} provides a homogeneous $f_i\in I$ of degree $d_i\leq r$ such that $e_i\in D_+(f_i)$. We now apply R\'emond's \cite[Prop 2.2]{remond:rational} with $D=d_i$, $K=\QQ$ and the integral $X=Y$ in $\mathbb P^4$. After scaling with suitable elements in $\QQ^{\times}$, this gives $g_j\in I_{d_i}\cap \ZZ[x_0,\dotsc,x_4]$ with coprime coefficients and $z_j\in \QQ$ such that $$f_i=\sum z_j g_j \quad \textnormal{and}\quad  \sum \log|g_j|\leq B$$  
for $B=r(r+1)h(Y_k)+2\tbinom{r+4}{r}\log(r+1)$. Here $\log |g_j|$ equals the usual projective height $h(g_j)$ of $g_j$, and we used \cite[$\mathsection$2]{remond:rational} which gives that $h(g_j)$ is at most the height of $g_j$ defined in \cite{remond:rational}. As $e_i\in D_+(f_i)$, it follows from $f_i=\sum z_j g_j$ that there exists $s_i\in\{g_j\}$ with $e_i\in D_+(s_i)$. Now, we consider the polynomial $f\in \ZZ[x_0,\dotsc,x_4]$ defined by  $$f=\sum t_i^2, \quad t_i=x_i^{r-d_i}s_i.$$ We observe that $f$ lies in $I$ and that $f$ is homogeneous of degree $2r$. Moreover, as $e_i\in D_+(s_i)$ we obtain that $s_i(e_i)^2>0$ and then we compute $f(e_i)=s_i(e_i)^2+\sum t_l(e_i)^2>0$ with the sum taken over all $l\neq i$. This shows that $D_+(f)$ contains all $e_i$. Furthermore the above displayed results and \cite[Lem 1.6.11]{bogu:diophantinegeometry} imply that  $$\tfrac{1}{5}|f|\leq \max |s_i^2|\leq 2^{10r}\exp(2B).$$ Lemma~\ref{lem:regdegbound} gives $r\leq 64d$ since $Y$ is geometrically connected, and Lemma~\ref{lem:heightybound} provides $h(Y_k)\leq 8(\log |F|+100d).$ Then, on combining everything and on simplifying the resulting estimates, we deduce the bounds for $|f|$ and $\deg(f)=2r$ claimed in Proposition~\ref{prop:p2}. \end{proof}
Finally we deduce Theorem~\ref{thm:p2}~(iii) from Proposition~\ref{prop:p2}. 
\begin{proof}[Proof of Theorem~\ref{thm:p2}~(iii)]By assumption $X=V_+(F)\subset\mathbb P^2$ is integral and $U(\QQ)$ is nonempty. Then the scheme theoretic image $Y=\overline{\tau(U)}$ is an integral variety with a $\QQ$-rational point and thus $Y$ is geometrically connected. Recall that in the proof of Theorem~\ref{thm:p2}~(i) we identified $Y$ with $(X_f)_\QQ$.  Then Proposition~\ref{prop:p2} gives a homogeneous $\tilde{f}\in \ZZ[x_0,\dotsc,x_4]$ with $\deg(\tilde{f})\leq 128d$ and $|\tilde{f}|\leq u|F|^v$ such that $$X_{\tilde{f}}=\mb\cap V_+(\tilde{f})\subset \mathbb P^4$$ has the following properties: The curve $(X_f)_\QQ=Y$ is contained in $X_{\tilde{f}}$, but $X_{\tilde{f}}$ does not intersect $\mbd=\{e_i\}$ and thus $X_{\tilde{f}}$ is a non-degenerate ico model by Lemma~\ref{lem:relations}. Here $X_{\tilde{f}}$ is a curve since $\mb$ is an integral surface with $Y\subseteq X_{\tilde{f}}\subsetneq \mb$ and thus all irreducible components of $X_{\tilde{f}}$ have dimension 1. This completes the proof of Theorem~\ref{thm:p2}~(iii).  
\end{proof}

\subsection{Proof of Proposition~\ref{prop:p2family} and of Corollaries~\ref{cor:effmordellp2} and \ref{cor:effmordellpullbackp2}}

In this section we prove Proposition~\ref{prop:p2family} using the geometric results for general ico models in Proposition~\ref{prop:generalico}. Also, we combine steps (i)\,-\,(iii) of Theorem~\ref{thm:p2} to deduce Corollaries~\ref{cor:effmordellp2} and \ref{cor:effmordellpullbackp2}. We continue our notation and we continue to work over $\QQ$.

\begin{proof}[Proof of Proposition~\ref{prop:p2family}~(i) and of Corollary~\ref{cor:effmordellpullbackp2}]
Let $n\in\NN$ and let $X\in\mathcal C_n$. We first show Proposition~\ref{prop:p2family}~(i). To prove that $X$ satisfies $(\tau)$, we construct a non-degenerate ico model $X_f$ which contains $\overline{\tau(X\setminus T_\tau)}$. As $X\in \mathcal C_n$ there is $f\in \mathcal F_n$ with $X=V_+(\tau^*f)$ as subsets of $\mathbb P^2$. Consider $X_f=\mb\cap V_+(f)$ inside $\mathbb P^4$, write $\rho^*F=F(\rho_0,\rho_1,\rho_2)$  for $F=\tau^*f$, and set $U=X\setminus C_\tau$. On using Lemmas \ref{lem:isotaurho} and \ref{lem:taurho}, we compute inside $\mathbb P^4$ that \begin{equation}\label{eq:nondegcompp2}
\tau(U)=\rho^{-1}(U)=X_{\rho^*F}\cap U_\rho=X_f\cap U_\rho\subseteq X_f.
\end{equation}
Lemma~\ref{lem:taurho} implies that $\overline{\tau(X\setminus T_\tau)}=\overline{\tau(U)}$, and \eqref{eq:nondegcompp2} gives that $\overline{\tau(U)}\subseteq X_f$. Then,  as the non-degenerate ico model $X_f$ contains no $e_i$ by Lemma~\ref{lem:relations}, we deduce that $X$ satisfies $(\tau)$ as desired. Here $X_f$ is a non-degenerate ico model, since all diagonal coefficients of $f\in\mathcal F_n$ are nonzero and since $X_f$ is a curve. Indeed, as $\overline{\tau(U)}\subseteq X_f\subsetneq \mb$, each irreducible component of $X_f$ has codimension 1 in the integral surface $\mb$ and thus has dimension 1. 
We next deduce Corollary~\ref{cor:effmordellpullbackp2}. Let $x\in X(\QQ)$. To bound $h(x)$ we use that $X$ satisfies $(\tau)$ by Proposition~\ref{prop:p2family}~(i). If $x\in (X\setminus U)(\QQ)$ then Theorem~\ref{thm:p2}~(ii) gives that $h(x)=0$. On the other hand, if $x\in U(\QQ)$ then \eqref{eq:nondegcompp2} assures that we can use in \eqref{eq:p2mordellbound} the non-degenerate ico model $X_f$ (viewed here over $\ZZ$) and hence Theorem~\ref{thm:p2}~(i) gives that $h(x)\leq c\nu_f^{24}$. We conclude that all $x\in X(\QQ)$ have height $h(x)$ at most $c\nu_f^{24}$, and thus at most $c\nu^{24}_X$ since the above arguments work for all $f\in \mathcal F_n$ with the property that $X=V_+(\tau^*f)$ as subsets of $\mathbb P^2$. This  completes the proof of Corollary~\ref{cor:effmordellpullbackp2}.
\end{proof}

We next prove (ii) and (iii) of Proposition~\ref{prop:p2family}. Let $n\in\NN$, define $r$ as in \eqref{eq:dimr}, write $M_n=\mathbb A^r\setminus \bigl(\cup_{i=1}^5V(z_i)\bigl)$ and let $s_1,\dotsc,s_r$ be as in \eqref{def:genicomodel} with $s_i=x_{i-1}^n$ for $i\leq 5$.

\begin{proof}[Proof of Proposition~\ref{prop:p2family}~(ii)]
Let $v\in M_n(\QQ)$. The diagonal coefficients of $f_v=\sum v_is_i$ are all nonzero, since $s_1,\dotsc,s_r$ are monomials in $\QQ[x_0,\dotsc,x_4]$ with $s_i=x_{i-1}^n$ for $i\leq 5$. Then it follows from Lemma~\ref{lem:taurho} that $F_v=\sum v_i\tau^*s_i=\tau^*f_v$ is nonzero. Thus $\deg(F_v)=12n$ and the curve $X_v=V_+(F_v)\subset\mathbb P^2$ lies in $\mathcal C_n$. Hence we obtain a map $$\iota:M_n(\QQ)/\QQ^\times\to \mathcal C_n, \quad v\mapsto X_v.$$
To show that $\iota$ is injective, we take $v,w\in M_n(\QQ)$ with $\iota(v)=\iota(w)$. Hence $X_v=X_{w}$ as closed subschemes of $\mathbb P^2$. This implies $(F_v)=(F_{w})$ since the ideals $(F_v)=(\tau^*f_v)$ and $(F_{w})=(\tau^*f_w)$ of $\QQ[x,y,z]$ are saturated by $(*)$ below. Thus there is $\lambda\in \Q[x,y,z]$ with \begin{equation}\label{eq:pulloutu}
\tau^*f_v=F_v=\lambda F_{w}=\lambda \tau^*f_{w}.
\end{equation}
It follows that $\lambda\in \QQ^\times$ since $F_v$ and $F_{w}$ both have degree 12. We now exploit that $\tau$ is a birational map with inverse $\rho$: There is a nonzero $u$ in $A=\QQ[x_0,\dotsc,x_4]/(\sigma_2,\sigma_4)$ such that for all $i$ it holds $\rho^*\tau_i=\tau_i(\rho_0,\rho_1,\rho_2)=u x_i$ inside $A$; such a polynomial $u$ can be computed explicitly in Macaulay2. Then applying $\rho^*$ to \eqref{eq:pulloutu} gives inside $A$ the equalities
$$
u^nf_v=\rho^*F_v=\lambda \rho^*F_{w}=u^n\lambda f_{w}.
$$
We deduce $f_v=\lambda f_{w}$ inside the integral ring $A$ and thus we conclude $v=\lambda w$ since the $s_i$ form a basis of the $\QQ$-module $A_n$. This shows that $\iota$ is injective as desired.  

Claim $(*)$: For any homogeneous $f\in \QQ[x_0,\dotsc,x_4]$ with all diagonal coefficients nonzero, the ideal $(\tau^*f)$ in $R=\QQ[x,y,z]$ is saturated. To prove $(*)$ it suffices to show that the saturation $(F)^{\textnormal{sat}}$ is contained in $(F)$ for $F=\tau^*f$. Let $s$ be in $(F)^{\textnormal{sat}}$. There exist $n\in \ZZ_{\geq 0}$ and $r\in R$ with $x^ns=rF$. If $n=0$ then $s$ lies in $(F)$ as desired. Thus we may and do assume that $n\geq 1$. As $R$ is a UFD and $x$ is prime in $R$, we obtain $x\mid r$ or $x\mid F$. We claim that $x\nmid F$. Indeed, as $V_+(F)\subset \mathbb P^2$ satisfies $(\tau)$ by Proposition~\ref{prop:p2family}~(i), the divisibility $x\mid F$ would give that $V_+(x)\subseteq V_+(F)$ also satisfies $(\tau)$ but $V_+(x)$ fails $(\tau)$. It follows that $x\mid r$ and thus $x^n\mid r$. Hence $t=r/x^n$ lies in $R$, which implies that $s=t F$ lies in $(F)$ as desired. This proves our claim $(*)$ and thus completes the proof of Proposition~\ref{prop:p2family}~(ii).\end{proof}

\begin{proof}[Proof of Proposition~\ref{prop:p2family}~(iii)]
We define $g=(2n+1)^2$ and we recall that $\mathcal C$ denotes the set of smooth, projective and geometrically connected curves inside $\mathbb P^4$. Proposition~\ref{prop:generalico} gives an open dense $U\subset M_n\subset \mathbb A^r$ together with an injective map 
\begin{equation}\label{eq:firstinjection}
U(\QQ)/\QQ^\times\hookrightarrow \mathcal C,\quad v\mapsto X'_v.
\end{equation}
Here $X'_v\subset \mathbb P^4$ is the curve of genus $g$ given by the non-degenerate ico model $X_f$ with $f=\sum v_is_i$ of degree $n$ as in \eqref{def:genicomodel}. As all diagonal coefficients of $f$ are nonzero, it follows from Lemma~\ref{lem:taurho} that $F_v=\sum v_i\tau^*s_i=\tau^*f$ is nonzero and hence $\deg(F_v)=12n$. In particular $F_v$ is not constant. Thus $X_v=V_+(F_v)$ is a curve inside $\mathbb P^2$ and we obtain a map 
$$\iota: U(\QQ)/\QQ^\times\to \mathcal C_n, \quad v\mapsto (X_v)_{\textnormal{red}}.$$
To simplify notation we write $X_v$ for $(X_v)_{\textnormal{red}}$ in what follows in this proof.  We now show that $X_v$ is a geometrically integral curve over $\QQ$ of geometric genus $g$. Proposition~\ref{prop:p2family}~(i) gives that $X_v$ satisfies $(\tau)$. Thus $U_v=X_v\setminus C_\tau$ is open dense in $X_v$ by Theorem~\ref{thm:p2},  and  \eqref{eq:nondegcompp2} provides a nonempty open $U'_v\subseteq X'_f$ such that $\tau$ restricts to an isomorphism \begin{equation}\label{eq:Uvtauiso}
\tau:U_v\isomto U'_v
\end{equation}
by Lemma~\ref{lem:isotaurho}.
Then, on using that the curves $X_v'$ and $U'_v$ are geometrically integral over $\QQ$ of geometric genus $g$, we deduce that the reduced curves $U_v$ and $X_v$ are geometrically integral over the characteristic zero field $\QQ$ and have again geometric genus $g$.
We next show that $\iota$ is injective. For this it suffices to observe that $\iota$ is the composition of the injective map \eqref{eq:firstinjection} with the map $X'_v\mapsto X_v$ which is also injective. Indeed if $v,w$ in $U(\QQ)$ satisfy  $X_v=X_{w}$, then the generic point $\xi$ of $X_v=X_{w}$ lies in $U_v=U_{w}$ and thus \eqref{eq:Uvtauiso} implies that $X'_v=\overline{\tau(\xi)}=X'_{w}$. This completes the proof of  Proposition~\ref{prop:p2family}~(ii). \end{proof}

Finally we deduce Corollary~\ref{cor:effmordellp2} from Theorem~\ref{thm:p2}.

\begin{proof}[Proof of Corollary~\ref{cor:effmordellp2}]
By assumption $F\in \ZZ[x,y,z]$ is homogeneous of degree $d\geq 1$ and $X=V_+(F)$ satisfies $(\tau)$. Let $x\in X(\QQ)$. We first reduce to the situation treated in Theorem~\ref{thm:p2}~(iii).  Recall that $U=X\setminus C_\tau$, and write $F=\prod F_i^{n_i}$ with $n_i\in\NN$ and $F_i\in \ZZ[x,y,z]$  of degree $d_i\geq 1$ such that $X_i=V_+(F_i)$ is integral.  Theorem~\ref{thm:p2}~(ii) gives $$h(x)=0 \quad \textnormal{if } \   x\in (X\setminus U)(\QQ).$$ Thus we may and do assume that $x\in U(\QQ)$. As $U=\cup U_i$ for $U_i=X_i\setminus C_\tau$, we see that $x$ lies in some $U_i(\QQ)$. The curve $X_i\subset\mathbb P^2$  satisfies $(\tau)$ since $X_i$ is contained in $X$ which satisfies $(\tau)$. Hence an application of  Theorem~\ref{thm:p2}~(i) and (iii) with the integral curve $X_i$ with $x\in U_i(\QQ)$ gives a non-degenerate ico model $X_f$ for $f\in \ZZ[x_0,\dotsc,x_4]$ with $$h(x)\leq c\nu_f^{24} \quad \textnormal{and} \quad |f|\leq u|F_i|^{v}.$$ 
Here we used that $d_i\leq d$, and we deduced $h(x)\leq c\nu_f^{24}$ by replacing in \eqref{eq:p2mordellbound} the ico model $X_f$ by the possibly larger ico model $X_{\tilde{f}}\supseteq X_f$ constructed in Theorem~\ref{thm:p2}~(iii).
The definition of $\nu_f$ in \eqref{eq:defnuf} gives that $\nu_f\leq |f|^5$, and \cite[Lem 1.6.11]{bogu:diophantinegeometry} combined with $F=\prod F_i^{n_i}$ implies that $|F_i|\leq 2^{3d}|F|$. Then the displayed bounds prove Corollary~\ref{cor:effmordellp2}.
\end{proof}

\section{Fermat problem inside a rational surface}\label{sec:fermat}

In this section we discuss the Fermat problem inside a rational surface $S$ over $\QQ$. After formulating the general problem and briefly reviewing the classical case $S=\mathbb P^2$, we focus on the case $S=S^{\textnormal{ico}}$ of the icosahedron surface. We also motivate the general case by discussing a Fermat conjecture inside any variety. For each $m\in \NN$ we write $\mathbb P^m=\mathbb P^m_\QQ$.

\paragraph{Fermat problem.}We now generalize the Fermat problem by replacing $\mathbb P^2$ with any birationally equivalent projective surface over $\QQ$. Let $S\subseteq \mathbb P^m$ be a projective rational surface over $\QQ$. We call $x\in \mathbb P^m(\QQ)$ trivial if $x=(x_i)$ with $x_i\in\{-1,0,1\}$, and we call a subset of $\mathbb P^m(\QQ)$ trivial if all its points are trivial. In particular the empty set is trivial. 

\vspace{0.3cm}

\noindent {\bf Problem (F).}\emph{ 
For any $a\in \ZZ^{m+1}$ with $a_i\neq 0$, try to construct $n_0\in \NN$ such that $X_n(\QQ)$ is trivial when $n\geq n_0$ where $X_n$ is the variety over $\QQ$ defined inside $S$ by}
$$X_n\subseteq S: a_0x_0^n+\dotsc+a_mx_m^n=0, \quad n\in\NN.$$

As $S$ is rational over $\QQ$, the infinite set $S(\QQ)$ is `large' and hence the Diophantine problem (F) is `non-trivial'. In particular, for any $n\in \NN$ there exists $a\in \ZZ^{m+1}$ with $a_i\neq 0$ such that $X_n(\QQ)$ is non-trivial and thus $n_0$ has to depend on $a$ if it exists. We conjecture that (F) can be solved  if and only if $(S\cap Z)(\QQ)$ is trivial, where $Z\subset \mathbb P^m$ is given by 
\begin{equation}\label{def:exceptionalz}
Z=\cap_{j}\cup_{i\neq j}V_+(f_{ij}), \quad f_{ij}=x_i^2x_j-x_j^3.
\end{equation}
This is a special case of Conjecture~\ref{conj:fermat} which we formulate for any projective variety over $\QQ$. To provide some motivation for (F) and Conjecture~\ref{conj:fermat}, we show in $\mathsection$\ref{sec:fermatconj} that a higher dimensional version of the $abc$-conjecture of Masser--Oesterl\'e implies Conjecture~\ref{conj:fermat}. We next discuss (F) for two rational surfaces $S$ with $(S\cap Z)(\QQ)$ trivial.

\paragraph{Classical case $S=\mathbb P^2$.} In this case (F) becomes the following classical Fermat problem which was solved by Wiles~\cite{wiles:modular} for $a=b=c=1$ with the optimal $n_0=3$: For arbitrary nonzero $a,b,c\in \ZZ$, try to construct $n_0\in \NN$ such that the generalized Fermat equation
$$ax^n+by^n=cz^n, \quad (x,y,z)\in \ZZ^{3}, \quad \gcd(x,y,z)=1, \quad n\in \NN,$$
has no solution with $xyz\notin\{-1,0,1\}$ when $n\geq n_0$. Many authors solved variations of $(F)$ for large classes of explicit $a,b,c$ with an optimal $n_0$, see the surveys  \cite{bechdaya:genfermat,bemisi:genfermat,grwi:genfermat}. While (F) is still open for general $a,b,c$, it is widely expected that $(F)$ can be solved. For instance the $abc$-conjecture of Masser--Oesterl\'e directly solves $(F)$.

\vspace{0.3cm}

\paragraph{The case $S=S^{\textnormal{ico}}$.} We now consider the surface $S^{\textnormal{ico}}\subset \mathbb P^4:\sigma_2=0=\sigma_4$ where $\sigma_k$ is the $k$-th elementary symmetric polynomial. It is a rational surface over $\QQ$ by classical constructions \cite{clebsch:clebschsurface,klein:cubicsurfaces}, and the following result solves $(F)$ for $S=S^{\textnormal{ico}}$. Let $h$ be the usual logarithmic Weil height on the curve $X_n\subset\mathbb P^4$ for $n\in\NN$ and put $c=10^{10^{12}}$.

\begin{corollary}\label{cor:fermat}
Suppose that $S=S^{\textnormal{ico}}$ and let $a\in \ZZ^5$ with $a_i\neq 0$. 
\begin{itemize}
\item[(i)] All $x\in\cup_{n\in \NN}X_n(\QQ)$ satisfy $h(x)\leq c\nu^{24}$ for  $\nu=\textnormal{rad}(\prod a_i)$.
\item[(ii)] There exists $n_0\in\NN$ such that $X_n(\QQ)$ is trivial when $n\geq n_0$.
\end{itemize}
\end{corollary}
While $abc$ directly solves (F) for $\mathbb P^2$, this does not work for $S^{\textnormal{ico}}\subset\mathbb P^4$ and it is unclear\footnote{In particular it is unclear whether (i) or (ii) follows from the proof (\cite{elkies:abcmordell}) of Mordell via an effective $abc$ with optimal exponent $1+\epsilon$: This proof gives height bounds depending inter alia on the Belyi degree.} whether one can deduce (i) or (ii) from $abc$. An effective higher dimensional $abc$ would give an effective $n_0$ in (ii), but it is again unclear whether one can deduce (i); see $\mathsection$\ref{sec:fermatconj}. 

Our $n_0$ is ineffective in the following sense: We can only control the cardinality of the set $\mathcal N=\{n\in\NN; \, X_n(\QQ)\,$non-trivial$\}$, but not its maximum $n_0=\max \mathcal N$. To explain this we now describe the proof of Corollary~\ref{cor:fermat}. After deducing the height bound (i) from Corollary~\ref{cor:effmordellxf}, we use (i) to reduce (F) to a specific class \eqref{eq:gensunit} of generalized unit equations $(u)$. Using the subspace theorem based on Diophantine approximations, finiteness for the non-degenerate solutions of $(u)$ was established independently by Evertse~\cite{evertse:gensunits} and van der Porten--Schlickewei~(see \cite[p.95]{schlickewei:sunitnf}). Moreover, the quantitative result of Evertse--Schlickewei--Schmidt~\cite{evscsc:uniteq} for $(u)$ allows to effectively bound $|\mathcal N|$ in terms of $\max|a_i|$. However proving effectivity for \eqref{eq:gensunit}, and thus for our $n_0$, is an interesting open problem.

 \paragraph{Comparison.}As the underlying geometry of (F) is equivalent for $\mathbb P^2$ and $S^{\textnormal{ico}}$, we conjecture that (F) behaves similarly (or is even related) for $\mathbb P^2$ and $S^{\textnormal{ico}}$. One can ask whether it is possible to transport arithmetic information for (F) via a birational map $\mathbb P^2\simeq S^{\textnormal{ico}}$? This works for effective Mordell as shown in Section~\ref{sec:p2}, but it requires new ideas for (F) since it is not clear how to suitably relate the curves $X_n\subset \mathbb P^2$ and $X_n\subset S^{\textnormal{ico}}$. 
 
Certain aspects of our strategy for $S^{\textnormal{ico}}$ are similar to the strategy for $\mathbb P^2$ developed (among others) by Frey~\cite{frey:linkssarav}, Ribet~\cite{ribet:levellowering}, Taylor--Wiles~\cite{taywil:modular} and Wiles~\cite{wiles:modular}. For example both strategies use a moduli interpretation of $X_n(\QQ)$ in terms of abelian varieties and both use modularity. However there are also important conceptual differences: We do not use level lowering and irreducibility theorems, but instead we rely on other (more Diophantine  analytic) tools such as for example Masser--W\"ustholz isogeny estimates~\cite{mawu:periods,mawu:abelianisogenies} based on transcendence and Schmidt's subspace theorem~\cite{schmidt:subspace} based on Diophantine approximations. 
After Wiles many authors further developed the strategy for $\mathbb P^2$ by introducing substantial new ideas: Can one combine some of these ideas (e.g. from Darmon's program~\cite{darmon:program}) with our strategy to solve new cases of (F)?



\subsection{Proof of Corollary~\ref{cor:fermat}}

We continue our notation. In this section we solve the Fermat problem (F) for the rational surface $S=S^{\textnormal{ico}}$ by combining Corollary~\ref{cor:effmordellxf} with the finiteness result for generalized unit equations which comes from the theory of Diophantine approximations.

\begin{proof}[Proof of Corollary~\ref{cor:fermat}]
Let $a\in \ZZ^5$ with $a_i\neq 0$ and let $n\in \NN$. To prove the height bound, we consider the homogeneous $f=\sum a_i x_i^n$ in $\ZZ[x_0,\dotsc,x_4]$ of degree $n\geq 1$ and we define $X_f\subset \mathbb P^4_\ZZ$ as in \eqref{def:icomod}. Then $X_f$ is a non-degenerate ico model over $\ZZ$ since all $a_i\neq 0$. We obtain $X_n=(X_f)_\QQ$ inside $\mathbb P^4_\QQ$ and we compute $\nu_f=\rad(\prod a_i)=\nu$. Thus Corollary~\ref{cor:effmordellxf} proves that all $x\in X_n(\QQ)$ satisfy $h(x)\leq c\nu^{24}$ as claimed in Corollary~\ref{cor:fermat}.

We next reduce $(F)$ to a special class of generalized $S$-unit equations for some finite set $S$ which is controlled in terms of $\nu$ by the height bound. Let $x\in X_n(\QQ)$ and write $x=(x_i)$ with $x_i\in \ZZ^5$ satisfying $\gcd(x_i)=1$. After relabeling, we may and do assume that $x_0,\dotsc,x_k$ are all nonzero and $x_{k+1},\dotsc, x_4$ are all zero. Then $\prod_{i\leq k} a_ix_i\neq 0$ and $k\geq 1$,  since $x\in X_n(\QQ)$ and each $a_i\neq 0$. Let $S$ be the set of rational primes dividing $\prod_{i\leq k} a_ix_i$.  As $x$ lies in $X_n(\QQ)$, it defines a solution $u=u(x)$ of the generalized $S$-unit equation:
\begin{equation}\label{eq:gensunit}
u_0+\dotsc+u_{k-1}=1, \quad  u\in \mathcal U^k; \quad \ \ u(x)_i=-\tfrac{a_i}{a_k}\bigl(\tfrac{x_i}{x_k}\bigl)^n
\end{equation}
for $\mathcal U$ the group of $S$-units in $\QQ$. We say that $u\in \mathcal U^k$ is degenerate if $\sum_{i\in I}u_i=0$ for some nonempty $I\subsetneq \{0,\dotsc,k-1\}$. Now the key point is as follows: As $S$ is controlled in terms of $\nu$ by our height bound, Diophantine approximations (\cite[Thm 1.1]{evscsc:uniteq}) give that \eqref{eq:gensunit} has at most finitely many non-degenerate solutions for our given $a$. Thus there  is a constant $c_a$ depending only on $a$ with the following property\footnote{In what follows in this proof we freely enlarge the constant $c_a$ whenever necessary.}: If $u=u(x)$ is non-degenerate, then \begin{equation}\label{eq:exponentbound}
nh\bigl(\tfrac{x_i}{x_k}\bigl)\leq c_a
\end{equation}
for all $i\leq k$ where $h(z)$ is the usual logarithmic Weil height of $z\in \QQ$. This allows us to solve $(F)$ as follows: Suppose that $x\in X_n(\QQ)$ is non-trivial. Then it holds $h(\tfrac{x_i}{x_k})\geq \log 2$ for some $i\leq k$, since $x_i=0$ for $i>k$ and hence $\gcd(x_0,\dotsc,x_k)=\gcd(x_i)=1$. In the case when $u(x)$ is non-degenerate, we thus deduce from \eqref{eq:exponentbound} that $n\leq c_a/\log 2$ as desired.

In the case when $u(x)$ is degenerate, a similar argument works since the relation $\sum_{i\in I} u_i=0$ for some nonempty $I\subsetneq\{0,\dotsc,k-1\}$  allows us to establish a variation of \eqref{eq:exponentbound}. However, to conclude in the degenerate case, it becomes much more subtle to exploit the crucial coprime information and we need to use in addition that $x\in S^{\textnormal{ico}}(\QQ)$. Assume that $u(x)$ is degenerate. We now go through all possible cases for $k$ and $|I|$.

Case $k=4$ and $|I|=3$. Assume that $I=\{0,1,2\}$. Then the relation $\sum_{i\in I}u_i=0$ gives that $u_{3}=1$ by \eqref{eq:gensunit} and that $v_0+v_1=1$ where $v_0=-\tfrac{u_0}{u_2}$ and $v_1=-\tfrac{u_1}{u_2}$ are $S$-units. As $S$ is controlled in terms of $\nu$, Mahler's theorem~\cite{mahler:approx1} gives that the solution $(v_0,v_1)$ of the $S$-unit equation $v_0+v_1=1$ has height bounded by $c_a$. Thus we obtain
\begin{equation}\label{eq:exponentbound2}
nh(z)\leq c_a, \quad  z\in\{\tfrac{x_0}{x_2},\tfrac{x_1}{x_2},\tfrac{x_3}{x_4}\}.
\end{equation}
Suppose that $x\in X_n(\QQ)$ is non-trivial. Then we claim that $\tfrac{x_0}{x_2}$, $\tfrac{x_1}{x_2}$ or $\tfrac{x_3}{x_4}$ lies not in $\{\pm 1\}$. This claim and \eqref{eq:exponentbound2} give that $n\leq c_a/\log 2$ as desired. To prove the claim, we assume for the sake of contradiction that the statement does not hold.  Then we obtain
\begin{equation}\label{eq:coprimcond}
x_3=\pm x_4 \quad \textnormal{ and } \quad x_0=\pm x_1=\pm x_2 \quad \textnormal{ with } \quad \gcd(x_0,x_3)=\gcd(x_i)=1.
\end{equation}
Hence, on exploiting that $x\in S^{\textnormal{ico}}(\QQ)$ satisfies $\sigma_2(x)=0=\sigma_4(x)$ and that all $x_i$ are nonzero, we deduce the following relations for some $r\in \pm\{1,3\}$ and $s,t\in \ZZ$ with $t\neq 0$: 
$$(i) \ rx_0^2+x_3^2=sx_0x_3 \quad  \textnormal{ and  } \quad  (ii) \ 2x_0=tx_3.$$  
Now, as $x$ is non-trivial by assumption, there exists a rational prime $p$ which divides one of the $x_0,\dotsc,x_4$. Thus $p$ divides $x_0$ or $x_3$ by \eqref{eq:coprimcond}. If $p\mid x_0$ then (i) shows that $p\mid x_3$, which contradicts \eqref{eq:coprimcond}. If $p\mid x_3$ then  (i) and \eqref{eq:coprimcond} imply that $p=3$, which gives a contradiction by (ii) and \eqref{eq:coprimcond}. This proves our claim, and thus completes the proof for $I=\{0,1,2\}$. For the other three subsets $I\subseteq \{0,1,2,3\}$ with $|I|=3$, we can use precisely the same arguments by symmetry since the polynomials $\sigma_2$ and $\sigma_4$ are symmetric.

Case $k=4$ and $|I|=2$. Assume that $I=\{0,1\}$. Then the relation $\sum_{i\in I}u_i=0$ gives that $u_{2}+u_3=1$ by \eqref{eq:gensunit} and that $\tfrac{u_0}{u_1}=-1$. As $S$ is controlled in terms of $\nu$, the solution $(u_2,u_3)$ of the $S$-unit equation $u_2+u_3=1$ has height bounded by $c_a$ and thus we obtain
\begin{equation}\label{eq:exponentbound3}
nh(z)\leq c_a, \quad  z\in\{\tfrac{x_2}{x_4},\tfrac{x_3}{x_4},\tfrac{x_0}{x_1}\}.
\end{equation}
A permutation of the $x_i$ transforms \eqref{eq:exponentbound3}  into \eqref{eq:exponentbound2}, and the polynomials $\sigma_2$ and $\sigma_4$ are symmetric. By symmetry we can thus apply precisely the same arguments as above to conclude that $\tfrac{x_2}{x_4}$, $\tfrac{x_3}{x_4}$ or $\tfrac{x_0}{x_1}$ lies not in $\{\pm 1\}$. Then \eqref{eq:exponentbound3} gives that $n\leq c_a/\log 2$ as desired.

We observe that the above two cases are the only possible cases since $|I|<k$. Indeed $|I|=1$ is not possible since each $u_i\neq 0$, while the case $k=3$  is not possible since $x_i\neq 0$ for all $i\leq 3$ but $\sigma_4(x)=0$. This completes the proof of Corollary~\ref{cor:fermat}.
\end{proof}

To deal with the degenerate case, we use in the above proof rather complicated computations. More conceptually, these computations simply check that $(S^{\textnormal{ico}}\cap Z)(\QQ)$ is trivial. This observation opens the way for formulating a general Fermat conjecture ($\mathsection$\ref{sec:fermatconj}).

Several parts of the above proof can be generalized in various directions. We are currently trying to work out a more conceptual description of (the limits of) our strategy.

\subsection{A Fermat conjecture inside any projective variety over $\QQ$}\label{sec:fermatconj}

To provide some motivation for Problem~(F), we formulate and discuss in this section a Fermat conjecture inside any projective variety over $\QQ$. We continue our notation.

 Let $m\in \NN$ and let $V\subseteq \mathbb P^m$ be a nonempty closed subscheme. We say that Fermat holds inside $V$ if for each $a\in\ZZ^{m+1}$ with $a_i\neq 0$ there exists $n_0\in\NN$ such that $X_n(\QQ)$ is trivial when $n\geq n_0$, where $X_n$ is the variety over $\QQ$ defined inside $V$ by
$$X_n\subseteq V: a_0x_0^n+\dotsc+a_mx_m^n=0.$$
We define the closed subscheme $Z\subset\mathbb P^m$ as in \eqref{def:exceptionalz}. The following conjecture, which we stated in \eqref{def:exceptionalz} when $V$ is a rational surface over $\QQ$, gives in particular a simple description of the class of projective rational surfaces over $\QQ$ for which Problem~(F) can be solved.
 
\begin{conjecture}\label{conj:fermat}
Fermat holds inside $V$ if and only if $(V\cap Z)(\QQ)$ is trivial.
\end{conjecture}We observe that $Z(\QQ)$ is trivial if $m\leq 2$ and that Conjecture~\ref{conj:fermat} holds for $m=1$. If $(V\cap Z)(\QQ)$ contains a non-trivial point $x$, then  \eqref{eq:computefermatz} allows to directly construct $a\in\ZZ^{m+1}$ with $a_i\neq 0$ such that $x\in X_n(\QQ)$ for infinitely many $n\in\NN$. This proves that one direction of Conjecture~\ref{conj:fermat} always holds. The  converse direction is widely open for most $V$. 
\paragraph{Conditional on $(abc)^n$.}Darmon--Granville~\cite[$\mathsection$9.2]{dagr:superell} observed that $(abc)^n$ directly implies a strong finiteness result for higher dimensional Fermat equations, where $(abc)^n$ is a higher dimensional version (\cite[$\mathsection$5.2]{dagr:superell}) of the $abc$-conjecture of Masser--Oesterl\'e. We now use the same idea to show that $(abc)^n$ directly implies Conjecture~\ref{conj:fermat}. For this we may and do assume that $m\geq 2$. Let $a\in\ZZ^{m+1}$ with $a_i\neq 0$, let $n\in \NN$ and make assumption $(*)$ There is a non-trivial point $x\in X_n(\QQ)$ and $(V\cap Z)(\QQ)$ is trivial. We compute \begin{equation}\label{eq:computefermatz}
Z(\QQ)=\{x\in\ZZ^{m+1}, \, \gcd(x_i)=1, \, x_i=0 \textnormal{ or } \tfrac{x_i}{x_j}=\pm 1 \textnormal{ for some } j\neq i \}.
\end{equation}
Then, as all $a_i\neq 0$ and $m\geq 2$, assumption $(*)$ gives $I\subseteq\{0,\dotsc,m\}$ with $|I|\geq 2$ such that $\sum_{i\in I} a_ix_i^n=0$ and no proper subsum vanishes. If $|I|=2$ then $(*)$ and \eqref{eq:computefermatz} imply that $n$ is bounded. If $|I|\geq 3$ then $(abc)^n$ combined with $(*)$ and \eqref{eq:computefermatz} implies that $n$ is bounded. This proves that $(abc)^n$ implies Conjecture~\ref{conj:fermat}, since the other direction always holds as shown above. Moreover this effectively bounds $n_0$ in terms of $\max|a_i|$ if the constants in $(abc)^n$ are effective. However it is not clear whether $(abc)^n$ allows for all $n<n_0$, with $X_n(\QQ)$ finite, to effectively bound the height of each $x\in X_n(\QQ)$ in terms of $\max|a_i|$.

\section{Effective Mordell}\label{sec:effmordell}

In this section we discuss various aspects of the effective Mordell problem over $\QQ$.

 Let $F\in \ZZ[x,y,z]$ be homogeneous of degree $d\geq 1$. Write $|F|=\max_\iota|a_\iota|$ for $a_\iota$ the coefficients of $F$ and let $h$ be the usual (\cite[p.16]{bogu:diophantinegeometry}) logarithmic Weil height on $\mathbb P^2=\mathbb P^2_\QQ$. 

\vspace{0.3cm} 

\noindent{\bf Conjecture (EM).}
\emph{
If $X=V_+(F)\subset\mathbb P^2$ has geometric genus $\geq 2$, then any $x\in X(\QQ)$ satisfies $h(x)\leq \mu|F|^{\kappa}$ for effective constants $\mu,\kappa$ depending only on $d$. }
             
\vspace{0.3cm}

Here by an effective constant $c$ depending only on $d$ we mean that   $c=\exp^{\circ n}(d)$ for some explicit $n\in\NN$, where $\exp^{\circ n}(\cdot)$ is the $n$-th iteration of $\exp(\cdot)$. For each curve $C$ over an arbitrary number field satisfying the elliptic setting of the Manin--Dem'janenko criterion,  Checcoli--Veneziano--Viada~\cite{chvevi:effmordell1,chvevi:effmordell2,vevi:effmordell} established a version  of Conjecture~(EM) with a drastically better dependence on the height of $C$. Our Corollary~\ref{cor:effmordellp2} proves Conjecture~(EM) for all $X\subset \mathbb P^2$ satisfying criterion $(\tau)$. One can try to exploit uniformity aspects of our height bounds in Corollaries \ref{cor:effmordell} and~\ref{cor:effmordellpullbackp2} to improve the dependence on $|F|$ for certain classes of curves, including subfamilies of the $\mathcal C_n$ constructed in $\mathsection$\ref{sec:families}. 

\paragraph{Dependence on $|F|$.}For most curves the `correct' dependence on $|F|$ should be logarithmic, see Zhang~\cite[Conj 1.4]{zhang:iccm}. To this end, Ih~\cite{ih:heightbound} deduced from Vojta's conjecture such a logarithmic dependence if one restricts to families of curves. However it is not clear\footnote{For example Elkies' proof (\cite{elkies:abcmordell}) of Mordell via $abc$ with optimal exponent $1+\epsilon$ gives height bounds depending polynomially on the Belyi degree and such bounds do not allow  to deduce $\ll\log |F|$.} whether the $abc$-conjecture of Masser--Oesterl\'e with optimal exponent $1+\epsilon$ or `Mordell effectif' in \cite{moret-bailly:effmordell} implies such a logarithmic dependence. In view of this we formulated Conjecture~(EM) with a polynomial dependence on $|F|$.

\paragraph{Algorithmic Mordell.}We now discuss the following algorithmic Mordell problem: Given a projective curve $X$ over $\QQ$ of geometric genus $\geq 2$, prove the existence of an algorithm which computes $X(\QQ)$. Conjecture~(EM) directly solves algorithmic Mordell for all $X\subset \mathbb P^2$. Moreover, there are several  conditional proofs of algorithmic Mordell in the literature, see for example Elkies~\cite{elkies:abcmordell} and Alp\"oge--Lawrence~\cite{alla:conditional}. The algorithm of Alp\"oge--Lawrence~\cite{alla:conditional} has the following feature: If it terminates then it provably computes $X(\QQ)$ and it always terminates conditional on certain standard conjectures. 

Based on the works of Skolem, Chabauty, Coleman, Kim and many others, there are plenty of results and methods in the literature which solve algorithmic Mordell for large classes of curves $X$. For example Poonen--Stoll~\cite{post:hyperelliptic} show that Chabauty's method solves algorithmic Mordell for a huge computable subfamily (of positive lower density, tending to $1$ if $g\to \infty$) of all odd degree hyperelliptic curves over $\QQ$ of genus $g\geq 3$. For a survey of many algorithmic methods/results, we refer to Balakrishnan et al~\cite{babebilamutrvo:effective} which also discusses algorithmic aspects of the proof of Mordell by Lawrence--Venkatesh \cite{lave:mordell}.

The method of Manin--Dem'janenko~\cite{demjanenko:rational,manin:torsion} solves algorithmic Mordell for all curves satisfying their criterion. The explicit bounds \cite{rvk:gl2,vkkr:hms,vkkr:chms} for the height $h_\phi$ solve algorithmic Mordell for all projective $X$ satisfying \raisebox{0.15ex}{{\scriptsize($\gl2$)}}, \raisebox{0.15ex}{{\scriptsize$(H)$}} and \raisebox{0.15ex}{{\scriptsize$(cH)$}} in \eqref{def:criteria} under the following technical assumption $(*)$ one can compute $\nu$ and all points $x\in X(\QQ)$ of bounded $h_\phi$. While this assumption should be harmless, a rigorous verification  appears to require a substantial (technical) effort. Let $F$ be a totally real field of odd degree. Building on the arguments of \cite{rvk:gl2}  and using the strategy of \cite{vkkr:hms}, Alp\"oge~\cite{alpoge:modularitymordell} solved over $F$ algorithmic Mordell for a family of curves satisfying \raisebox{0.15ex}{{\scriptsize($\gl2$)}} and for a special class of curves satisfying \raisebox{0.15ex}{{\scriptsize$(H)$}}: Those curves with a quasi-finite morphism to a representable Hilbert modular variety. It would be interesting to explore whether our Corollary~\ref{cor:effmordell} for non-degenerate curves over $\QQ$ can be useful for proving new cases of algorithmic Mordell.

\newpage

{\tiny
\bibliographystyle{amsalpha}
\bibliography{../literature}

\newcommand{\etalchar}[1]{$^{#1}$}
\def\cprime{$'$}
\providecommand{\bysame}{\leavevmode\hbox to3em{\hrulefill}\thinspace}
\providecommand{\MR}{\relax\ifhmode\unskip\space\fi MR }
\providecommand{\MRhref}[2]{%
  \href{http://www.ams.org/mathscinet-getitem?mr=#1}{#2}
}
\providecommand{\href}[2]{#2}
\begin{thebibliography}{BCDY15}

\bibitem[AL24]{alla:conditional}
L.~Alp\"oge and B.~Lawrence, \emph{Conditional algorithmic {M}ordell},
  Preprint, arXiv: 2408.11653 (2024), 69 pages.

\bibitem[Alp21]{alpoge:modularitymordell}
L.~Alp\"oge, \emph{Modularity and effective {M}ordell {I}}, Preprint,
  arXiv:2109.07917 (2021), 27pages.

\bibitem[AS98]{arse:genus}
F.~Arslan and S.~Sert\"oz, \emph{Genus calculations of complete intersections},
  Comm. Algebra \textbf{26} (1998), no.~8, 2463--2471.

\bibitem[AV02]{abvi:compactstable}
D.~Abramovich and A.~Vistoli, \emph{Compactifying the space of stable maps}, J.
  Amer. Math. Soc. \textbf{15} (2002), no.~1, 27--75.

\bibitem[BBB{\etalchar{+}}21]{babebilamutrvo:effective}
J.~Balakrishnan, A.~Best, F.~Bianchi, B.~Lawrence, S.~M\"{u}ller,
  N.~Triantafillou, and J.~Vonk, \emph{Two recent {$p$}-adic approaches towards
  the (effective) {M}ordell conjecture}, Arithmetic {L}-functions and
  differential geometric methods, Progr. Math., vol. 338,
  Birkh\"{a}user/Springer, Cham, 2021, pp.~31--74.

\bibitem[BCDY15]{bechdaya:genfermat}
M.~Bennett, I.~Chen, S.~Dahmen, and S.~Yazdani, \emph{Generalized {F}ermat
  equations: a miscellany}, Int. J. Number Theory \textbf{11} (2015), no.~1,
  1--28.

\bibitem[BG06]{bogu:diophantinegeometry}
E.~Bombieri and W.~Gubler, \emph{Heights in {D}iophantine geometry}, New
  Mathematical Monographs, vol.~4, Cambridge University Press, Cambridge, 2006.

\bibitem[BMS16]{bemisi:genfermat}
M.~Bennett, P.~Mihailescu, and S.~Siksek, \emph{The generalized {F}ermat
  equation}, Open problems in mathematics, Springer, [Cham], 2016,
  pp.~173--205.

\bibitem[BMZ07]{bomaza:anomalous}
E.~Bombieri, D.~Masser, and U.~Zannier, \emph{Anomalous
  subvarieties---structure theorems and applications}, Int. Math. Res. Not.
  IMRN (2007), no.~19, Art. ID rnm057, 33.

\bibitem[Cha41]{chabauty:method}
C.~Chabauty, \emph{Sur les points rationnels des courbes alg\'ebriques de genre
  sup\'erieur \`a{} l'unit\'e}, C. R. Acad. Sci. Paris \textbf{212} (1941),
  882--885.

\bibitem[Cha90]{chai:hilbmod}
C.-L. Chai, \emph{Arithmetic minimal compactification of the
  {H}ilbert-{B}lumenthal moduli spaces}, Ann. of Math. (2) \textbf{131} (1990),
  no.~3, 541--554.

\bibitem[Cle71]{clebsch:clebschsurface}
A.~Clebsch, \emph{Ueber die {A}nwendung der quadratischen {S}ubstitution auf
  die {G}leichungen 5 {$^{ten}$} {G}rades und die geometrische {T}heorie des
  ebenen {F}\"{u}nfseits}, Math. Ann. \textbf{4} (1871), no.~2, 284--345.

\bibitem[CLZ24]{coloza:effsiegel}
P.~Corvaja, D.~Lombardo, and U.~Zannier, \emph{Examples of effectivity for
  integral points on certain curves of genus 2}, Preprint, arXiv: 2411.17930
  (2024), 45 pages.

\bibitem[Col85]{coleman:method}
R.~Coleman, \emph{Effective {C}habauty}, Duke Math. J. \textbf{52} (1985),
  no.~3, 765--770.

\bibitem[CVV17]{chvevi:effmordell1}
S.~Checcoli, F.~Veneziano, and E.~Viada, \emph{On the explicit torsion
  anomalous conjecture}, Trans. Amer. Math. Soc. \textbf{369} (2017), no.~9,
  6465--6491.

\bibitem[CVV19]{chvevi:effmordell2}
\bysame, \emph{The explicit {M}ordell conjecture for families of curves}, Forum
  Math. Sigma \textbf{7} (2019), Paper No. e31, 62.

\bibitem[Dar00]{darmon:program}
H.~Darmon, \emph{Rigid local systems, {H}ilbert modular forms, and {F}ermat's
  last theorem}, Duke Math. J. \textbf{102} (2000), no.~3, 413--449.

\bibitem[Dem66]{demjanenko:rational}
V.~A. Dem'janenko, \emph{Rational points of a class of algebraic curves}, Izv.
  Akad. Nauk SSSR Ser. Mat. \textbf{30} (1966), 1373--1396.

\bibitem[DFL{\etalchar{+}}16]{fuselier:hypergeom}
A.~Deines, J.~Fuselier, L.~Long, H.~Swisher, and F.~Tu, \emph{Generalized
  {L}egendre curves and quaternionic multiplication}, J. Number Theory
  \textbf{161} (2016), 175--203.

\bibitem[DG95]{dagr:superell}
H.~Darmon and A.~Granville, \emph{On the equations {$z^m=F(x,y)$} and
  {$Ax^p+By^q=Cz^r$}}, Bull. London Math. Soc. \textbf{27} (1995), no.~6,
  513--543.

\bibitem[DGH21]{digaha:mordell}
V.~Dimitrov, Z.~Gao, and P.~Habegger, \emph{Uniformity in {M}ordell-{L}ang for
  curves}, Ann. of Math. (2) \textbf{194} (2021), no.~1, 237--298.

\bibitem[DP94]{depa:hilbertmodular}
P.~Deligne and G.~Pappas, \emph{Singularit\'es des espaces de modules de
  {H}ilbert, en les caract\'eristiques divisant le discriminant}, Compositio
  Math. \textbf{90} (1994), no.~1, 59--79.

\bibitem[Elk91]{elkies:abcmordell}
N.~D. Elkies, \emph{{$ABC$} implies {M}ordell}, Internat. Math. Res. Notices
  (1991), no.~7, 99--109.

\bibitem[ESS02]{evscsc:uniteq}
J.-H. Evertse, H.~P. Schlickewei, and W.~M. Schmidt, \emph{Linear equations in
  variables which lie in a multiplicative group}, Ann. of Math. (2)
  \textbf{155} (2002), no.~3, 807--836.

\bibitem[Eve84]{evertse:gensunits}
J.-H. Evertse, \emph{On sums of {$S$}-units and linear recurrences}, Compositio
  Math. \textbf{53} (1984), no.~2, 225--244.

\bibitem[Fal83]{faltings:finiteness}
G.~Faltings, \emph{Endlichkeitss\"atze f\"ur abelsche {V}ariet\"aten \"uber
  {Z}ahlk\"orpern}, Invent. Math. \textbf{73} (1983), no.~3, 349--366.

\bibitem[Fal86]{faltings:history}
\bysame, \emph{Some historical notes}, Arithmetic geometry ({S}torrs, {C}onn.,
  1984), Springer, New York, 1986, pp.~1--8.

\bibitem[FC90]{fach:deg}
G.~Faltings and C.-L. Chai, \emph{Degeneration of abelian varieties},
  Ergebnisse der Mathematik und ihrer Grenzgebiete (3) [Results in Mathematics
  and Related Areas (3)], vol.~22, Springer-Verlag, Berlin, 1990, With an
  appendix by David Mumford.

\bibitem[FOV99]{bertini:book}
H.~Flenner, L.~O'Carroll, and W.~Vogel, \emph{Joins and intersections},
  Springer Monographs in Mathematics, Springer-Verlag, Berlin, 1999.

\bibitem[Fre86]{frey:linkssarav}
G.~Frey, \emph{Links between stable elliptic curves and certain {D}iophantine
  equations}, Ann. Univ. Sarav. Ser. Math. \textbf{1} (1986), no.~1, iv+40.

\bibitem[Gia06]{giaimo:castelnuovo}
D.~Giaimo, \emph{On the {C}astelnuovo-{M}umford regularity of connected
  curves}, Trans. Amer. Math. Soc. \textbf{358} (2006), no.~1, 267--284.

\bibitem[GLP83]{grlape:castelnuovo}
L.~Gruson, R.~Lazarsfeld, and C.~Peskine, \emph{On a theorem of {C}astelnuovo,
  and the equations defining space curves}, Invent. Math. \textbf{72} (1983),
  no.~3, 491--506.

\bibitem[GW24]{grwi:genfermat}
B.~Grechuk and A.~Wilcox, \emph{Generalised {F}ermat equation: a survey of
  solved cases}, Preprint, arXiv: 2412.11933 (2024), 49 pages.

\bibitem[Hir76]{hirzebruch:ck}
F.~Hirzebruch, \emph{The {H}ilbert modular group for the field {${\bf Q}(\surd
  5)$}, and the cubic diagonal surface of {C}lebsch and {K}lein}, Uspehi Mat.
  Nauk \textbf{31} (1976), no.~5(191), 153--166, Translated from the German by
  Ju. I. Manin.

\bibitem[Ih02]{ih:heightbound}
S.~Ih, \emph{Height uniformity for algebraic points on curves}, Compositio
  Math. \textbf{134} (2002), no.~1, 35--57.

\bibitem[JL21]{jalo:stackychevalley}
A.~Javanpeykar and D.~Loughran, \emph{Arithmetic hyperbolicity and a stacky
  {C}hevalley-{W}eil theorem}, J. Lond. Math. Soc. (2) \textbf{103} (2021),
  no.~3, 846--869.

\bibitem[vK14]{rvk:intpointsmodell}
R.~von K\"anel, \emph{Integral points on moduli schemes of elliptic curves},
  Trans. London Math. Soc. \textbf{1} (2014), no.~1, 85--115.

\bibitem[vK21]{rvk:gl2}
\bysame, \emph{The effective {S}hafarevich conjecture for abelian varieties of
  {GL}2-type}, Forum Math. Sigma \textbf{9} (2021), Paper No. e39, 1--39.

\bibitem[vK24]{rvk:cetraro}
\bysame, \emph{Integral points on moduli schemes}, Journal of Number Theory,
  Proceedings of the Second JNT Biennial Conference 2022, available online
  (doi.org/10.1016/j.jnt.2024.07.005) (2024), 70 pages.

\bibitem[vKK19]{vkkr:hms}
R.~von K\"anel and A.~Kret, \emph{Integral points on {H}ilbert moduli schemes},
  Preprint: arXiv:1904.03503 (2019), 53 pages.

\bibitem[vKK23]{vkkr:chms}
\bysame, \emph{Integral points on coarse {H}ilbert moduli schemes}, Preprint,
  arXiv:2307.06944 (2023), 150 pages.

\bibitem[vKM23]{vkma:computation}
R.~von K\"{a}nel and B.~Matschke, \emph{Solving {$S$}-unit, {M}ordell, {T}hue,
  {T}hue-{M}ahler and generalized {R}amanujan-{N}agell equations via the
  {S}himura-{T}aniyama conjecture}, Mem. Amer. Math. Soc. \textbf{286} (2023),
  no.~1419, vi+142.

\bibitem[Kim05]{kim:siegel}
M.~Kim, \emph{The motivic fundamental group of {$\mathbb
  P^1\setminus\{0,1,\infty\}$} and the theorem of {S}iegel}, Invent. Math.
  \textbf{161} (2005), no.~3, 629--656.

\bibitem[Kle73]{klein:cubicsurfaces}
F.~Klein, \emph{Ueber {F}l\"{a}chen dritter {O}rdnung}, Math. Ann. \textbf{6}
  (1873), no.~4, 551--581.

\bibitem[KM85]{kama:moduli}
N.~M. Katz and B.~Mazur, \emph{Arithmetic moduli of elliptic curves}, Annals of
  Mathematics Studies, vol. 108, Princeton University Press, Princeton, NJ,
  1985.

\bibitem[KM97]{kemo:coarse}
S.~Keel and S.~Mori, \emph{Quotients by groupoids}, Ann. of Math. (2)
  \textbf{145} (1997), no.~1, 193--213.

\bibitem[Kod67]{kodaira:construction}
K.~Kodaira, \emph{A certain type of irregular algebraic surfaces}, J. Analyse
  Math. \textbf{19} (1967), 207--215.

\bibitem[Kuh21]{kuhne:equidistribution}
L.~Kuhne, \emph{Equidistribution in families of abelian varieties and
  uniformity}, Preprint, arXiv:2101.10272 (2021), 48 pages.

\bibitem[Lan13]{lan:compactifications}
K.~Lan, \emph{Arithmetic compactifications of {PEL}-type {S}himura varieties},
  London {M}athematical {S}ociety {M}onographs, vol.~36, Princeton University
  Press, Princeton, 2013.

\bibitem[LV20]{lave:mordell}
B.~Lawrence and A.~Venkatesh, \emph{Diophantine problems and {$p$}-adic period
  mappings}, Invent. Math. \textbf{221} (2020), no.~3, 893--999. \MR{4132959}

\bibitem[Mah33]{mahler:approx1}
K.~Mahler, \emph{Zur {A}pproximation algebraischer {Z}ahlen. {I}}, Math. Ann.
  \textbf{107} (1933), no.~1, 691--730.

\bibitem[Man69]{manin:torsion}
Y.~Manin, \emph{The {$p$}-torsion of elliptic curves is uniformly bounded},
  Izv. Akad. Nauk SSSR Ser. Mat. \textbf{33} (1969), 459--465.

\bibitem[MB90]{moret-bailly:effmordell}
L.~Moret-Bailly, \emph{Hauteurs et classes de {C}hern sur les surfaces
  arithm\'etiques}, Ast\'erisque (1990), no.~183, 37--58, S{\'e}minaire sur les
  Pinceaux de Courbes Elliptiques (Paris, 1988).

\bibitem[Mum65]{mumford:picardmoduli}
D.~Mumford, \emph{Picard groups of moduli problems}, Arithmetical {A}lgebraic
  {G}eometry ({P}roc. {C}onf. {P}urdue {U}niv., 1963), Harper \& Row, New York,
  1965, pp.~33--81.

\bibitem[MW93a]{mawu:abelianisogenies}
D.~W. Masser and G.~W{\"u}stholz, \emph{Isogeny estimates for abelian
  varieties, and finiteness theorems}, Ann. of Math. (2) \textbf{137} (1993),
  no.~3, 459--472.

\bibitem[MW93b]{mawu:periods}
\bysame, \emph{Periods and minimal abelian subvarieties}, Ann. of Math. (2)
  \textbf{137} (1993), no.~2, 407--458.

\bibitem[NN81]{nano:polarizations}
M.~S. Narasimhan and M.~V. Nori, \emph{Polarisations on an abelian variety},
  Proc. Indian Acad. Sci. Math. Sci. \textbf{90} (1981), no.~2, 125--128.

\bibitem[Ols16]{olsson:stacks}
M.~Olsson, \emph{Algebraic spaces and stacks}, American Mathematical Society
  Colloquium Publications, vol.~62, American Mathematical Society, Providence,
  RI, 2016.

\bibitem[Par68]{parshin:construction}
A.~N. Par{\v{s}}in, \emph{Algebraic curves over function fields. {I}}, Izv.
  Akad. Nauk SSSR Ser. Mat. \textbf{32} (1968), 1191--1219.

\bibitem[PBT09]{poto:cubicsurfacesalgo}
I.~Polo-Blanco and J.~Top, \emph{A remark on parameterizing nonsingular cubic
  surfaces}, Comput. Aided Geom. Design \textbf{26} (2009), no.~8, 842--849.

\bibitem[PS14]{post:hyperelliptic}
B.~Poonen and M.~Stoll, \emph{Most odd degree hyperelliptic curves have only
  one rational point}, Ann. of Math. (2) \textbf{180} (2014), no.~3,
  1137--1166.

\bibitem[Rap78]{rapoport:hilbertmodular}
M.~Rapoport, \emph{Compactifications de l'espace de modules de
  {H}ilbert-{B}lumenthal}, Compositio Math. \textbf{36} (1978), no.~3,
  255--335.

\bibitem[R{\'e}m99]{remond:construction}
G.~R{\'e}mond, \emph{Hauteurs th\^eta et construction de {K}odaira}, J. Number
  Theory \textbf{78} (1999), no.~2, 287--311.

\bibitem[R{\'e}m10]{remond:rational}
\bysame, \emph{Nombre de points rationnels des courbes}, Proc. Lond. Math. Soc.
  (3) \textbf{101} (2010), no.~3, 759--794.

\bibitem[Rib90]{ribet:levellowering}
K.~A. Ribet, \emph{On modular representations of {${\rm Gal}(\overline{\bf
  Q}/{\bf Q})$} arising from modular forms}, Invent. Math. \textbf{100} (1990),
  no.~2, 431--476.

\bibitem[Sch72]{schmidt:subspace}
W.~M. Schmidt, \emph{Norm form equations}, Ann. of Math. (2) \textbf{96}
  (1972), 526--551.

\bibitem[Sch90]{schlickewei:sunitnf}
H.~Schlickewei, \emph{{$S$}-unit equations over number fields}, Invent. Math.
  \textbf{102} (1990), no.~1, 95--107.

\bibitem[Sko34]{skolem:method}
T.~Skolem, \emph{Ein {V}erfahren zur {B}ehandlung gewisser exponentialer
  {G}leichungen und diophantischer {G}leichungen}, Skand. Mat.-Kongr.
  \textbf{8} (1934), 163--188.

\bibitem[{Sta}]{sp}
The {Stacks Project Authors}, \emph{\textit{Stacks Project}},
  \url{https://stacks.math.columbia.edu}.

\bibitem[Szp85]{szpiro:faltings}
L.~Szpiro, \emph{La conjecture de {M}ordell (d'apr\`es {G}. {F}altings)},
  Ast\'erisque (1985), no.~121-122, 83--103, Seminar Bourbaki, Vol. 1983/84.

\bibitem[Tsu85]{tsuyumine:hilbertkodaira}
S.~Tsuyumine, \emph{On the {K}odaira dimensions of {H}ilbert modular
  varieties}, Invent. Math. \textbf{80} (1985), no.~2, 269--281.

\bibitem[TW95]{taywil:modular}
R.~Taylor and A.~Wiles, \emph{Ring-theoretic properties of certain {H}ecke
  algebras}, Ann. of Math. (2) \textbf{141} (1995), no.~3, 553--572.

\bibitem[Ull04]{ullmo:ratpoints}
E.~Ullmo, \emph{Points rationnels des vari\'{e}t\'{e}s de {S}himura}, Int.
  Math. Res. Not. (2004), no.~76, 4109--4125.

\bibitem[Ven18]{venkatesh:icm}
A.~Venkatesh, \emph{Cohomology of arithmetic groups---{F}ields {M}edal
  lecture}, Proceedings of the {I}nternational {C}ongress of
  {M}athematicians---{R}io de {J}aneiro 2018. {V}ol. {I}. {P}lenary lectures,
  World Sci. Publ., Hackensack, NJ, 2018, pp.~267--300.

\bibitem[Ven24]{venkatesh:heightconjecture}
\bysame, \emph{Heights of automorphic forms and motives}, Arithmetic geometry,
  Tata Inst. Fundam. Res. Stud. Math., vol.~41, Tata Inst. Fund. Res., Mumbai,
  [2024] \copyright 2024, pp.~435--468.

\bibitem[VV21]{vevi:effmordell}
F.~Veneziano and E.~Viada, \emph{Explicit height bounds for {$k$}-rational
  points on transverse curves in powers of elliptic curves}, Pacific J. Math.
  \textbf{315} (2021), no.~2, 477--503.

\bibitem[Wil95]{wiles:modular}
A.~Wiles, \emph{Modular elliptic curves and {F}ermat's last theorem}, Ann. of
  Math. (2) \textbf{141} (1995), no.~3, 443--551.

\bibitem[Yua21]{yuan:bogomolov}
X.~Yuan, \emph{Arithmetic bigness and a uniform bogomolov-type result},
  Preprint, arXiv:2108.05625 (2021), 125 pages.

\bibitem[Zaa05]{zaal:compactsubvarieties}
C.~Zaal, \emph{Complete subvarieties of moduli spaces of algebraic curves}, PhD
  thesis, {U}niversity of {A}msterdam (2005), 1--83, Available at:
  https://pure.uva.nl/ws/files/3915897/36235-Thesis.pdf.

\bibitem[Zha95]{zhang:positivejams}
S.~Zhang, \emph{Positive line bundles on arithmetic varieties}, J. Amer. Math.
  Soc. \textbf{8} (1995), no.~1, 187--221.

\bibitem[Zha01]{zhang:iccm}
\bysame, \emph{Geometry of algebraic points}, First {I}nternational {C}ongress
  of {C}hinese {M}athematicians ({B}eijing, 1998), AMS/IP Stud. Adv. Math.,
  vol.~20, Amer. Math. Soc., Providence, RI, 2001, pp.~185--198.

\end{thebibliography}
}

{\scriptsize

\vspace{0.4cm}

\noindent Shijie Fan, MCM Chinese Academy of Sciences, Beijing, E-mail address: {\sf  fanshijie@gmail.com}

\vspace{0.1cm}

\noindent Rafael von K\"anel, IAS Tsinghua University, Beijing, E-mail address: {\sf rafaelvonkanel@gmail.com}
}

\end{document}